\documentclass[12pt,reqno]{amsart}
\usepackage{amsmath,amsfonts,amssymb,amscd,amsthm,amsbsy, color}
\usepackage{graphicx}
\usepackage{comment}
\usepackage[normalem]{ulem}

\DeclareMathOperator{\arcsinh}{arcsinh}
\DeclareMathOperator{\arccot}{arccot}

\theoremstyle{plain} 
\newtheorem{thm}{Theorem}[section]
 
\newtheorem{lemma}[thm]{Lemma} 
\newtheorem{prop}[thm]{Proposition}
\newtheorem{defn}[thm]{Definition}
\newtheorem{defns}[thm]{Definitions}
\newtheorem{thmv}[thm]{Theorem}
\newtheorem{lemmav}[thm]{Lemma}

\newtheorem{rem}[thm]{Remark}
\newtheorem{rems}[thm]{Remarks}
\theoremstyle{definition}

\newtheorem{notaz}[thm]{Notation}
\numberwithin{equation}{section}

\newcommand\E{\mathcal{E}}
\newcommand\om{\omega^2}
\newcommand\R{\mathbb{R}}
\newcommand\Rpi{\mathbb{R}_{/2\pi\mathbb{Z}}}
\newcommand\Eh{\E+h}
\newcommand\ep{\epsilon}
\newcommand\gep{\gamma_{\epsilon}}
\newcommand{\x}[1]{\xi_{#1}}

\newcommand\xt{\tilde{\xi}}
\newcommand\F{\mathcal{F}}
\newcommand{\stkout}[1]{\ifmmode\text{\sout{\ensuremath{#1}}}\else\sout{#1}\fi}

\title{On some refraction billiards}
\author{Irene De Blasi and Susanna Terracini}

\address{Dipartimento di Matematica ``G. Peano''
	\newline\indent
	Universit\`a degli Studi di Torino
	\newline\indent
	Via Carlo Alberto 10, 10123 Torino, Italy\\}
\email{irene.deblasi@unito.it}
\email{susanna.terracini@unito.it}

\date{\today} 



\begin{document}

\begin{abstract}
The aim of this work is to continue the analysis, started in \cite{deblasiterraciniellissi}, of the dynamics of a point-mass particle $P$ moving in a galaxy with an harmonic biaxial core, in whose center sits a Keplerian attractive center (e.g. a Black Hole).  Accordingly, the plane $\mathbb{R}^2$ is divided into two complementary domains, depending on whether the gravitational effects of the galaxy's mass distribution or of the Black Hole prevail. Thus, solutions alternate arcs of Keplerian hyperbol\ae~ with harmonic ellipses; at the 
interface, the trajectory is refracted according to Snell's law. The model was introduced in \cite{Delis20152448}, in view of applications to astrodynamics.

In this paper we address the general issue of periodic and quasi-periodic orbits and associated caustics when the domain is a perturbation of the circle, taking advantage of  KAM and Aubry-Mather theories. 
\end{abstract}

	\maketitle


\section{Introduction and statement of the results}	\label{sec: intro}
 
 In this paper we deal with the dynamics of a point-mass particle $P$ moving in a galaxy with an harmonic biaxial core, in whose center  sits a Keplerian attractive center (e.g. a Black Hole).  In our setting, the plane $\mathbb{R}^2$ is divided into two complementary domains, depending on whether the gravitational effects of the galaxy's mass distribution or of the Black Hole prevail. We set the domain of influence of the the Black Hole's to be a generic regular domain $\boldsymbol{0}\in D\subset\mathbb{R}^2$, so that the particle moves on the plane governed by  the potential 
\begin{equation}\label{potential}
	V(z)=
	\begin{cases}
		V_I(z)=\mathcal{E}+h+\frac{\mu}{\vert z\vert} \quad &\text{if }z\in D\\
		V_E(z)=\mathcal{E}-\frac{\omega^2}{2}\vert z\vert^2 &\text{if } z\notin D, 
	\end{cases}
\end{equation}   
whith $\mathcal{E}+h,\mu, \omega>0$,  while the behaviour of the solution, when crossing the boundary $\partial D$, is ruled by a generalization of Snell's law (see  \S\ref{subs:snell}) of refraction. Thus, as sketched in Figure \ref{fig: archi},  trajectories alternate arcs of Keplerian hyperbol\ae~ with harmonic ellipses, with a refraction at the boundary $\partial D$. It is worthwhile noticing that, when also the domain $D$ is radially symmetrical the system is integrable, resulting in a shift of the polar angle only depending on the shooting one.  This paper deals with anisotropic perturbations of the circular domain, and possibly of the potentials.

\begin{figure}
	\centering
	\includegraphics[width=0.7\linewidth]{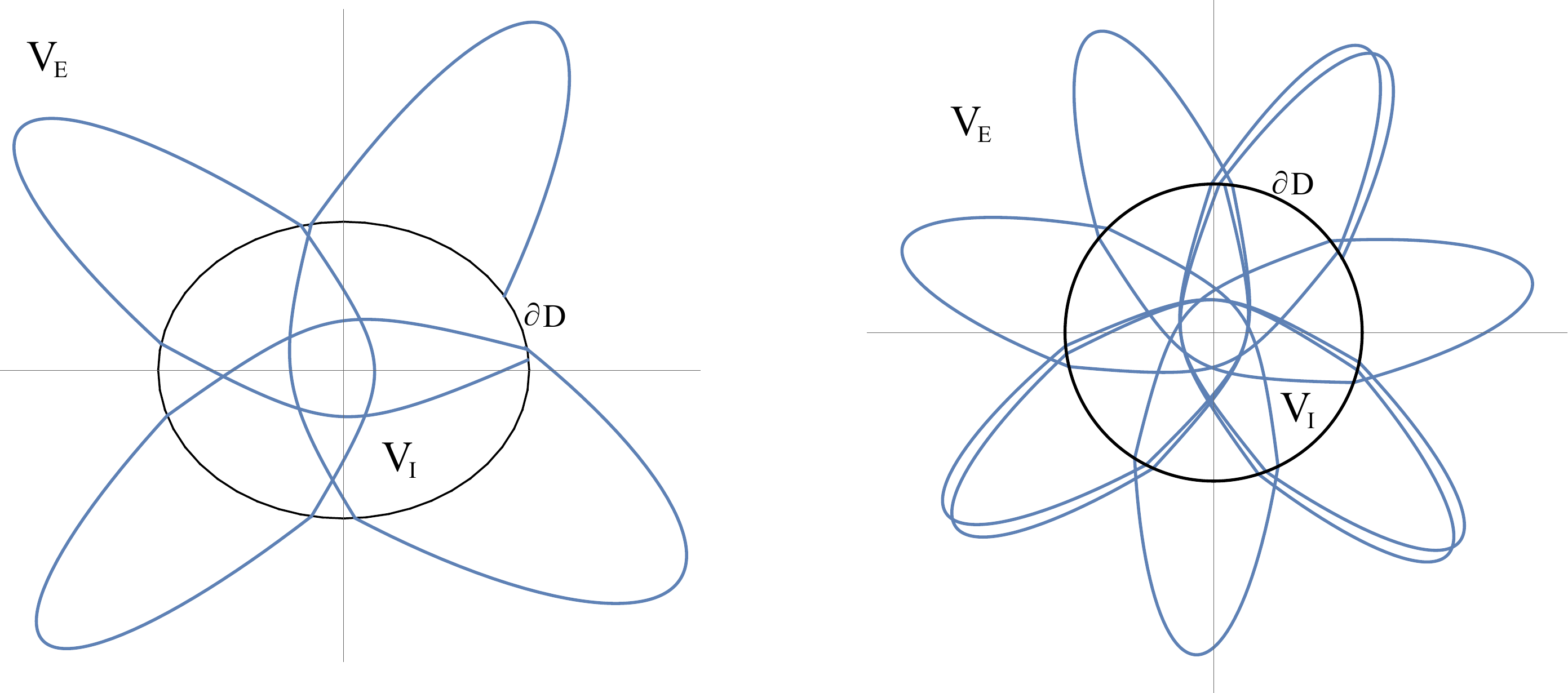}
	\caption{Examples of trajectories for $\E=2.5$, $\omega=1$, $h=2$ and $\mu=2$. Left: general trajectory for an elliptic domain with eccentricity $e=0.6$. Right: quasi-periodic trajectory for a circular domain.}
	\label{fig: archi}
\end{figure}

The model was introduced in \cite{Delis20152448}, for specific applications to astrodynamical problems; besides the motivations coming from Celestial Mechanics, as mentioned in  \cite{deblasiterraciniellissi}, our class of systems are of interests in view of possible applications in engineering artificial optical devices \cite{Krishnamoorthy2012205}.  Another model which presents strong analogies with the one described in this work is studied in \cite{cavallari2020sun}, where the complex dynamics deriving by patching orbits coming from pure Kepler's and Stark's problem (the so-called \textit{Sun-shadow dynamics}) is considered. \\
In \cite{deblasiterraciniellissi}, the authors addressed the problem of the stability of homotetic and brake solutions, depending on the values of the physical parameters $\mathcal{E},h,\mu, \omega$, the curvature of the boundary and its distance from the Keplerian center. The analysis carried there shows an interesting and varied landscape of bifurcations of homothetic and brake periodic orbits. \\
The aim of this work is to continue the analysis started in \cite{deblasiterraciniellissi}, addressing the general issue of periodic and quasi-periodic orbits and associated caustics when the domain is a perturbation of the circle, taking advantage of both  KAM and Aubry-Mather theories. \\
In order to tackle the problem, we shall consider the first\footnote{The phrase \emph{first return} may be confusing, since it is in fact a second crossing of the boundary; however, such apparent ambiguity will be clarified in \S\ref{sec: first return}. }  return map at the interface $\partial D$, after two consecutive (outer and inner) excursions, working at the zero energy level for the potential $V$ as above. Shooting from boundary points,  after performing a Levi-Civita regularization (cfr. \cite{Levi-Civita}) of Kepler potential, our system looks like a billiard, with  reflection at the boundary replaced by an outer excursion in between two refractions. We shall exploit the Lagrangian structure of the problem, building (locally) a generating function as the sum of an inner and outer contribution. Here comes a first problem, as these can not be globally defined. In addition, major difficulties arise from the return map not being globally defined,  from the singularity of the attraction center and from a lack of twist condition (it can be in shown that, in certain regimes of the parameters, even the completely intergrable circular case admits at least a twist change, see Remark \ref{oss cambio twist}). Here are our main results.

\begin{thm} [Circular domains] When $D$ is the unit circle, there are action-angle coordinates $(\xi, I)\in\Rpi\times(-I_c,I_c)$, where $\xi$ is the polar angle, such that we can express the first return map as a shift

\begin{equation}\label{mappa xi i 1}
	\begin{aligned}
	\mathcal{F}:&\text{ }\Rpi\times(-I_c,I_c)\to\Rpi\times(-I_c,I_c),
	\\&(\x0,I_0)\mapsto(\x1, I_1)=(\xi_0+\bar\theta(I_0), I_0), 
\end{aligned}
\end{equation}
where $\bar\theta(I)=f(I)+g(I)$ and
\begin{equation*}
	f(I)=
	\begin{cases}
		\arccot\left(\frac{\E-2I^2}{I\sqrt{4\E-2(2I^2+\om)}}\right)\quad &\text{if }I\in(0,I_c)\\
		0 &\text{if } I=0\\
		\arccot\left(\frac{\E-2I^2}{I\sqrt{4\E-2(2I^2+\om)}}\right)-\pi\quad &\text{if }I\in(-I_c,0)
	\end{cases}
\end{equation*}
and 

\begin{equation*}
	g(I)=
	\begin{cases}
		2\arccos\left(\frac{2I^2-\mu}{\sqrt{4(\Eh)I^2+\mu^2}}\right)-2\pi\quad &\text{if }I\in(0,I_c)\\
		0 &\text{if } I=0\\
		-2\arccos\left(\frac{2I^2-\mu}{\sqrt{4(\Eh)I^2+\mu^2}}\right)+2\pi\quad &\text{if }I\in(-I_c,0)
	\end{cases}
\end{equation*}
are real analytic functions in $(-I_c,I_c)$. 
	For every $I\in(-I_c, I_c)$, except for a finite number (at most ten) of points, there holds $\bar\theta'(I)\neq 0$.
\end{thm}
The critical value $I_c$ corresponds to the action associated with the \emph{total reflection} of the trajectory at the boundary, i.e. when the outgoing refracted trajectory becomes tangent to the boundary (see Remark \ref{oss riflessione totale}) and is given in \eqref{eq:Ic}. In the circular case, the \emph{rotation number} of the orbit is $\rho_I=\bar\theta(I)=f(I)+g(I)$. Depending whether this value is rational with $2\pi$ or not, the corresponding orbits are periodic or quasi-periodic. In both cases they determine an invariant curve and a pair of caustics, that is,  smooth closed curves such that every trajectory which starts tangent to remains tangent after every passage in and out the domain $D$. Let us point out that caustics come in pair, respectively  in $D$ and  in its complement. The proof of this Theorem is performed in \S\ref{sec: cerchio}. 

Next we consider a perturbation  $D_\ep$ of the domain whose boundary $\partial D_\ep=supp(\gep)$ is given by a radial deformation of the circle of the form 
\begin{equation}\label{def dominio perturbato 1}
\gep:\Rpi\to \R^2\quad \gep(\xi)=\left(1+\ep f(\xi; \ep)\right)e^{i\xi}, 
\end{equation}
where $f(\xi; \ep)$ is a suitably smooth function of $\Rpi\times [-C, C]$ and  $\xi$ is the polar angle. Then, the first return map on the perturbed boundary can be extended as $\mathcal F(\x, I; \ep)$ (see definition  \ref{pert func}), where $(\xi,I)$ are canonical variables defined in suitable neighbourhoods of such invariant curves. The following theorem resumes the results stated in Theorems \ref{teorema curve invarianti}, here considering a single invariant curve, and \ref{teorema caustiche pert}.   

\begin{thm}[Persistence of invariant curves (KAM)]\label{teorema curve invarianti intro}
	Let $f\in \mathcal C^k$, with $k>5$. Let us suppose that $\bar\theta'(I_0)\neq0$, and assume $\rho_0=\bar\theta(I_0)$ has a Diophantine ratio with $2\pi$ (see Definition \ref{defKAM}). Then there exists $\bar\ep_{\rho_0}$ such that for every $\ep\in\R$, $|\ep|<\bar\ep_{\rho_0}$ the map $\mathcal F(\x, I; \ep)$  admits a closed invariant curve of class $C^1$ with rotation numbers $\rho_0$.  Each of these invariant curves generates a pair of regular caustics. 
\end{thm}

Two invariant curves with Diophantine rotation numbers border an invariant region for the map $\mathcal F(\x, I; \ep)$, subject to the application of Poincar\'e-Birkhoff theorem and Aubry-Mather theory (see \cite{aubry1983discrete, mather58084existence, moser1962invariant}). As the map is area-preserving, we only need to verify the twist condition. This is a nontrivial issue, as the function $\bar\theta(I)$ may indeed change its monotonicity. This  fact poses some technical difficulties but also gives rise to a richer phenomenology. We have the following result.

\begin{thm}[Existence of Aubry-Mather invariant sets]\label{thm intro aubry}
Let $\bar\rho_-<\bar\rho_+$ be Diophantine rotation numbers, such that there are no critical values of the function $\bar\theta$ in the range $[\bar\rho_-,\bar\rho_+]$. Then there exists $\bar\ep>0$ such that for every $\ep\in\R$ with $|\ep|<\bar\ep$ and for every $\rho\in[\bar\rho_-, \bar\rho_+]$ the map $\mathcal F(\x, I_; \ep)$ admits at least one orbit with rotation number $\rho$. When  $\rho=2\pi\frac{m}{n}$ then for $\ep$ sufficiently small there are at least $2$ $(m,n)$-orbits, namely, such that, denoted with $\left\{(\xi_k, I_k)\right\}_{k\in\mathbb N}=\left\{\mathcal F(\xi_0, I_0)\right\}_{k\in\mathbb N}$ the orbit generated by the initial point $(\xi_0, I_0)$, one has 
\begin{equation*}
	 \forall k\in\mathbb N\quad (\xi_{k+n}, I_{k+n})=(\xi_k+2\pi m, I_k)\equiv_{2\pi}(\xi_k, I_k).
\end{equation*}
\end{thm}

This statement can be easily deduced from Theorem \ref{thm:korbits};  as we shall see there,  however, the actual number of solutions can be larger, depending on the number of monotonicity changes of the twist $\bar\theta(I)$.

To conclude this preamble, let us remark that, although the computations are not explicitely performed here, with means of the same analytical tools and techniques other variations of the considered model, more similar to the one described in \cite{Delis20152448}, can be investigated. \\
As an example, let us consider $\epsilon\in\R$ and a non-isotropic perturbation of the outer potential given by 
\begin{equation*}
	\tilde V_E(z)=\E-\frac{\omega^2}{2}x^2-\frac{(\omega+\epsilon)^2}{2}y^2, 
\end{equation*} 
where $z=(x,y)\in\R^2$. Taking, if necessary, $h\in\R$ instead of $h>0$, one has that the dynamics induced by the potential  
\begin{equation}\label{potential pert}
	\tilde V(z)=
	\begin{cases}
		V_I(z) \quad &\text{if }z\in D\\
		\tilde V_E(z) &\text{if } z\notin D 
	\end{cases}
\end{equation}
 with the usual Snell's refraction rule on the boundary is conservative if $\partial D$ is the set of points in $\R^2$ such that the potentials differ by the same constant $A\in\R$, that is, 
 \begin{equation*}
 	\partial D=\{z\in\R^2\text{ }|\text{ }V_I(z)=\tilde V_E(z)+A\}; 
 \end{equation*}
 note that, if $A=0$, the potential is continous and the refraction reduces to a conservation of both the position and velocity on $\partial D$, leading to a $C^1$ junction between inner and outer arcs. For $\ep\in\R$ small, this system be considered again as a small perturbation of the circular case. 

\subsection{Analogies and differences with Birkhoff billiards}

In some sense, the model investigated in this work falls into the category of billiards and it is interesting to look for analogies. There are, however, fundamental differences: first of all the rays are curved by the gravitational force inside $D$;  moreover, reflection at the boundary is replaced by an excursion in the outer region in between two refractions.  Similarly to billiards, our model can be described by an  area preserving map of the cylinder, but, as we shall show, the twist condition may be violated even in the simplest case of a circular domain. It should be noted that refraction imposes a new constraint, because the interface can only be crossed outwards when the inner arc is transverse enough to the boundary (cfr Remark \ref{oss buona definizione F}).\\
There is a wide literature on Birkhoff billiards, with recent relevant advances (see the book \cite{MR2168892} and papers \cite{MR3868417,MR3815464,MR3788206,MR4085879}), including some cases of composite billiard  with reflections and refractions \cite{Baryakhtar2006292}, also in the case of a periodic inhomogenous lattice \cite{Glendinning2016}. Special mention should be paid to the work on magnetic billiards, where the trajectories of a charged particle in this setting are straight lines concatenated with circular arcs of a given Larmor radius \cite{GasiorekThesis,gasiorek2019dynamics}. Let us add that, compared with the cases quoted above, additional difficulties arise because the corresponding return map is not globally well defined and from the singularity of the Kepler potential.\\
As a concrete example on how the considered model presents intrinsic analogies with classical billiards, as well as important differences, let us consider the case of an elliptic domain. In classical elliptic billiards of every eccentricity (see for example \S4 in \cite{MR2168892}), the straigh orbit segment corresponding to the major axis determines always a saddle point for the associated billiard map, while the one coinciding with the minor axis is a center. A similar behaviour can be observed in our refractive model when the domain is an ellipse with small eccetricity, but with a fundamental difference: while this model admits the same equilibrium orbits of the classical billiard, their stabilities depend on the value of the physical parameter $\E, h, \omega, \mu$. This bifurcation phenomenon is derived in \cite{deblasiterraciniellissi}, and can be stated as follows: 
\begin{prop}[\cite{deblasiterraciniellissi}, Corollary 1.3]\label{eccpiccola}
	Let us consider an elliptic domain, with horizontal unitary major semiaxis and eccentricity $e$, and let $\bar{z}_0$ and $\bar{z}_{\pi/2}$ denote the horizontal and vertical homotetic periodic solutions of collision-reflection type. 
	If $\frac{\sqrt{\E+h+\mu}}{\mu}<\frac{\sqrt{2\E-\om}}{2\sqrt{2}\E}$, then, for small eccentricities, 
	$\bar{z}_0$ is stable and $\bar{z}_{\pi/2}$ is unstable.  Symmetrically, if $\frac{\sqrt{\E+h+\mu}}{\mu}>\frac{\sqrt{2\E-\om}}{2\sqrt{2}\E}$,  then, for small eccentricities, 
	$\bar{z}_0$ is unstable and $\bar{z}_{\pi/2}$ is stable.  
\end{prop}
Figure \ref{fig: poincare} shows the orbits of $\mathcal F$ in the two cases of Proposition \ref{eccpiccola}; one can notice that, far from the homotetic solutions, the system shows evidences of integrability. 
\begin{figure}
	\centering
	\includegraphics[width=0.7\linewidth]{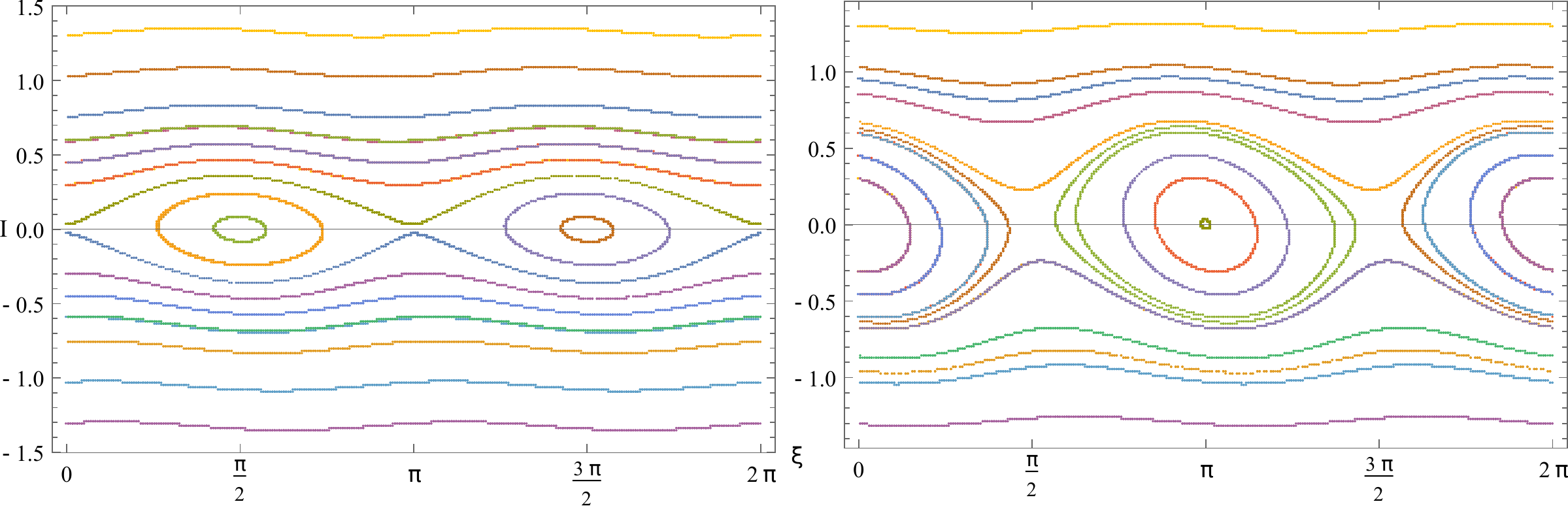}
	\caption{Orbits of the first return map $F$ for $\E=10$, $\omega=1$, $h=3$, $\mu=44$ (left) and $\mu=55$ (right). The shift in the stability of the homotetic equilibrium is evident, as well as the presence of invariant curves far from the $\xi$-axis in both cases. }
	\label{fig: poincare}
\end{figure}

As a final remark, let us highlight one more connection between classical Birkhoff billiards and our refraction model, related to the methods used to study the respective dynamics: in this work, the fundamental argument to prove Theorem \ref{thm intro aubry} is the extension to small perturbations of the existence results obtained for a circular domain, which is completely integrable. Analogous reasonings are used for example in \cite{avila2016integrable} and \cite{MR3815464} to prove a local version of Birkhoff conjecture starting from small deformations respectively of a circular and an elliptic domain, both sharing the complete integrability property in the classic case.   

\section{Preliminaries and notation}\label{sec: preliminaries}
Let us consider a general fixed-ends differential problem of the type 
\begin{equation}\label{problema preliminari}
	\begin{cases}
		z''(s)=\nabla V(z(s))\quad &s\in[0, T]\\
		\frac{1}{2}|z'(s)|-V(z(s))=0 &s\in[0, T]\\
		z(0)=z_0, \text{ }z(T)=z_1
	\end{cases}
\end{equation}
where $z_0, z_1\in\R^2$ or some subset and $V(z)$ is regular and greater than zero almost everywhere. The solutions of problems of type (\ref{problema preliminari}) can be studied by
taking advantage of the variational structure, which also incorporates the  junction rule between the inner and outer geodesics. \\
This Section aims to provide the basic definitions and properties of the main analytical tools needed to study the dynamics described in \S 1, along with a derivation of the generalized Snell's law which connects the geodesic arcs. The latter, as well as a more detailed description of the results presented, can be found in Appendix A of \cite{deblasiterraciniellissi}.

\subsection{Maupertuis functional and Jacobi length}

\begin{defn}
	Given $z_0, z_1\in \R^2$, we denote with $M([0,1],z(t))$ the Maupertuis functional
	
	\begin{equation*}
		M(z)=M([0,1],z(t))=\int_0^1|\dot{z}(t)|^2V(z(t))dt
	\end{equation*}
	which is defined on the set
	\begin{equation*}
		H_{z_0,z_1}=\{z(t)\in H([0,1],\R^2)\text{ }|\text{ }z(0)=z_0, z(1)=z_1\}.  
	\end{equation*}
	Furthermore, the Jacobi length is given by 
	\begin{equation*}
		L(z)=L([0,1],z(t))=  \int_0^1|\dot{z}(t)|\sqrt{V(z(t))}dt, 
	\end{equation*}
	and is defined on the closure of $H_0=\{z(t)\in H_{z_0,z_1}\text{ }| \text{ }|\dot{z}(t)|>0, V((z(t)))>0 \text{ }\forall t\in[0,1] \}$ in the weak topology of $H^1([0,1],\R^2)$. 
	
\end{defn}

It can be shown (see for example \cite{deblasiterraciniellissi, soave2014symbolic}) that the critical points at positive levels of $M(z)$ are reparametrizations of classical solutions of (\ref{problema preliminari}); moreover, the H\"older inequality, along with the energy conservations law, allow to prove that, if $z(t)$ is a critical point of $M$, then $L^2(z)=2M(z)$: the search for solutions of (\ref{problema preliminari}) is hence equivalent to the study of the critical points of the Jacobi length, taking into account the different time scales of the problems. 

\begin{defn}
	We define: 
	\begin{itemize}
		\item the \textbf{geodesic time} $t\in[0,1]$ as the time parameter to be used in $M(z(t))$ and $L(z(t))$; 
		\item the \textbf{cinetic time} $s\in [0,T]$  as the physical time parameter through which $z(s)=z(t(s))$ solves (\ref{problema preliminari}).  
	\end{itemize}
	We denote with $\dot{ }=d/dt$ and $'=d/ds$ respectively the derivatives with respect to the geodesic and the cinetic time.
\end{defn}  
The relation between the geodesic and the cinetic time is proved in \cite{deblasiterraciniellissi} and is given by 
\begin{equation}\label{times}
	\frac{d}{dt}=\frac{L}{\sqrt{2}V(z(t(s)))}\frac{d}{ds}, 
\end{equation}
where $L=|\dot{z}(t)|\sqrt{V(z(t))}$ is constant along the solutions of (\ref{problema preliminari}). 
\subsection{Generalized Snell's law}\label{subs:snell}

Besides the analysis of the inner and outer dynamics taken separately, the study of the dynamical system described in \S1 is heavily influenced by the choice of a suitable junction law between two consecutive arcs through the interface $\partial D$. In our case, a local minimization argument leads to the determination of a connection rule which turns out to be a generalization for curved interfaces and non-Euclidean metrics of the classical Snell's refraction law. \\
In general, let us consider an inner potential $V_I$ and an outer one $V_E$, along with their associated distances $d_E$ and $d_I$, obtained by taking the metrics $\left(g_{E\backslash I}\right)_{ij}(z)=V_{E\backslash I}(z)\delta_{ij}$. Fixed $z_I^0\in D$ and $z_E^0\notin D$ in a suitable neighborood of $\partial D$, one can search for the point $\bar z$ on the interface such that the sum of the distances $d_E(z_E^0, z)+d_I(z, z_I^0)$ is minimal. Denoted with $T(z)$ the unit tangent vector to $\partial D$ in $z$, this extremality condition turns out to be equivalent to require
\begin{equation}\label{snell 1}
	\sqrt{V_E(\bar z)}\frac{\dot{z}_E(1)}{|\dot{z}_E(1)|}\cdot T(\bar z)=\sqrt{V_I(\bar z)}\frac{\dot{z}_I(0)}{|\dot{z}_I(0)|}\cdot T(\bar z), 
\end{equation}   
where $z_E(t)$ and $z_I(t)$ are respectively the geodesic arcs connecting $\bar z$ with $z_E^0$ and $z_I^0$ parametrised by the geodesic time $t\in[0,1]$. Equivalently, taking into account the relation between the geodesic and the cinetic time stated in(\ref{times}), Equation \eqref{snell 1} can be expressed as
\begin{equation*}
\frac{1}{\sqrt{2}}z_E'(T_E)\cdot T(z)=\frac{1}{\sqrt{2}}z_I'(0)\cdot T(z),  
\end{equation*}
where $T_E=s(1)$ is the cinetic ime value associated to $t=1$. 
Geometrically, equation (\ref{snell 1}) represents the conservation of the tangent component of the velocity vector through the interface, and, if $\alpha_E$ and $\alpha_I$ are respectively the angles of $\dot{z}_E(1)$ and $\dot{z}_I(0)$ with the normal vector to $\partial D$ in $\bar z$, it can be rephrased as 
\begin{equation}\label{snell 2}
	\sqrt{V_E(\bar z)}\sin\alpha_E=\sqrt{V_I(\bar z)}\sin\alpha_I; 
\end{equation}
clearly, the transition from the inside of the domain $D$ can be treated in analogous way obtaining again (\ref{snell 2}).
\begin{rem}\label{oss riflessione totale}
	Since for every $z\in\R^2\backslash\{0\}$ $V_E(z)<V_I(z)$, one has that the equation 
	\begin{equation*}
		\alpha_I=\arcsin\left(\sqrt{\frac{V_E(z)}{V_I(z)}}\sin\alpha_E\right)
	\end{equation*}
is always solvable, while 
\begin{equation*}
\alpha_E=\arcsin\left(\sqrt{\frac{V_I(z)}{V_E(z)}}\sin\alpha_I\right)
\end{equation*}
admits a solution if and only if $|\sqrt{V_I(z)/V_E(z)}\sin\alpha_I|\leq1$, that is, iff 
\begin{equation}\label{alphacrit}
	|\alpha_I|\leq\alpha_{I, crit}=\arcsin\left(\sqrt{\frac{V_E(z)}{V_I(z)}}\right).
\end{equation}
Whenever condition (\ref{alphacrit}) is satisfied, the passage from the inner to the outer arc is possible. Note that condition (\ref{alphacrit}) depend globally on the physical parameters $\E, h, \mu, \omega$ and, locally, by $z\in\partial D$. 
\end{rem}

\section{First return map}\label{sec: first return}

The dynamics described in Section \ref{sec: intro} will be studied by considering the corresponding discrete dynamical system which keeps track of the crossings of the total orbit, obtained by the juxtaposition of outer and inner arcs, and the domain's boundary $\partial D$. Such a map  is named \textit{first return map} and is the analogous of the billiard map in the classical theory of Birkhoff billiards, see \cite{MR1328336}: this Section is devoted to its construction, leaving to \S \ref{sec: cerchio} and \ref{sec: perturbazioni} the analysis of its good definition for particular domains. \\
\begin{notaz}\label{notazione}
	From now on we will denote the differential problems connected to the outer and inner dynamics with the equivalent notations
	\begin{equation}\label{problema generale}
		\begin{cases}
			z''(s)=\nabla V_{E\backslash I}(z(s)) \quad &s\in[0,T_{E\backslash I}]\\
			\frac{1}{2}|z'(s)|^2-V_{E\backslash I}(z(s))=0 &s\in[0,T_{E\backslash I}]\\
			\text{boundary conditions}
		\end{cases}
	\Leftrightarrow
	\begin{cases}
		(HS_{E\backslash I})[z(s)]\quad &s\in[0,T_{E\backslash I}]\\
		\text{boundary conditions}
	\end{cases}
	\end{equation}
where $(HS_{E\backslash I})[z]$ denote the differential problem $z''=\nabla V_{E\backslash I}(z(s))$,  along with the corresponding energy conservation law. According to the scope of problem (\ref{problema generale}), the boundary conditions will prescribe either the initial position and velocity or the endpoints of $z(s)$. 
\end{notaz}
Let us suppose that the boundary of the regular domain $D$ defined in Section \ref{sec: intro} can be parametrized by a regular closed curve $\gamma:\Rpi\rightarrow\mathbb{R}^2$. Given some initial conditions $p_0^{(I)}, v_0^{(I)}, p_0^{(E)}$ and $v_0^{(E)}$, let us consider the solutions  $z_I(s)$ and $z_E(s)$ of the two systems 
\begin{equation}\label{EI problem}
	\begin{cases}
		(HS_I)[z(s)] &s\in[0,T_I]\\
		z_I(0)=p_0^{(I)}, z'_I(0)=v_0^{(I)}
	\end{cases}
\quad 
	\begin{cases}
		(HS_E)[z(s)] &s\in[0,T_E]\\
		z_E(0)=p_0^{(E)}, z'_E(0)=v_0^{(E)}
	\end{cases}
\end{equation}
for some $T_I$, $T_E>0$. \\
Fixed $p_0\in\partial D$, $v_0\in\mathbb{R}^2$ such that it points towards the exterior of $D$, we want to describe (supposing that it exists) the orbit obtained by the juxtaposition of an exterior orbit $z_E$ and the subsequent inner orbit $z_I$, namely, the function  $z_{EI}(s): \left[0,T_E+T_I\right]\to\R^2$ defined by
\begin{equation}\label{zEI def}
	z_{EI}(s)=
	\begin{cases}
		z_E(s)\quad &s\in[0,T_E)\\
		z_I(s) &s\in[T_E,T_E+T_I)\\
		z_E^{(1)}(s) &s=T_E+T_I,
	\end{cases}
\end{equation}
where, defined 
\begin{equation*}
	\begin{split}
		&\tilde p_1=z_E(T_E), \quad \tilde p_2=z_I(T_E+T_I), \\
		&u_{1}=\frac{z'_E(T_E)}{|z'_E(T_E)|}, \quad u'_{1}=\frac{z_I'(T_E)}{|z_I'(T_E)|}, \quad u_2=\frac{z'_I(T_E+T_I)}{|z'_I(T_E+T_I)|}, \quad u_2'=\frac{{z'}^{(1)}_E(T_E+T_I)}{|{z'}^{(1)}_E(T_E+T_I)|}, 
	\end{split}
\end{equation*}
the arcs $z_E, z_I, z_E^{(1)}$ are solutions of either the outer or inner problem and the the initial conditions of each branch satisfy the Snell's law. More precisely,  
\begin{equation}\label{extprob}
	\begin{cases}
		(HS_E)[z_E(s)] &s\in[0,T_E]\\
		z_E(s)\notin D, \tilde p_1\in\partial D &s\in(0,T_E)\\
		z_E(0)=p_0, z'_E(0)=v_0 
	\end{cases}
\end{equation}
\begin{equation}\label{intprob}
	\begin{cases}
		(HS_I)[z_I(s)] &s\in[T_E,T_E+T_I]\\
		z_I(s)\in D, \tilde p_2\in\partial D &s\in(T_E,T_E+T_I)\\
		 z_I(T_E)=\tilde p_1, \sqrt{V_E(\tilde p_1)}u_1\cdot e_1= 
		\sqrt{V_I(\tilde p_1)}u_1'\cdot e_1
	\end{cases}
\end{equation}
\begin{equation}
	\begin{cases}
		(HS_E)[z_E^{(1)}(s)] &s\in[T_E+T_I,T_E+T_I+\tilde{T}]\\
		z_E^{(1)}(s)\notin D &s\in(T_E+T_I,T_E+T_I+\tilde{T}]\\
		z_E^{(1)}(T_E+T_I)=\tilde p_2, \sqrt{V_E(\tilde p_2)}u_2'\cdot e_2= 
		\sqrt{V_I(\tilde p_2)}u_2\cdot e_2,  
	\end{cases}
\end{equation}
for some $T_E,T_I,\tilde{T}>0$ and where $e_1$ and $e_2$ are the unit vectors tangent to $\partial D$ respectively in $\tilde p_1$ and $\tilde p_2$. \\
Once defined the concatenation of arcs $z_{EI}(s)$, the first return map is given by the iteration map which expresses $(p_1,v_1)=(z_{EI}(T_E+T_I),z'_{EI}(T_E+T_I))$ as a function of $(p_0,v_0)$ in a suitable set of coordinates. \\
Recalling that we set $\partial D=supp(\gamma)$, with $\gamma$ regular closed curve in $\Rpi$, let us suppose that $p_0=z_{EI}(0)=\gamma(\xi_0)\in\partial D$ is the starting point of the first outer branch of $z_{EI}(s)$: then, denoting with $t(\xi_0)$ and $n(\xi_0)$ respectively the tangent and the outward-pointing normal unit vectors to $\gamma$ in $p_0$, the initial velocity $v_0$ can be expressed as $v_0=\sqrt{2V_E(\gamma(\xi_0))}(\cos{\alpha_0}~ n(\xi_0)+\sin{\alpha_0}~ t(\xi_0))$, where $\alpha_0\in[{-\pi/2,\pi/2}]$ is the angle between $v_0$ and $n(\xi_0)$, positive if $v\cdot t(\xi_0)\geq0$ and negative otherwise. Then, once $\xi_0$ is fixed, the vector $v_0$ is completely determined by $\alpha_0$. We can then consider the map 
\begin{equation}
	F:\Rpi\times[-\pi/2,\pi/2]\rightarrow\Rpi\times[-\pi/2, \pi/2], \quad (\xi_0,\alpha_0)\mapsto(\xi_1,\alpha_1)=(\xi_1(\xi_0,\alpha_0),\alpha_1(\xi_0,\alpha_0)),
\end{equation} 
where the pair $(\xi_1,\alpha_1)$ completely determines $(z_{EI}(T_E+T_I),z'_{EI}(T_E+T_I))$. 
\begin{rem}\label{oss buona definizione F}
	In general, one can not guarantee that $F$ is well defined, as the existence of $\x1$ an $\alpha_1$ as described depends strongly on the interaction between the outer and inner arcs with the domain's boundary. The conditions  for $F$ to be well defined are essentially two: 
	\begin{itemize}
		\item[(i)] the existence and uniqueness of the outer and inner arcs for fixed endpoints on $\partial D$;
		\item[(ii)] the good definition of the refraction from the inside to the outside of $D$. As a matter of fact, from Remark \ref{oss riflessione totale} one has that, while the refraction exterior-interior is always possible, an inner arc can be refracted in an outer one if and only if,  denoted with $\beta_1$ the angle between $z_I'(T_E+T_I)$ and the inward-pointing normal unit vector of $\gamma$ in $\gamma(\x1)$, $|\beta_1|<\beta_{crit}=\arcsin\left(\sqrt{V_E(\gamma(\x1))/V_I(\gamma(\x1))}\right)$. 
	\end{itemize}
In other words, condition $(ii)$ is verified if the inner Keplerian arcs $z_I$ are transverse enough to the domain's boundary $\partial D$. This obstruction to the good definition of $F$ may be cicumvented by considering a suitable prolonging according to which, whenever $|\beta_1|>\beta_{crit}$, the test particle returns back in the interior of the domain with an angle $\beta=\beta_1$. This extension, corresponding to the so-called \textit{total reflection}, is somehow suggested by physical intuition and by analogy with the classical Birkhoff billiards, as well as the traditional Snell's law for light rays. However, as this approach leads to technical difficulties due to the passage to the tangent case for $|\beta_1|=\beta_{crit}$ and the consequent loss in regularity, in this work such extension is not considered, and, by definition, the map $F$ will be well defined only when  $|\beta_1|<\beta_{crit}$ for every inner arc. \\ 
 As we will see in Sections \ref{sec: esistenza archi cerchio} and \ref{sec: esistenza archi perturbati}, this is true if $D$ is a disk or a perturbation of the latter, provided $\alpha_0\neq\pm\pi/2$.
\end{rem}
%

\subsection{Variational approach}\label{sec: variational general}

Choosing the right set of canonical variables, the first return map can be expressed in a variational form, which allows to take advantage of a wide range of powerful theoretical instruments, coming from KAM and Aubry-Mather theories, which allow to tackle the case of small perturbations of a circular domain $D_0$ (the latter is discussed in \S \ref{sec: cerchio},  while the perturbational approach is the subject of \S \ref{sec: perturbazioni}). Following e.g. \cite{gole2001symplectic}, we can consider the \textit{generating function}
\begin{equation}
	S: \Rpi\times\Rpi\to \mathbb{R}, \quad S(\x0, \x1)=S_E(\x0, \tilde{\xi})+S_I(\tilde{\xi}, \x1)=d_E(\gamma(\x0), \gamma(\tilde{\xi}))+d_I(\gamma(\tilde{\xi}), \gamma(\x1)), 
\end{equation}
where $d_E$ and $d_I$ are defined as in \S \ref{sec: preliminaries} and, according to Snell's law, the intermediate point $\tilde p=\gamma(\tilde \xi)$ is such that $\tilde{\xi}$ is a critical point for the function $f(\x0, \xi, \x1)=d_E(\gamma(\x0),\gamma(\xi))+d_I(\gamma(\xi), \gamma(\x1))$ with $\x0$ and $\x1$ fixed. In other words, $\tilde{\xi}$ is a solution of 
\begin{equation}\label{snell equation}
	\partial_bS_E(\xi_0, \tilde{\xi})+\partial_aS_I(\tilde{\xi}, \x1)=0, 
\end{equation}
where $\partial_a$ and $\partial_b$ denote respectively the partial derivatives with respect to the first and second variable.
In general, one can not guarantee   the global good definition of the generating function $S(\x0, \x1)$, which depends strongly on the specific geometry of the domain and on the values of the physical parameters $\E, h, \mu, \omega$. On the other hand, the local defiinition of $S(\x0, \x1)$ needs a nondegeneration condition: under the assumption of existence and uniqueness of the inner and outer arcs, which ensures the differentiability of $d_E(p_0,p_1)$ and $d_I(p_0, p_1)$ separately, let us consider $(\x0,\xt, \x1)\in\left(\Rpi\right)^3$ such that (\ref{snell equation}) is satisfied. If 
\begin{equation}\label{nondeg S}
	\partial_{\xt}(\partial_bS_E(\x0, \xt)+\partial_aS_I(\xt, \x1))=\partial^2_bS_E(\x0, \xt)+\partial^2_aS_I(\xt, \x1)\neq 0, 
\end{equation}
then locally around $\x0, \x1$ one can express $\xt=\xt(\x0, \x1)$ as a function of the endpoints. Moreover, one has 
\begin{equation}\label{dxitilde}
	\begin{split}
	\partial_{\x0}\xt(\x0, \x1)=-\frac{\partial_{ab}S_E(\x0, \xt(\x0, \x1))}{\partial^2_bS_E(\x0, \xt(\x0, \x1))+\partial^2_aS_I(\xt(\x0, \x1), \x1)},\\
		\partial_{\x1}\xt(\x0, \x1)=-\frac{\partial_{ab}S_I(\xt(\x0, \x1), \x1)}{\partial^2_bS_E(\x0, \xt(\x0, \x1))+\partial^2_aS_I(\xt(\x0, \x1), \x1)}. 
	\end{split} 
\end{equation}
If (\ref{nondeg S}) holds, the generating function is well defined locally around $\x0$ and $\x1$, an one can define the \textit{canonical actions} associated to the system and to the above coordinates by the relations 
\begin{equation}\label{def I}
I_0=-\partial_{\x0}S(\x0, \x1), \quad I_1=\partial_{\x1}S(\x0, \x1); 
\end{equation} 
note that, when $S(\x0, \x1)$ is well defined, the same is true also for the actions as functions of the angles $\x0$ and $\x1$. \\
In order to define a first return map in the new canonical action-angle variables, one needs to express $\x1$ and $I_1$ as functions of $\x0$ and $I_0$:  this is possible if a second nondegeneracy condition holds. Let us consider $\x0,\x1, \xt(\x0, \x1)$ such that (\ref{nondeg S}) holds, and define $I_0$ as in (\ref{def I}): if   
\begin{equation}\label{nondeg I0 1}
	\partial_{\x1}(I_0+\partial_{\x0}S(\x0, \x1))\neq0,
\end{equation}
 one can find $\x1=\x1(\x0, I_0)$ as a function of the initial action-angle variables. In particular, making use of (\ref{snell equation}) and (\ref{dxitilde}), condition (\ref{nondeg I0 1}) translates in 
\begin{equation}\label{nondeg I0 2}
	\frac{\partial_{ab}S_E(\x0, \xt(\x0, \x1))\partial_{ab}S_I(\xt(\x0, \x1), \x1)}{\partial_b^2S_E(\x0, \xt(\x0, \x1))+\partial^2_aS_I(\xt(\x0, \x1), \x1)}\neq0, 
\end{equation}
which is well defined in view of (\ref{nondeg S}). If (\ref{nondeg I0 2}) holds, one can then find two neighborhoods $[\x0-\lambda_{\x0}, \x0+\lambda_{\x0}]$ and  $[I_0-\lambda_{I_0}, I_0+\lambda_{I_0}]$ such that the new \textit{local first return map}
\begin{equation}\label{first return azione-angolo}
	\begin{aligned}
	\mathcal\F:[\x0-\lambda_{\x0}, \x0+&\lambda_{\x0}]\times[I_0-\lambda_{I_0}, I_0+\lambda_{I_0}]\to \Rpi\to\R, \\ &(\x0, I_0)\mapsto(\x1(\x0, I_0), I_1(\x0, I_0)),
	\end{aligned} 
\end{equation}
where $I_1(\x0, I_0)=\partial_{\x1}S(\x0, \x1(\x0, I_0))$, is well defined. \\
The switch from a \textit{Lagrangian} approach adopted by using the original first return map $F$ and an \textit{Hamiltonian} one involving the generating function and the canonical action-angle variables is crucial as it will allow, in Section \ref{sec: perturbazioni},  to take advantage of the results coming from KAM and Aubry-Mather theories to derive the existence of orbits with prescribed rotation numbers in the case of small perturbations of a circula domain (see \S \ref{sec: perturbazioni}). \\
In the circular case, which will be analyzed in details in \S \ref{sec: cerchio}, both the potentials and the domain are centrally symmetric: a consequence of the subsequent invariance under rotations is that the nondegeneracy conditions (\ref{nondeg S}) and (\ref{nondeg I0 2}) will result in fact equivalent, and $S$,  denoted in this case with $S_0$, will be well defined almost everywhere. \\
 As a final remark, one can observe that (\ref{snell 2}), along with (\ref{def I}) and (\ref{snell equation}), provides a general relation between the actions $I_0, I_1$ and the angles $\alpha_0, \alpha_1$  defined above: in particular, 
 \begin{equation}\label{I alpha}
 	\begin{aligned}
 		I_0(\x0, \x1)&=-\partial_{\x0}(S_E(\x0, \xt)+S_I(\xt, \x1)))=-\partial_aS_E(\x0, \xt)=\\
 		&=\frac{1}{\sqrt{2}}z'_E(0)\cdot\frac{\dot\gamma(\x0)}{|\dot\gamma(\x0)|}=\sqrt{V_E(\gamma(\x0))}\sin\alpha_0\\I_1(\x0, \x1)&=\partial_{\x1}(S_E(\x0, \xt)+S_I(\xt, \x1)))=\partial_bS_I(\xt, \x1)=\\
 		&=\frac{1}{\sqrt{2}}z'_I(T_E+T_I)\cdot\frac{\dot\gamma(\x1)}{|\dot\gamma(\x1)|}=\sqrt{V_I(\gamma(\x0))}\sin\alpha'_1=\sqrt{V_E(\gamma(\x1))}\sin\alpha_1.
 	\end{aligned}
 \end{equation}
Eqs.(\ref{I alpha}) provides a natural boundary for the values that the actions can assume, which, in a inhomogeneous case, depend on $\x0$ and $\x1$: in particular
\begin{equation*}
	|I_0|\leq \sqrt{V_E(\gamma(\x0))}\text{ and }|I_1|\leq\sqrt{V_E(\gamma(\x1))}, 
\end{equation*}
where the equalities correspond to the tangent case $\alpha_{0\backslash 1}=\pm \pi/2$. \\
In the case of a circular domain, the bound on the actions is uniform and given by 
\begin{equation}\label{eq:Ic}
	I_0, I_1\in\left[-\sqrt{\E-\frac{\omega^2}{2}},\sqrt{\E-\frac{\omega^2}{2}} ~\right]=\left[-I_c, I_c\right]. 
\end{equation}
For reasons which will be clear in \S 4, related to the good definition of the map $\mathcal F$, we will not consider the tangent case, restricting ourselves to the open interval $(-I_c, I_c)$.  


\section{The unperurbed case: circular domain}\label{sec: cerchio}
Let us suppose that $D$ is a disk of radius $1$, and denote it with $D_0$: in this particular case, both the potentials and the domain are centrally symmetric, and, as a consequence, the system is integrable. In particular, it is possible to find the explicit expression of the first return map in action-angle variables, which in this case is denoted with $\mathcal F_0$, as it will be done in Section \ref{sec: mappa cerchio esplicita}. \\

\subsection{Good definition of $\boldsymbol{d_E}$ and $\boldsymbol{d_I}$.}\label{sec: esistenza archi cerchio}
As already pointed out in Remark \ref{oss buona definizione F}, the preliminary condition for the generating function $S(\x0, \x1)$ to be well defined is that the inner and outer distances $d_E(p_0, p_1)$ and $d_I(p_0, p_1)$ are differentiable with respect to their variables. This is true when the inner and outer arcs connecting two points on $\partial D$ are (locally) unique, as Theorems \ref{thm ex globale esterna unp} and \ref{thm ex globale interna unp} yield for the circular domain. The proofs of these theorems are given, along with some preliminary results leading to them,  in Appendix \ref{appA}.

\begin{thmv}\label{thm ex globale esterna unp}
	If $\E>\om$, for every $p_0, p_1\in \partial D_0$ such that $|p_0-p_1|<2$, there exists $T>0$ and a unique $z(s;p_0,p_1):[0,T]\to\mathbb{R}^2$ solution of the fixed ends problem
	\begin{equation}\label{problema esterno testo}
		\begin{cases}
			(HS_E)[z(s)] &s\in[0,T]\\
			|z(s)|>1 &s\in(0,T)\\
			z(0)=p_0, \quad z(T)=p_1. 
		\end{cases}
	\end{equation} 
	Moreover, $z(s;p_0,p_1)$ is of class $C^1$ with respect to variations of the endpoints. 
\end{thmv}

The analogous of Theorem \ref{thm ex globale esterna unp} for the inner dynamics states the existence and uniqueness of the Keplerian arc connecting two points with an additional topological constraint. To define it, it is necessary to recall and adapt the classical concept of winding number to open arcs which connects points on the circle; to do that, we use a definition which recalls the one given in \cite{soave2014symbolic}. 
In general, let $p_0,p_1\in \partial D_0$ two non-antipodal points on the circle, that is, $|p_0-p_1|<2$, and a regular, simple curve $\alpha:[a,b]\to\R^2$ such that $\alpha(a)=p_0$ and $\alpha(b)=p_1$ and $\alpha(s)\neq\boldsymbol{0}$ for every $s\in[a,b]$. Let us then consider a closed curve $\Gamma_{\alpha}$ which is equal to $\alpha(s)$ for $s\in[a,b]$ and then follows the shortest arc of $\partial D_0$ which connects $p_1$ to $p_0$.  We can then define the winding number associated to the curve $\alpha$ with respect to $0\in\R^2$ as
\begin{equation}\label{wind unp}
	I(\alpha([a,b]),0)=\frac{1}{2\pi i}\int_{\Gamma_{\alpha}}\frac{dz}{z}.
\end{equation}

\begin{thmv}\label{thm ex globale interna unp}
	For every $p_0,p_1\in\partial D_0$, $|p_0-p_1|<2$, there is a unique $T>0$ and a unique solution $z(s;p_0,p_1)$ of  	
	\begin{equation}\label{PB2}
		\begin{cases}
			(HS_I)[z(s)] &s\in[0,T]\\
			|z(s)|<1 &s\in(0,T)\\
			z(0)=p_0, \text{ }z(T)=p_1
		\end{cases}
	\end{equation}
	such that $z(s;p_0,p_1)$ is of class $C^1$ with respect to $p_0$ and $p_1$ and: 
	\begin{itemize}
		\item if $p_0=p_1$, $z(s;p_0,p_0)$ is an ejection-collision solution; 
		\item if $p_0\neq p_1$, $z(s;p_0,p_1)$ is a classical solution of (\ref{PB2}) such that $|Ind(z([0,T]),0)|=1$. If $p_1\to p_0$, $z(s;p_0,p_1)$ tends to the ejection-collision solution $z(s;p_0,p_0)$. 
	\end{itemize} 
	Moreover, there is $0<C<1$ such that, for every $p_0,p_1$ as above, \begin{equation}\label{transv testo}
		-p_0\cdot\frac{z_1'(0;p_0, p_1)}{|z_1'(0;p_0, p_1)|}>C\quad\text{ and }\quad 
		p_1\cdot\frac{z_1'(T_1;p_0, p_1)}{|z_1'(T_1;p_0, p_1)|}>C. 
	\end{equation}
The quantity $C$ depends on the physical parameters $\Eh, \mu$ and on $|p_0-p_1|$. In particular, it tends to $1$ when $\E\to \infty$ or $|p_0-p_1|\to 0$
\end{thmv}

The estimates given by (\ref{transv testo}) are crucial  to ensure that the inner arcs are transversal  to $\partial D_0$ as much as needed: this is necessary for the first return map to be well defined, since,  according to Remark  \ref{oss buona definizione F}, it is clear that these arcs can not be tangent to the interface.

\subsection{Study of the map $\mathcal F_0$}\label{sec: studio F0}

Once the good definition and differentiability of the distances $d_I(p_0, p_1)$ and $d_E(p_0,p_1)$ is ensured, one can eventually consider the generating function introduced in Section \ref{sec: variational general} and given by
\begin{equation*}
	S_0(\x0, \x1)=S_{E, 0}(\x0, \xt)+S_{I,0}(\xt, \x1)=d_E(\gamma_0(\x0), \gamma_0(\xt))+d_I(\gamma_0(\xt), \gamma_0(\x1)), 
\end{equation*} 
where $\gamma_0:\Rpi\to\R^2$ denotes the circle of radius $1$, and investigate the associated nondegeneracy conditions (\ref{nondeg S}) and (\ref{nondeg I0 2}). Although the complete analysis on its good definition will be done after the derivation of the explicit formulation of the  associated first return map $\mathcal F_0$ in Section \ref{sec: mappa cerchio esplicita}, the central symmetry of the circular case allows to give some preliminary informations.  As both the outer and inner systems are invariant under rotations, the associated generating functions can be expressed as univariate functions depending on the angle spanned by the arc; more precisely, 
\begin{equation*}
S_0(\x0,\x1)=\tilde S_0(\x1-\x0)=\tilde S_{E, 0}(\xt-\x0)+\tilde S_{I, 0}(\x1-\xt). 
\end{equation*}
Going through the same analysis described in general in Section \ref{sec: variational general}, the intermediate coordinate $\xt$ is implicitely determined  as a function of $\x0$ and $\x1$ by the relation 
\begin{equation*}
	\partial_{\xt}( \tilde S_{E, 0}(\xt-\x0)+\tilde S_{I, 0}(\xi-\xt))=0
\end{equation*} 
if (\ref{nondeg S}) is verified. In the circular case, the latter translates in 
\begin{equation}\label{nondeg cond S cerchio}
	\tilde S_{E, 0}''(\xt-\x0)+\tilde S_{I, 0}''(\x1-\xt)\neq0. 
\end{equation}
If (\ref{nondeg cond S cerchio}) holds, one has 
\begin{equation*}
	\partial_{\x0}\xt=\frac{\tilde S''_{E, 0}(\xt-\x0)}{\tilde S_{E, 0}''(\xt-\x0)+\tilde S_{I, 0}''(\x1-\xt)}, \quad \partial_{\x1}\xt=\frac{\tilde S''_{I, 0}(\x1-\xt)}{\tilde S_{E, 0}''(\xt-\x0)+\tilde S_{I, 0}''(\x1-\xt)},
\end{equation*}
and the canonical actions are defined by 
\begin{equation}\label{azioni cerchio}
	\begin{aligned}
	I_0=-\partial_{\x0}\tilde S_0(\x1-\x0)=\tilde S'_{E,0}(\xt(\x0, \x1)-\x0), \\ I_1=\partial_{\x1}\tilde S_0(\x1-\x0)=\tilde S'_{I, 0}(\x1-\xt(\x0,\x1)). 
	\end{aligned}
\end{equation}
The first equation in (\ref{azioni cerchio}) defines implicitely $\x1=\x1(\x0, I_0)$, and, as a consequence, the map $\mathcal F_0$, if (\ref{nondeg I0 2}) holds, that is, if 
\begin{equation}\label{nondeg I cerchio}
	\partial_{\x1}(I_0-\tilde S'_{E, 0}(\xt(\x0,\x1)-\x0))=-\frac{\tilde S''_{E, 0}(\xt(\x0,\x1)-\x0)\tilde S''_{I, 0}(\x1-\xt(\x0,\x1))}{	\tilde S_{E, 0}''(\xt(\x0, \x1)-\x0)+\tilde S_{I, 0}''(\x1-\xt(\x0, \x1))}\neq 0. 
\end{equation}
As a final remark, note that the validity of conditions (\ref{nondeg cond S cerchio}) and (\ref{nondeg I cerchio}) are strongly related to the \textit{twist condition} (see \cite{jurgen1986recent, gole2001symplectic, aubry1983discrete}) associated to the map $\mathcal F_0$, defined as $\partial_{I_0}\x1\neq0$. As a matter of fact, one has 
\begin{equation*}
	\partial_{I_0}\x1=\frac{\tilde S_{E, 0}''(\xt(\x1-\x0)-\x0)+\tilde S_{I, 0}''(\x1-\xt(\x1-\x0))}{\tilde S''_{E, 0}(\xt(\x0,\x1)-\x0)\tilde S''_{I, 0}(\x1-\xt(\x1-\x0))}: 
\end{equation*}
one can then say that, if $\mathcal F_0$ is given, the twist condition is equivalent to require the nondegenerations (\ref{nondeg cond S cerchio}) and (\ref{nondeg I cerchio}) to be true. \\

\subsection{Explicit formulation of $\mathcal{F}_0$}\label{sec: mappa cerchio esplicita}

When the domain $D$ is circular, the first return map $F:(\xi_0,\alpha_0)\mapsto(\xi_1,\alpha_1)$ can be explicitely determined: in this case, the nondegeneracy given through (\ref{nondeg cond S cerchio}) and (\ref{nondeg I cerchio}) can be investigated in the equivalent form given by the twist condition. The boundary $\partial D_0$ can be parametrized as $\gamma(\xi)=(\cos{\xi},\sin{\xi})$, with $\xi\in\Rpi$, and the symmetry properties of the potentials $V_E$ and $V_I$ and the isotropy of the Snell's law on a circular domain imply that $F$ is of the form 
\begin{equation}\label{mappa cerchio}
	F(\xi_0,\alpha_0)=(\xi_1(\xi_0,\alpha_0),\alpha_1(\xi_0,\alpha_0))=(\xi_0+\bar{\theta}(\alpha_0),\alpha_0); 
\end{equation} 
in other words, the first return map on the circle reduces to a conservation of the velocity variable $\alpha$ and a shift in the angle $\xi$ of a suitable quantity $\bar{\theta}$ which depends only on the physical parameters of the problem and on $\alpha_0$. The Jacobian matrix $DF((\xi_0,\alpha_0))$ can be then expressed for every pair $(\xi_0,\alpha_0)\in\Rpi\times(-\pi/2,\pi/2)$ as 
\begin{equation*}
	DF(\xi_0,\alpha_0)=
	\begin{pmatrix}
		1&\frac{\partial\bar{\theta}}{\partial\alpha_0}(\alpha_0)\\0&1
	\end{pmatrix}.
\end{equation*}
From the above considerations, we have that $\bar{\theta}(\alpha_0)=\bar{\theta}_E(\alpha_0)+\bar{\theta}_I(\alpha_0)$, where $\bar{\theta}_E(\alpha_0)$ and $\bar{\theta}_I(\alpha_0)$ represent the excursions in the angles due respectively to the outer and the inner arcs of the orbit $z_{EI}(s)$.

\paragraph{\textbf{Outer shift}}

The outer shift has been already computed as an additional result in the proof of Theorem \ref{thm ex globale esterna unp} in Appendix \ref{appA}, and is equal to 
\begin{equation}\label{thetaE}
	\theta_E(\alpha)=
	\begin{cases}
		\theta_E^+(\alpha)=\text{arccot}\left(\frac{\om}{(2\E-\om)\sin{(2\alpha)}}+\cot{(2\alpha)}\right) &\text{ if }\alpha>0, \\
		0 &\text{ if }\alpha=0, \\
		\theta_E^-(\alpha)=\text{arccot}\left(\frac{\om}{(2\E-\om)\sin{(2\alpha)}}+\cot{(2\alpha)}\right)-	\pi &\text{ if }\alpha<0.  
	\end{cases}
\end{equation}

\paragraph{\textbf{Inner shift}}

As the system is invariant under rotations, without loss of generality let $p_0=(1,0)$ to be the initial point of the inner orbit and, denoted by $v_0$ its initial velocity, let $\beta_0\in(-\pi/2,\pi/2)$ the angle between $v_0$ and the inward-pointing radial unit vector, namely, $-p_0$; then $v_0=\sqrt{2(\Eh)+2\mu}(-\cos\beta_0,\sin{\beta_0})=(v_x,v_y)$. The inner Cauchy problem is then given by 
\begin{equation}\label{inner}
	\begin{cases}
		z''(s)=-\frac{\mu}{|z(s)|^3}z(s), \quad s\in[0,T]\\
		z(0)=p_0, z'(0)=v_0. 
	\end{cases}
\end{equation} 
Unlike the outer case, for the inner Keplerian orbit it is not possible to decouple the $2-$dimensional system into two one-dimensional systems in the variables $(x,y)$; we shall rely on other classical techniques (see \cite{celletti2010stability}) which require the passage in polar coordinates: consider the functions $r(s)\in\R^+$ and $\theta(s)\in\Rpi$ such that $z(s)=r(s)e^{i\theta(s)}$. From the conservation of the angular momentum, we have that 
\begin{equation}\label{k}
	r(s)^2\theta'(s)=const=k=|p_0||v_0|\sin{\beta_0}=\sqrt{2\E+2h+2\mu}\sin{\beta_0}, 
\end{equation}
while the energy conservation law implies 
\begin{equation}\label{rdot}
	\Eh=\frac{1}{2}\left(r'(s)^2+r(s)^2\theta'(s)^2\right)-\frac{\mu}{r(s)}\Rightarrow r'(s)=-\sqrt{2(\Eh)-\frac{k^2}{r^2(s)}+\frac{2\mu}{r(s)}},  
\end{equation}
where the sign depend by the fact that, according to the chosen initial conditions, $r(s)$ is decreasing. 
Taking together (\ref{k}) and (\ref{rdot}), one has then 
\begin{equation}\label{dtheta}
	d\theta=-\frac{k}{r^2\sqrt{2(\Eh)-\frac{k^2}{r^2}+\frac{2\mu}{r}}}dr.
\end{equation}
The classical results for the two-body problem ensure that, for positive energies, the Kepler problem is unbounded, and $r(t)$ reaches its unique minimum $r_p$ at a time $s_p>0$. The value of $r_p$ is given by (see \cite{celletti2010stability})
\begin{equation}\label{rp}
	r_p=\frac{k^2}{\mu}\left(1+\sqrt{1+\frac{2(\Eh)k^2}{\mu^2}}\right)^{-1}.
\end{equation}
If we denote with $\theta_p$ the polar angle of the pericenter and consider the initial conditions given by (\ref{inner}), taking into account the simmetry of $r(s)$ with respect to $s_p$, we have that the inner shift angle is given by 
\begin{equation}\label{thetaI}
	\bar{\theta}_I=2\theta_p,
\end{equation}
where $\theta_p$ can be obtained by integration from (\ref{dtheta}): 
\begin{equation*}
	\theta_p=\int_0^{\theta_p}d\theta=-k\int_1^{r_p}\frac{dr}{r^2\sqrt{2(\Eh)-\frac{k^2}{r^2}+\frac{2\mu}{r}}}=\frac{k}{|k|}\int_1^{\frac{1}{r_p}}\frac{du}{\sqrt{\frac{2(\Eh)}{k^2}-u^2+\frac{2\mu}{k^2}u}}
\end{equation*}
setting $x=u-\mu/k^2$, $x_0=1-\mu/k^2$ and $x_1=\frac{\mu}{k^2}\sqrt{1+\frac{2\left(\Eh\right)k^2}{\mu^2}}$: 
\begin{equation*}
	\theta_p=\frac{k}{|k|}\int_{x_0}^{x_1}\frac{dx}{\sqrt{\frac{2(\Eh)}{k^2}+\frac{\mu^2}{k^4}-x^2}}.
\end{equation*}
Finally, defining $y=x\left(\frac{2(\Eh)}{k^2}+\frac{\mu^2}{k^4}\right)^{-1/2}$, $y_0=(k^2-\mu)\left(2(\Eh)k^2+\mu^2\right)^{-1/2}$ and $y_1=1$:
\begin{equation}\label{thetapf}
	\theta_p=\frac{k}{|k|}\int_{y_0}^{y_1}\frac{dy}{\sqrt{1-y^2}}=\frac{k}{|k|}\arccos{y_0}=\frac{k}{|k|}\arccos\left(\frac{k^2-\mu}{\sqrt{2(\Eh)k^2+\mu^2}}\right). 
\end{equation} 
Casting together (\ref{k}), (\ref{thetaI}) and  (\ref{thetapf}), one obtains
\begin{equation}\label{thetaIf}
	\bar{\theta}_{I}(\beta_0)=
	\begin{cases}
		\bar{\theta}^+_I(\beta_0)=2\arccos\left(\frac{(2\E+2h+2\mu)\sin{\beta_0}^2-\mu}{\sqrt{4(\Eh)(\Eh+\mu)\sin{\beta_0}^2+\mu^2}}\right)-2\pi  &\text{ if } \beta_0\geq 0\\
		\bar{\theta}^-_I(\beta_0)=-2\arccos\left(\frac{(2\E+2h+2\mu)\sin{\beta_0}^2-\mu}{\sqrt{4(\Eh)(\Eh+\mu)\sin{\beta_0}^2+\mu^2}}\right)+2\pi  &\text{ if } \beta_0< 0, 
	\end{cases}
\end{equation}
where the shift is such that $\bar{\theta}_I(\beta_0)\in(-\pi, \pi)$ for every $\beta_0\in(-\pi/2, \pi/2)$.  Note that 
\begin{equation*}
	\begin{aligned}
		&\lim_{\beta_0\to0^+}\bar{\theta}_I^+(\beta_0)=0=\lim_{\beta_0\to0^-}\bar{\theta}_I^-(\beta_0),\\
		\lim_{\beta_0\to0^+}\frac{d}{d\beta_0}&\bar{\theta}_I^+(\beta_0)=-\frac{4(\Eh+\mu)}{\mu}=\lim_{\beta_0\to0^-}\frac{d}{d\beta_0}\bar{\theta}_I^-(\beta_0)
	\end{aligned}
\end{equation*}
and then $\bar{\theta}_I\in C^1(-\pi/2,\pi/2)$. 

\paragraph{\textbf{Total shift and properties of the overall trajectories}}

The total shift angle $\bar{\theta}(\alpha_0)$ is computed by taking the sum of the outer and the inner shifts and taking into account the transition laws for the velocities across the interface $\partial D_0$. In particular, if $\tilde\alpha$ and $\tilde\beta$ denote respectively the angles wit the normal unit vector of the outer and the inner velocities of an orbit crossing the interface in a point $\tilde p\in\partial D_0$, from (\ref{snell 2}) one has
\begin{equation*}
	\sqrt{\E-\frac{\om}{2}|\tilde p|^2}\sin{\tilde\alpha}=\sqrt{\Eh+\frac{\mu}{|\tilde p|}}\sin{\tilde\beta}; 
\end{equation*}
in the particular case of a circular domain, the Snell's law is uniform over all the points of $\partial D_0$, and the initial and final angles with the radial direction are equal for every brach of the orbit. Performing in (\ref{thetaIf}) the substitution $\sin{\beta_0}=\sqrt{(2\E-\om)/(2(\Eh+\mu))}\sin{\alpha_0}$, one obtains the total shift 
\begin{equation*}
	\bar{\theta}(\alpha_0)=
	\begin{cases}
		\bar{\theta}_E^+(\alpha_0)+\bar{\theta}_I^+(\alpha_0) &\text{ if }\alpha_0>0\\
		0 &\text{ if }\alpha_0=0\\
		\bar{\theta}_E^-(\alpha_0)+\bar{\theta}_I^-{(\alpha_0)} &\text{ if }\alpha_0<0,
	\end{cases}
\end{equation*}
where $\bar{\theta}_E^+(\alpha_0)$ and $\bar{\theta}_E^-(\alpha_0)$ are given by (\ref{thetaE}) and 
\begin{equation*}
	\bar{\theta}^+_I(\alpha_0)=2\arccos\left(\frac{(2\E-\om)\sin{\alpha_0}^2-\mu}{\sqrt{2(\Eh)(2\E-\om)\sin{\alpha_0}^2+\mu^2}}\right)=-\bar{\theta}_I^-(\alpha_0). 
\end{equation*}
The map is continuous and differentiable with respect to $\alpha_0$, and 
\begin{equation*}
	\begin{aligned}
		\frac{d}{d\alpha_0}\bar{\theta}(0)=&\lim_{\alpha_0\to0^+}\frac{d}{d\alpha_0}\bar{\theta}(\alpha_0)=\lim_{\alpha_0\to0^-}\frac{d}{d\alpha_0}\bar{\theta}(\alpha_0)=\\&=\frac{2\E-\om}{\E}-\frac{2\sqrt{2}\sqrt{\Eh+\mu}\sqrt{2\E-\om}}{\mu}. 
	\end{aligned}
\end{equation*}
Passing to the canonical coordinates $(\xi, I)$, the axisymmetry of the potentials and the isotropy of Snell's law on the circle translates in the conservation of the quantity $I$ both in the endpoints and the transition point $\tilde{\xi}$. The first claim is a straightforward consequence of Eq.(\ref{mappa cerchio}), while to prove the conservation of the action across the intermediate point one needs to consider the actions $I^E$ and $I^I$ associated to $S_{E, 0}$ and $S_{I, 0}$ separately:   
	\begin{equation}\label{az cerchio}
		\begin{aligned}
			&I^E_1(\x0,\x1)=\partial_bS_{E, 0}(\x0,\tilde{\xi})=\sqrt{V_E(\gamma(\tilde{\xi}))}\sin{\beta_0}=\sqrt{V_I(\gamma(\tilde{\xi}))}\sin{\beta_1}\\ 
			\quad &I^I_0(\x0,\x1)=-\partial_aS_{I, 0}(\tilde{\xi},\x1)=\sqrt{V_I(\gamma(\tilde{\xi}))}\sin{\beta_1}.
		\end{aligned} 
	\end{equation}
	Since on the circle $\alpha_0=\beta_0$ and $\beta_1=\alpha_1'$, we have that for every $\x0,\x1\in\Rpi$
	\begin{equation}\label{cons azioni}
		I_0(\x0,\x1)=I^E_1(\x0,\x1)=I^I_0(\x0,\x1)=I_1(\x0,\x1)\equiv I(\x0,\x1).
	\end{equation}
Moreover, from (\ref{I alpha}) one has that in the circular case the global domain of definition of the actions does not depend on the points $\x0, \x1$, that is 
\begin{equation*}
	I_0, I_1\in\left(-\sqrt{\E-\frac{\om}{2}},\sqrt{\E-\frac{\om}{2}} \right)=(-I_c, I_c)=\mathcal I. 
\end{equation*} 
Taking into account Eq.(\ref{mappa cerchio}), the definitions of $\bar{\theta}_E$ and $\bar{\theta}_I$ and the relations (\ref{I alpha}), (\ref{cons azioni}), in the new set of canonical coordinates $(\xi, I)\in\Rpi\times\mathcal I$ we can express the first return map as 
\begin{equation}\label{mappa xi i}
	\begin{aligned}
	&\mathcal{F}:\Rpi\times\mathcal I\to\Rpi\times\mathcal I,
	\\(\x0,I_0)\mapsto&(\x1, I_1)=(\xi_0+\bar\theta(I_0), I_0)=(\xi_0+f(I_0)+g(I_0), I_0), 
\end{aligned}
\end{equation}
where 
\begin{equation*}
	f(I)=
	\begin{cases}
		\arccot\left(\frac{\E-2I^2}{I\sqrt{4\E-2(2I^2+\om)}}\right)\quad &\text{if }I\in(0,I_c)\\
		0 &\text{if } I=0\\
		\arccot\left(\frac{\E-2I^2}{I\sqrt{4\E-2(2I^2+\om)}}\right)-\pi\quad &\text{if }I\in(-I_c,0)
	\end{cases}
\end{equation*}
and 

\begin{equation*}
	g(I)=
	\begin{cases}
		2\arccos\left(\frac{2I^2-\mu}{\sqrt{4(\Eh)I^2+\mu^2}}\right)-2\pi\quad &\text{if }I\in(0,I_c)\\
		0 &\text{if } I=0\\
		-2\arccos\left(\frac{2I^2-\mu}{\sqrt{4(\Eh)I^2+\mu^2}}\right)+2\pi\quad &\text{if }I\in(-I_c,0)
	\end{cases}
\end{equation*}
are $C^1$ functions in $\mathcal I$. 
\begin{rem}
	Direct computations show that for every $\E>\om, h>0, \mu>0$ and for every $I\in\mathcal I$ one has $f'(I)>0$ and $g'(I)<0$: the outer and inner shifts are then invertible in $\mathcal I$, and one can define the inverse functions $\tilde{f}(\theta)=f^{-1}(I)_{|I=I(\theta)}$ and $\tilde{g}(\theta)=g^{-1}(I)_{|I=I(\theta)}$. From the regularity of both $f$ and $g$, we have that $\tilde f$ and $\tilde g$ are of class $C^1$ in the respective domains. In particular,
	\begin{equation*}
		f\left(\mathcal I\right)=(-\pi, \pi), \quad g\left(\mathcal I\right)=(-\bar{\theta}, \bar{\theta}), 
	\end{equation*}
\begin{equation}\label{bartheta}
 \bar{\theta}=2\pi-2\arccos\left(\frac{2\E-\om-\mu}{\sqrt{2(\Eh)(2\E-\om)+\mu^2}}\right).
 \end{equation} 
Moreover, 
	\begin{equation}\label{ftilde gtilde}
		\begin{aligned}
			\tilde{\xi}=\xi_0+f(I)\Leftrightarrow I=\tilde{f}(\tilde{\xi}-\x0)\equiv I^E_1(\x0,\tilde{\xi}), \\
			\x1=\tilde{\xi}+g(I)\Leftrightarrow I=\tilde{g}(\x1-\tilde{\xi})\equiv I^I_0(\tilde{\xi}, \x1). 
		\end{aligned}
	\end{equation}
\end{rem}
\begin{lemma}
	For every $I\in\mathcal I$, except for a finite number if points, $f'(I)+g'(I)\neq 0$.
\end{lemma}
\begin{proof}
	Direct computations lead to 
	\begin{equation}\label{derivata f+g}
		\begin{aligned}
			f'(I)+g'(I)&=\frac{\sqrt{2}(2\E^2-(\E+2I^2)\om)}{\sqrt{2\E-\om-2I^2}(\E^2-2\om I^2)}-\frac{8(\Eh)I^2+4(\Eh)\mu+4\mu^2}{\sqrt{\Eh+\mu-I^2}(4(\Eh)I^2+\mu^2)}=\\
			&=\frac{A(I^2)}{B(I^2)}-\frac{C(I^2)}{D(I^2)}. 
		\end{aligned}
	\end{equation}
	Since for every $I\in\mathcal I$ we have that $A(I^2),B(I^2),C(I^2),D(I^2)>0,$  
	\begin{equation*}
		f'(I)+g'(I)=0\Leftrightarrow X=I^2\in\left[0,I_c^2\right) \text{ is a solution of }p(x)=0, 
	\end{equation*}
	where $p(x)=A^2(x)D^2(x)-B^2(x)C^2(x)$. 
	As $p(x)$ is a real polynomial of degree $5$ in $X$, one can have at most ten values of $I\in\mathcal I$ such that $f'(I)+g'(I)=0$. 
\end{proof}
We define $\bar{\mathcal{I}}=\{I\in\mathcal I\text{ }|\text{ }f'(I)+g'(I)=0\}$ as the set of the critical points of the function $f+g$. 
\begin{prop}\label{buona definizione gen unp}
	The generating function $S_0(\x0,\x1)$ is well defined in $\Rpi\times\Rpi$ except for a finite number of pairs $(\x0,\x1)$ in the quotient space $(\Rpi\times\Rpi)/\sim$, where $(\x0,\x1)\sim(\xi_0', \x1') \Leftrightarrow \x1-\x0=\x1'-\x0'$. 
\end{prop}
\begin{proof}
	For $S_0(\x0,\x1)=S_{E, 0}(\x0,\tilde{\xi})+S_{I,0}(\tilde{\xi}, \x1)$ to be well defined, one needs to verify condition (\ref{nondeg S}). From the definition of the actions
	\begin{equation*}
		\begin{aligned}
			\partial_{\tilde{\xi}}(\partial_bS_{E, 0}(\x0, \tilde{\xi})+\partial_aS_{I, 0}(\tilde{\xi},\x1))&=\partial_{\tilde{\xi}}(I_1^E(\x0, \tilde{\xi})-I_0^I(\tilde{\xi}, \x1))=\partial_{\tilde{\xi}}I_1^E(\x0, \tilde{\xi})-\partial_{\tilde{\xi}}I_0^I(\tilde{\xi}, \x1)=\\
			&=\partial_{\tilde{\xi}}\tilde{f}(\tilde{\xi}-\x0)-\partial_{\tilde{\xi}}\tilde{g}(\x1-\tilde{\xi})=\left(\frac{1}{f'(I)}+\frac{1}{g'(I)}\right)_{I=I(\x0,\x1)}=\\
			&=\left(\frac{f'(I)+g'(I)}{f'(I)g'(I)}\right)_{I=I(\x0,\x1)}  
		\end{aligned}
	\end{equation*}
	which is zero if and only if $\x1-\x0\in(f+g)(\bar{\mathcal{I}})$. 
\end{proof}

\begin{prop}\label{prop mappa unp}
	 For every $(\x0, I_0)\in\Rpi\times\left(\mathcal I\backslash\bar{\mathcal{I}}\right)$ the first return map  $\mathcal F_0$
	\begin{enumerate}
		\item is conservative; 
		\item satisfies the \textit{twist condition}
		\begin{equation*}
			\frac{\partial \xi_1}{\partial I_0}(\x0, I_0)\neq 0. 
		\end{equation*} 
	\end{enumerate}
\end{prop}
\begin{proof}
	The conservativity of $\mathcal F_0$ is a direct consequence of the variational formulation of the problem: when $\tilde{\xi}$ is well defined, we have (expressing $\x1=\x1(\x0,I_0)$)
	\begin{equation*}
		\begin{aligned}
			&\partial_{\x0}\x1=-\frac{\partial_a^2S_0(\x0,\x1)}{\partial_{ab}S_0(\x0,\x1)}, \quad &&\partial_{I_0}\x1=-\frac{1}{\partial_{ab}S_0(\x0,\x1)}\\
			&\partial_{\x0}I_1=\partial_{ab}S_0(\x0,\x1)-\frac{\partial_b^2S_0(\x0,\x1)\partial_a^2S_0(\x0,\x1)}{\partial_{ab}S_0(\x0,\x1)},\quad &&\partial_{I_0}I_1=-\frac{\partial_b^2S_0(\x0,\x1)}{\partial_{ab}S_0(\x0,\x1)}, 
		\end{aligned}
	\end{equation*}	
	where, from (\ref{ftilde gtilde}),  
	\begin{equation*}
		\partial_{ab}S(\x0,\x1)=\partial_{\x0\x1}S(\x0,\x1)=\partial_{ab} S_I(\xt,\x1)\partial_{\x0}\tilde{\xi}=-\frac{\tilde{g}'(\x1-\tilde{\xi})\tilde{f}'(\tilde{\xi}-\x0)}{\tilde{f}'(\tilde{\xi}-\x0)+\tilde{g}'(\x1-\tilde{\xi})}
	\end{equation*}
	is well defined and different from zero for every $(\x0, I_0)\in\Rpi\times	\mathcal I\backslash\bar{\mathcal{I}}$. 
	Whenever $\mathcal{F}_0$ is well defined, the determinant of its Jacobian matrix is 
	\begin{equation*}
		D_{(\x0,I_0)}\mathcal{F}_0=\partial_{\x0}\x1\partial_{I_0}I_1-\partial_{I_0}\x1\partial_{\x0}I_1=1, 
	\end{equation*}
	thus $\mathcal{F}_0$ is conservative. \\
	As for the twist condition, we have $\partial_{I_0}\x1=f'(I_0)+g'(I_0)$, which is nonzero whenever $I_0\notin\bar{\mathcal{I}}$.
\end{proof}
Summarizing the previous results, we can then conclude that the set $\mathcal I\backslash\bar{\mathcal{I}}$ is the finite union of open intervals (at most eleven, but possibly the whole $(-I_c, I_c)$ if $\bar{\mathcal{I}}=\emptyset$\footnote{Numerical investigations shows that this case is consistent, in the sense that there are values of the parameters $\E, h, \mu$ and $\omega$ such that the sign of $f'+g'$ is constant (for example $\E=2.5, \omega=2, \mu=2$ and $h=2$).}), in which $\mathcal{F}_0$ is well defined, conservative and satisfies the twist condition with constant sign. 
\begin{rem}\label{oss cambio twist}
	Locally around $\pm I_c$ and $0$ the sign of $\partial_{I_0}\x1$ can be easily determined as a function of the physical parameters $\E, \omega, h, \mu$: as a matter of fact, one has 
	\begin{equation*}
		\lim_{I\to I_c^-}\partial_{I_0}\x1=\lim_{I\to-I_c^+}\partial_{I_0}\x1=+\infty
	\end{equation*}
	and 
	\begin{equation*}
		\partial_{I_0}\x1(\x0,0)=\frac{2\sqrt{\E-\frac{\om}{2}}}{\E}-\frac{4\sqrt{\Eh+\mu}}{\mu}, 
	\end{equation*}
	then, for every $\E>\om, h>0, \mu>0$
	\begin{itemize}
		\item $\exists\bar{I}\in(0,I_c)$ such that for every $I \in\mathcal I$ with $|I|>\bar{I}$ it results $\partial_{I_0}\x1>0$; 
		\item if $\frac{2\sqrt{\E-\frac{\om}{2}}}{\E}>\frac{4\sqrt{\Eh+\mu}}{\mu}$ (resp.  $\frac{2\sqrt{\E-\frac{\om}{2}}}{\E}<\frac{4\sqrt{\Eh+\mu}}{\mu}$),  $\exists\bar{\bar{I}}\in(0,I_c)$ such that for every $I \in(-I_c,I_c)$ with $|I|<\bar{\bar{I}}$ one has $\partial_{I_0}\x1>0$ (resp. $\partial_{I_0}\x1>0$); 
		\item additionally, if $\frac{2\sqrt{\mathcal E-\omega^2/2}}{\E}<\frac{4\sqrt{\Eh+\mu}}{\mu}$, the derivative $\partial_{I_0}\x1$ admits at least a change of sign, which corresponds to a change of twist for the map $\mathcal F_0$. 
	\end{itemize}
\end{rem}

\subsection{Periodic solutions on the circle}\label{sec: cerchio orbite}

Once the general properties of $\mathcal{F}_0$ on the circle are defined, we can pass to the study of its orbits. To this end, given $(\x0, I_0)\in \Rpi\times\mathcal{I}\backslash\tilde{\mathcal{I}}$, let us define the orbit of $(\x0, I_0)$ as the sequence of the iterates  $\{(\xi_k,I_k)\}_{k\in\mathbb{Z}}=\{\mathcal{F}_0^k(\x0,I_0)\}_{|_{k\in\mathbb{Z}}}$. 
\begin{defn}
	The \textit{rotation number}\footnote{\textcolor{black}{With an abuse of notation, in Section \ref{sec: kam e mather} we will use the same definition to identify the rotation number of the \textit{lift} of a map of the annulus $\Rpi\times[a,b]$, that is, its periodic extension to $\R\times[a,b]$. }} associated to $(\x0,I_0)$ through $\F_0$ is given by 
	\begin{equation}\label{def rotation number}
		\rho(\x0,I_0)=\lim_{k\to\infty}\frac{\xi_k-\x0}{k}. 
	\end{equation}
\end{defn}
In the circular case, one can easily see that for every $(\x0,I_0)$ for which $\F_0$ is well defined one has $\rho(\x0,I_0)=\bar\theta(I_0)$. \\
As the action $I_0$ is preserved on the circle, we have that, taking into account the phase space $(\xi,I)\in\Rpi\times\mathcal I $, the straight lines $\Rpi\times \{I_0\}$ are invariant for the dynamics induced by $\F_0$. We can then distinguish between two types of orbits:  
\begin{itemize}
	\item if $\bar\theta(I_0)/2\pi=p/q\in \mathbb{Q}$, then 
	\begin{equation*}
		(\xi_{q},I_q)=(\xi_0+2\pi p, I_0)\equiv_{2\pi}(\x0, I_0);  
	\end{equation*}
	in this case, we say that the point $(\x0, I_0)$ and the associated orbit are \textit{(p,q)-periodic}; 
	\item if $\bar\theta(I_0)/2\pi\notin\mathbb{Q}$, then for all $\x0\in\Rpi$ the orbit with initial point $(\x0,I_0)$ is dense in $\Rpi\times\{I_0\}$. 
\end{itemize}
A particular class of fixed points for $\F_0$ is given by the ejection-collision solutions, which form an invariant line of periodic points of period one  defined on $\Rpi\times \{0\}$. 
Taking advantage of the continuity of the function $f+g$ on $\mathcal I$, one can state the following existence result. 
 
\begin{prop}
	Given $C=\bar\theta-\pi$, where $\bar\theta$, as in (\ref{bartheta}), depends only on the physical parameters $\E, h, \mu, \omega$, for every $\rho\in(-C,C)$ there are two values $I^{\pm}\in(-I_c, I_c)$ of the actions such that, for every $\x0\in\Rpi, $ $\rho(\x0, I_0^\pm)=\rho$.\\
	In particular, for every $p,q\in\mathbb{Z}$ such that $-C<2\pi p/q<C$, there are $I^{(p,q)}_\pm\in(-I_c,I_c)$ such that for every $\x0\in\Rpi$ the points $(\x0,I_+^{(p,q)})$ and $(\x0,I_-^{(p,q)})$ are $(p,q)$-periodic.  
\end{prop}
In the circular case, the existence of two orbits of all the rotation numbers is a simple consequence of the continuity of the total shift function. As it will be analysed in \S 5, a deformation of the boundary $\partial D_0$ breaks the symmetry of the system: in general, the first return map will be not integrable anymore and more sophisticated tools should be used to retrieve, at least partially, analogous existence results. In this framework, the persistence of the twist condition under small perturbations of the boundary will play a crucial role, and this is the reason why, although not immediately used, this nondegeneracy condition has been investigated in the circular case. 
\begin{figure}
	\centering
	\includegraphics[width=.8\linewidth]{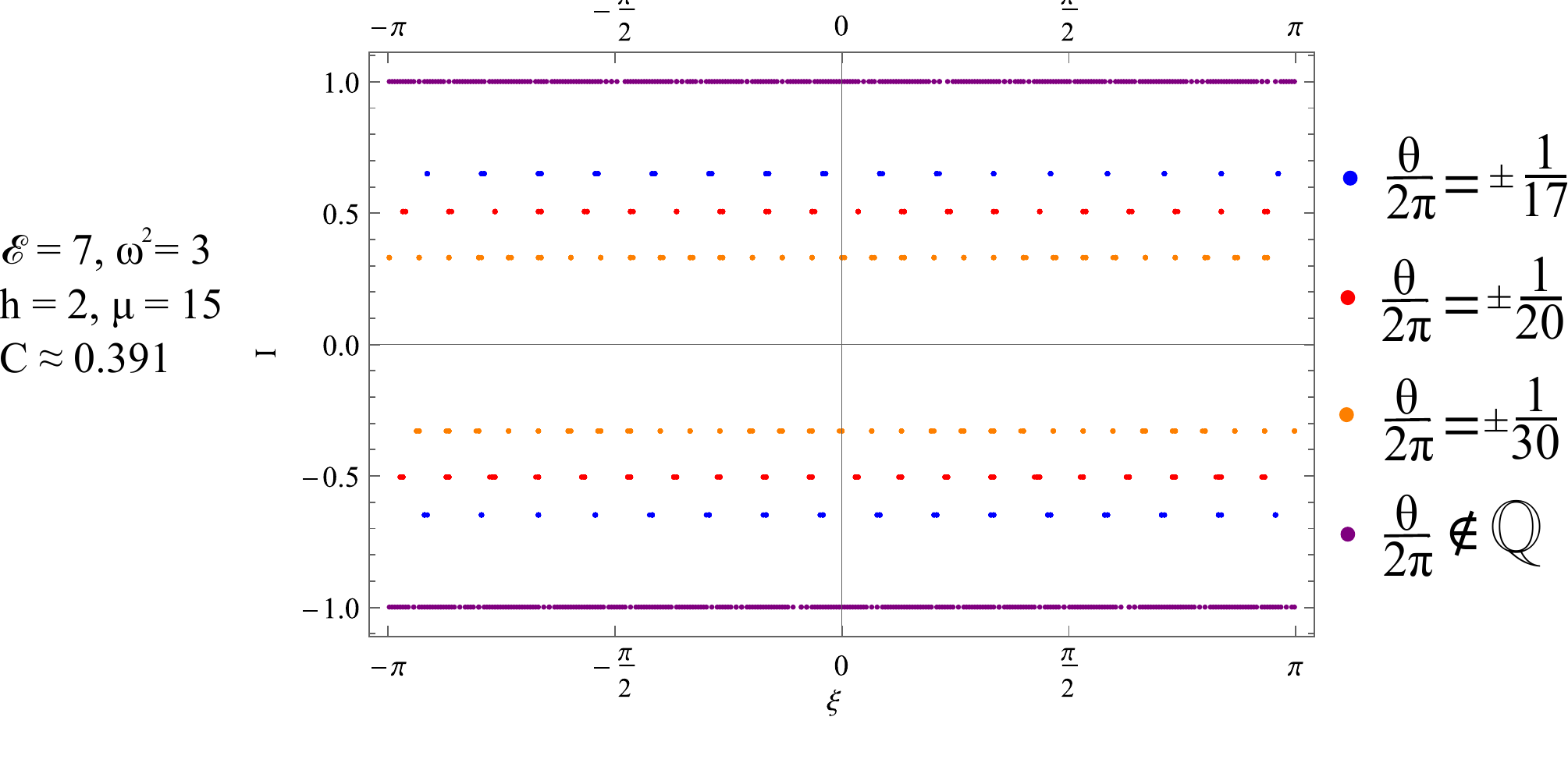}
	\caption{Examples of periodic and non-periodic orbits on the circle in the phase space $\left(\xi, I\right)\in\Rpi\times\mathcal I$.}
	\label{fig:orbite-periodiche}
\end{figure}

Under particular assumptions on the physical parameters, one can prove the existence of a second type of fixed points different from the ones which correspond to the ejection-collision solutions: 
\begin{prop}\label{prop: non homotetic}
	Fixed $\E>\om>0$, let us define 
	\begin{equation*}
		\bar{\mu}=\frac{4\E+\sqrt{8\E^3(4\E-\om)}}{2\E-\om}>2\E-\om,\quad \bar{h}=\frac{2\E-\om}{8\E^2}\mu^2-(\E+\mu).  
	\end{equation*}
	If $(\mu>\bar{\mu}\text{ }\wedge\text{ } h>\bar{h})$ or $(2\E-\om<\mu\leq\bar{\mu}\text{ }\wedge\text{ } h>0)$ there is $\bar{I}^{(1)}\in(0,I_c)$ such that for every $\x0\in\Rpi$ the points $(\x0, \bar{I}^{(1)})$ and $(\x0,-\bar{I}^{(1)})$ are non-homotetic fixed points of $\mathcal{F}$. 
	
\end{prop}

\begin{proof}
	Recalling that $f(0)+g(0)=0$, $\lim_{I\to I_c^-}f(I)+g(I)=\bar{\theta}+\pi$ and Eq.(\ref{derivata f+g}), from direct computations one has that
	\begin{itemize}
		\item fixed $\E>\om>0$ and $\mu>0$, 
		\begin{equation*}
			f'(0)+g'(0)<0\Longleftrightarrow h>\bar{h}; 
		\end{equation*}
		\item fixed $\E>\om>0$, 
		\begin{equation*}
			\bar{h}>0 \Leftrightarrow \mu>\bar{\mu} \quad\text{ and }\quad\bar{\theta}+\pi>0 \Leftrightarrow \mu>2\E-\om.
		\end{equation*}
	\end{itemize}
	If $(\mu>\bar{\mu}\text{ }\wedge\text{ } h>\bar{h})$ or $(2\E-\om<\mu\leq\bar{\mu}\text{ }\wedge\text{ } h>0)$, we have then that $f'(0)+g'(0)<0$ and $\lim_{I\to I_c^-}f(I)+g(I)>0$: as a consequence, there exists $\bar{I}^{(1)}>0$ such that $f(\bar{I}^{(1)})+g(\bar{I}^{(1)})=0=f(-\bar{I}^{(1)})+g(-\bar{I}^{(1)}),$ and then for every $\x0\in\Rpi$, $(\x0, \pm\bar I^{(1)})$ are fixed points of $\mathcal F$. Given that $\bar I^{(1)}\neq0$, these points are ot homotetic (see Figure \ref{fig:punto-fisso-non-banale}).  
\end{proof}
\begin{figure}
	\centering	
	\includegraphics[width=0.5\linewidth]{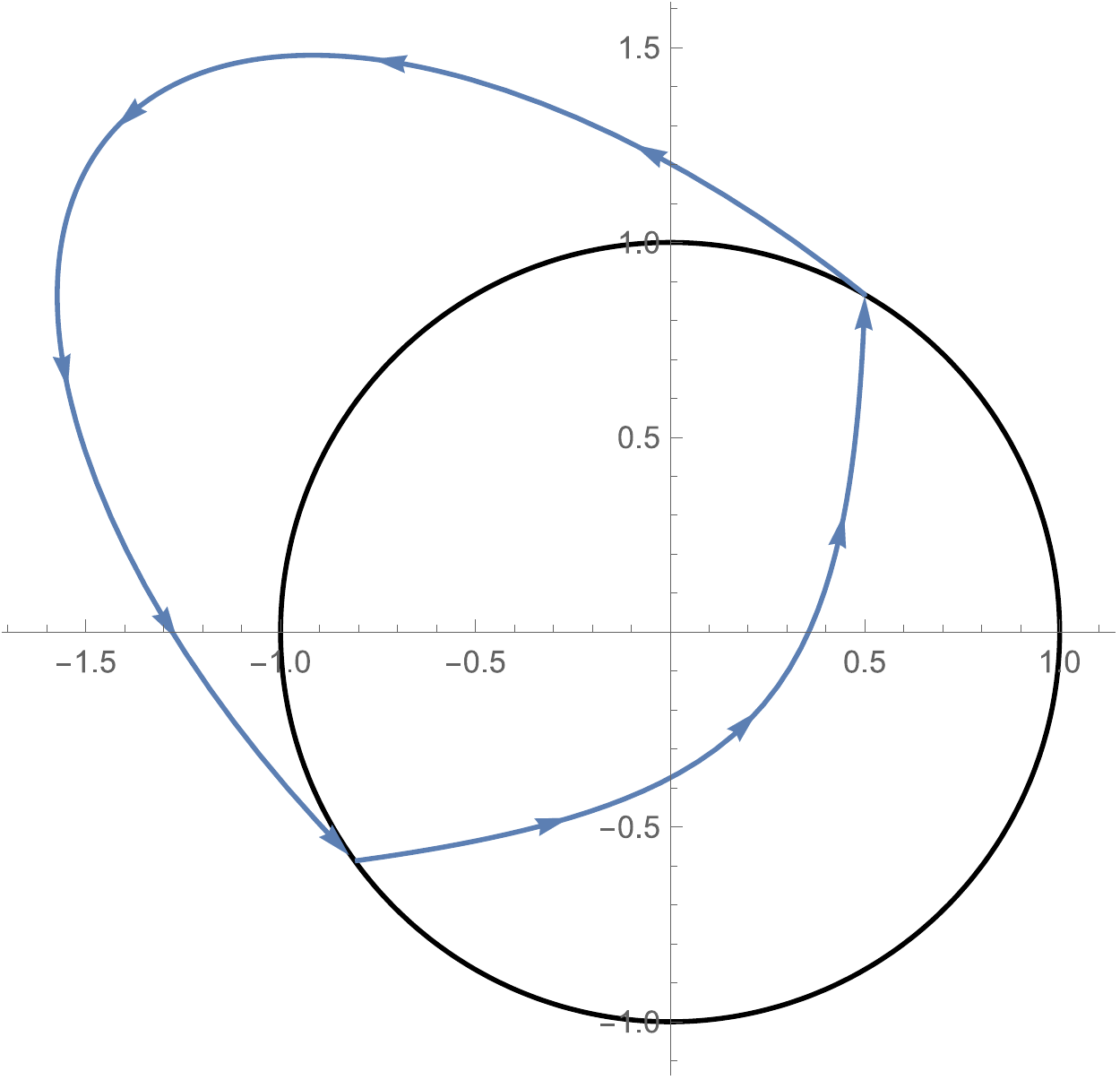}
	\caption{Example of a non-homotetic fixed point for $\mathcal{F}$, with $\E=7, \om=3, h=2, \mu=15$. In this case, with reference to Proposition \ref{prop: non homotetic}, $\bar{\mu}=41.6287$.}
	\label{fig:punto-fisso-non-banale}
\end{figure}

\subsection{Caustics for the unperturbed case}\label{sec: caustiche imperturbate}

A question of great interest in the study of billiards is that of caustics, which plays a key role in the determining the regions of the plane where the orbits can access.  A caustic is a smooth closed curve $\Gamma$ such that every trajectory which is tangent to $\Gamma$ in a point remains tangent to the latter after every passage in and out the domain $D$. The issue of the existence of caustics in standard billiards (\cite{lazutkin1973existence,MR2168892}) and its variants (\cite{gasiorek2019dynamics}) has been widely studied; in particular, in the framework of a standard convex billiard $D$, Lazutkin used the KAM approach to prove that, if $\partial D$ is sufficiently smooth (of class $C^{553}$ in the original paper \cite{lazutkin1973existence}, later improved to $C^6$ by Douady in \cite{douady1982applications}), then there exists a discontinous family of caustics in a small neighborhood of $\partial D$. \\
The aim of this Section is to extend the concept of caustic to our refractive model in the circular case: in view of the presence of two distinct dynamics inside and outside the domain $D$, one shall search for two of such curves, which can be studied separately. Moreover, by the central symmetry typical of the circular case, it is reasonable to foresee that the inner and outer caustics are circles of suitable radii depending on the action $I_0$. 
\begin{thmv}\label{th caustiche}
	For every $\E, h, \omega, \mu>0, $ $\E>\om$, given $I_0\in\left(-I_c, I_c\right)$: 
	\begin{itemize}
		\item the exterior caustic $\Gamma_E(\zeta; I_0)$ is given by the locus of the apocenters of the outer ellipses, namely, 
		\begin{gather*}
			\Gamma_E:[0,2\pi]\to\R^2, \quad \Gamma_E(\zeta; I_0)=R_E(\cos{\zeta}, \sin{\zeta}), \\
			R_E=\frac{\sqrt{\E+\sqrt{\E^2-2 I_0^2\om}}}{\omega}; 
		\end{gather*}
	\item the interior caustic $\Gamma_I(\zeta; I_0)$ is the locus of the pericenters of the inner Keplerian hiperbol\ae. In particular, 
	\begin{gather*}
		\Gamma_I:[0,2\pi]\to\R^2\quad \Gamma_I(\zeta; I_0)=R_I(\cos{\zeta}, \sin{\zeta}), \\
		R_I=\frac{p}{1+e}, 
	\end{gather*} 
where
\begin{equation*}
	p=\frac{2 I_0^2}{\mu}, \quad e=\sqrt{1+\frac{4 I_0^2 (\Eh)}{\mu^2}}. 
\end{equation*}
	\end{itemize}
\end{thmv}
In general, following \cite{MR2168892} and \cite{capitanio2004geodesic}, one shall give the following characterization for the caustic: take an orbit for our dynamical system and suppose that one of its (interior or exterior) branches is implicitely defined through the relation 
\begin{equation}\label{impl branch}
	G(x,y; \xi)=0, 
\end{equation}
where $G:\R^2\to \R$ is of class $C^2$ in all the variables and $\xi$ acts as a parameter (for example, it could denote the polar angle of the initial point of the branch, its pericenter or apocenter). The caustic $\Gamma$ can be then seen as th envelope of the family of curves defined by (\ref{impl branch}) varying $\xi\in[0,2\pi]$, that is, the set of points $(x_0(\xi), y_0(\xi))$ satisfying 
\begin{equation}\label{sistema caustica}
	\begin{cases}
	G(x,y;\xi)=0\\
	\partial_\xi G(x,y; \xi)=0 
	\end{cases}.
\end{equation}
By means of the implicit function theorem, it is straightforward that if 
\begin{equation}\label{nondeg cond}
	\nabla_{(x,y)}G(x,y;\xi)\nparallel\nabla_{(x,y)}\partial_\xi G(x,y;\xi) \text{ on the solutions of (\ref{sistema caustica})}, 
\end{equation}
then (\ref{sistema caustica}) defines a regular curve $\Gamma(\xi)=(x_0(\xi), y_0(\xi))$. \\
The proof of Theorem \ref{th caustiche} relies on the evaluation of (\ref{sistema caustica}) in the particular cases of the inner and outer dynamics: in the case of  circular domains, the solutions of such system can be computed explicitely.

\paragraph{\textbf{Outer caustic}}

 Given $p_0=(p_x, p_y)=e^{i\x0}, v_0=(v_x,v_y)\in \R^2$, $|v_0|=\sqrt{2\E-\om}$, from the proof of Theorem \ref{thm ex globale esterna unp} (see Appendix \ref{appA}), one has that the solution of 
\begin{equation}\label{sistema esterna caustica}
	\begin{cases}
		(HS_E)[z(s)] &s\in[0,T_E]\\
		z(0)=p_0, ~z'(0)=v_0
	\end{cases}
\end{equation}
can be parametrized as 
\begin{equation*}
	(x(s), y(s))=\left(p_x \cos(\omega s)+\frac{v_x}{\omega}\sin(\omega s), p_y \cos(\omega s)+\frac{v_y}{\omega}\sin(\omega s)\right). 
\end{equation*}
If, as in \S \ref{sec: first return},  $\alpha\in(-\pi/2, \pi/2)$ denotes the angle between $p_0$ and $v_0$, recalling the definition of canonical action (\ref{I alpha}) one has 
\begin{equation*} 
	v_0=\sqrt{2\E-\om}(\cos\alpha~ p_0+\sin\alpha~ t_0)=\sqrt{2\E -\om -2I_0^2}~p_0+\sqrt{2}I_0~ t_0, 
\end{equation*}
were $t_0=ie^{i\xi}$ is the tangent unit vector to $\partial D_0$ in $p_0$. Since in the circular case  the action $I_0$ is constant along the orbits, it can be treated as a parameter in $\mathcal I$. Additionally, consider the non-homotetic case, that is, suppose $I_0\neq0$ (the case $I_0=0$ can be easily analysed separately, leading to te same result).\\
Taking the function $r^2(s)=x^2(s)+y^2(s)$, by direct computations one has 
\begin{equation*}
	\begin{aligned}
		a^2=\max_{s\in\left[0,\frac{2\pi}{\omega}\right]}r^2(s)=\frac{\E+\sqrt{\E^2-2I_0^2\om}}{\om}>0\\
		b^2=\min_{s\in\left[0,\frac{2\pi}{\omega}\right]}r^2(s)=\frac{\E-\sqrt{\E^2-2I_0^2\om}}{\om}>0. 
	\end{aligned}
\end{equation*}

In the reference frame whose axes coincide with the ellipse's ones, denoted with $R(O, x'', y'')$, the outer arc can be then implicitely defined as a segment of the conic 
\begin{equation*}
	G_{E,0}(x'',y'')=\frac{{x''}^2}{a^2}+\frac{{y''}^2}{b^2}-1=0. 
\end{equation*}
Denoting by $\zeta$ the polar angle of one of the apocenter points, all the solutions of (\ref{sistema esterna caustica}) with $\angle{(p_0, v_0)}=\alpha$ are then implicitely defined by 
\begin{equation*}
	G_E(x,y;\zeta)=\frac{\left(x \cos\zeta+y\sin\zeta\right)^2}{a^2}+\frac{\left(y \cos\zeta-x \sin\zeta\right)^2}{b^2}-1=0
\end{equation*}
where $\zeta\in[0,2\pi]$ is treated as a parameter. \\
As 
\begin{equation*}
	\partial_\zeta G_E(x, y; \zeta)=\frac{\sqrt{\E-2I_0^2\om}}{I_0^2}\left((x^2-y^2)\sin(2\zeta)-2xy\cos(2\zeta)\right), 
\end{equation*}
the explicit formulation of (\ref{sistema caustica}) for the outer arcs is then, for $\zeta\neq k\frac{\pi}{2}$, $k=0,1,2,3$ 
\begin{equation}\label{sistema caustica ext 2}
	\begin{cases}
		\frac{\left(x \cos\zeta+y\sin\zeta\right)^2}{a^2}+\frac{\left(y \cos\zeta-x \sin\zeta\right)^2}{b^2}-1=0\\
		\frac{\sqrt{\E^2-2I_0^2\om}}{I^2}\sin(2\zeta)\left(x+y \cot\zeta\right)\left(x-y \tan\zeta\right); 		
	\end{cases}
\end{equation}
note that, in the degenerate cases $\zeta=k \frac{\pi}{2}$, from $\partial_\zeta G_E(x,y; \zeta)=0$ one obtains $x=0$ or $y=0$. \\
The solutions of (\ref{sistema caustica ext 2}) are the ellipse's apocenters and pericenters: since the outer arc of the considered dynamical systems involves only the first apocenter, the only admissible solution of (\ref{sistema caustica ext 2}) is given by 
\begin{equation*}\label{caustiche sol ext}
	(\bar{x}(\zeta), \bar{y}(\zeta))=\frac{\sqrt{\E+\sqrt{\E^2-2I_0^2\om}}}{\omega}\left(\cos\zeta,\sin\zeta\right), 
\end{equation*} 
which describes, for $\zeta\in[0,2\pi]$, the caustic $\Gamma_E(\zeta; I_0)$ as the circle of raduis $R_E$ of Theorem \ref{th caustiche}. \\
Although the caustics for the circular domain are completely determined, let us investigate the nondegeneracy condition (\ref{nondeg cond}), which will be generalized for small perturbations of $D_0$ in Section \ref{sec: caustiche perturbate}. By direct computations, one has
\begin{equation*}
	\begin{aligned}
		\nabla_{(x,y)}G_E(x,y;\zeta)_{|(\bar{x}, \bar{y})}=2a\frac{\sqrt{\E^2-2I_0^2\om}}{I_0^2}\left(-\cos\zeta, -\sin\zeta\right), \\	\nabla_{(x,y)}\partial_\zeta G_E(x,y;\zeta)_{|(\bar{x}, \bar{y})}=2a\frac{\sqrt{\E^2-2I_0^2\om}}{I_0^2}\left(\sin\zeta, -\cos\zeta\right), 
	\end{aligned}
\end{equation*}
leading to 
\begin{equation}\label{nondeg ex nonpert}
	\nabla_{(x,y)}G_E(x,y;\zeta)_{|_{(\bar{x}, \bar{y})}}\perp \nabla_{(x,y)}\partial_\zeta G_E(x,y;\zeta)_{|_{(\bar{x}, \bar{y})}}. 
\end{equation}

\paragraph{\textbf{Inner caustic}}

Let us now consider the inner problem 
\begin{equation}\label{caustiche inner prob}
	\begin{cases}
		(HS_I)[z(s)] &s\in[0,T_I]\\
		z(0)=p_0, ~z'(0)=v_0, 
	\end{cases}
\end{equation}
and denote with $\alpha=(\pi/2, 3\pi/2)$ the angle between $p_0$ and $v_0$. Recalling (\ref{polar}) and (\ref{p e int}), and given that 
\begin{equation*}
	k=|p_0\wedge v_0|=|p_0||v_0|\sin\alpha=\sqrt{2}I_0, 
\end{equation*}
one has that the polar equation of the Keplerian inner arc is 
\begin{equation*}
	r=\frac{p}{1+e\cos{f}}, 
\end{equation*}
with 
\begin{equation*}
	p=\frac{2I_0^2}{\mu},\quad e=\frac{\sqrt{\mu^2+4I_0^2(\Eh)}}{\mu}.  
\end{equation*}
Choosing the reference frame $R(O, x'', y'')$ where the pericenter is on the positive branch of the $x$-axis, the inner Keplerian arc is expressed by 
\begin{equation}\label{cartesiana iperbole 0}
	G_{I}(x'',y'')=(e^2-1){x''}^2-{y''}^2-2 p e x''+p^2=0, \quad x''\leq \frac{p}{e+1}, 
\end{equation}
where the inequality condition expresses the choice of the branch of the hiperbola whith the concavity in the direction of the central mass.	\\
As in the outer case, denoting with $\zeta$ the polar angle of the pericenter, one has that all the Keplerian hyperbol\ae~ with central mass $\mu$, energy $\Eh$ and angular momentum $k=\sqrt{2}I_0$ are given by 
\begin{gather*}
	G_I(x,y;\zeta)=(e^2-1)(x\cos\zeta+y\sin\zeta)^2-(y\cos\zeta-x\sin\zeta)^2-2pe(x\cos\zeta+y\sin\zeta)+p^2=0,\\
	x\cos\zeta+y\sin\zeta\leq\frac{p}{e+1} , 
\end{gather*}
with $\zeta\in[0,2\pi]$. The system (\ref{sistema caustica}) in the inner case becomes then
\begin{equation}\label{caustiche sistema inner}
	\begin{cases}
		(e^2-1)(x\cos\zeta+y\sin\zeta)^2-(y\cos\zeta-x\sin\zeta)^2-2pe(x\cos\zeta+y\sin\zeta)+p^2=0\\
		2e \left(y \cos\zeta+x \sin\zeta\right)(p-e~ x \cos\zeta+e~ y \sin\zeta)=0
	\end{cases}
\end{equation}
with the additional condition $(x\cos\zeta+y\sin\zeta)\leq p/(1+e)$. Problem (\ref{caustiche inner prob}) admits the unique solution 
\begin{equation*}
	\left(\bar{x}(\zeta), \bar{y}(\zeta)\right)=\frac{p}{1+e}(\cos\zeta, \sin\zeta)
\end{equation*}
which corresponds to the position of the pericenter of the corresponding Keplerian arc, taking $\zeta$ as a parameter. The inner caustic $\Gamma_I(\zeta, I_0)$ is then expressed by a circle with radius $R_I$ as in Theorem \ref{th caustiche}. \\
It is straightforward to verify that the nondegeneracy condition (\ref{nondeg cond}) is verified: from 
\begin{equation*}
	\begin{aligned}
		\nabla_{(x,y)}G_I(x,y; \zeta)_{|_{(\bar{x}, \bar{y})}}=-2p(\cos{\zeta}, \sin\zeta), \quad 
		\nabla_{(x,y)}\partial_\zeta G_I(x,y; \zeta)_{|_{(\bar{x}, \bar{y})}}=\frac{2ep}{1+e}(\sin{\zeta}, -\cos\zeta), 
	\end{aligned}
\end{equation*}
one has
\begin{equation*}
	\nabla_{(x,y)}G_I(x,y; \zeta)_{|(\bar{x}, \bar{y})}\perp \nabla_{(x,y)}\partial_\zeta G_I(x,y; \zeta)_{|(\bar{x}, \bar{y})}.
\end{equation*}
\section{Perturbations of the circle}\label{sec: perturbazioni}

Many of the results obtained in the circular case, although significant in themselves, can be generalized to non-circular smooth domains, provided they are close enough to $D_0$ in a way which will be soon specifed. This extension can be performed by means of classical perturbation theory, as well as of more sophisticated results such as KAM and Aubry-Mather theorems (see \cite{gole2001symplectic, moser1962invariant, jurgen1986recent, mather58084existence, aubry1983discrete}). \\
To this end, let us consider a class of domains $D_\ep$ whose bundary $\partial D_\ep=supp(\gep)$ is given by a radial deformation of the circle of the form 
\begin{equation}\label{def dominio perturbato}
\gep:\Rpi\to \R^2\quad \gep(\xi)=\left(1+\ep f(\xi; \ep)\right)e^{i\xi}, 
\end{equation}
where $f(\xi; \ep)$ is a smooth function of $\Rpi\times [-C_\ep, C_\ep]$, with $C_\ep>0 $ arbitrarily high; note that, from the choice of the parametrization of $\gep$, the variable $\xi$ still represents the polar angle of $\xi$.\\
This Section aims to analyze the generating function $S_\ep$, with particular emphasis to its good definition and nondegeneracy properties, and the associated first return map $\mathcal F_\ep$, whose orbits, when possible, will be studied in terms of their rotation numbers.

\subsection{Global existence of the outer and inner arcs for the perturbed dynamics}\label{sec: esistenza archi perturbati}	
 As for the circular case (see Section \ref{sec: esistenza archi cerchio}), the generating function associated to the so-called \textit{perturbed dynamics}, that is, the dynamics induced by the potential (\ref{potential}) inside and outside the perturbed domain $D_\ep$, is given by 
 \begin{equation*}
 	S(\x0, \x1; \ep)=d_E(\gep(\x0), \gep(\xt))+d_I(\gep(\xt), \gep(\x1)),
 \end{equation*}
where $\gep(\xt)$ is the passage point of $z_{EI}(s)$, as defined in \S \ref{sec: first return}, through $\partial D_\ep$. \\
A preliminary passage to discuss the good definition of $S_\ep$ as a whole is to ensure that the functions $d_E(p_0, p_1)$ and $d_I(p_0, p_1)$ are differentiable as functions of $p_0, p_1\in\partial D_\ep$, namely, that the inner and outer dynamics admit a unique geodesic arc joining $p_0$ and $p_1$. \\
In view of the results of Section \ref{sec: esistenza archi cerchio} and Appendix \ref{appA}, this follows from the continous dependence of the solutions of the fixed ends problems (\ref{problema esterno testo}) and (\ref{PB2}) with respect to $p_0$ and $p_1$. To fix the notation, let us denote with $z_{E\backslash I}(s; p_0, p_1; 0)$ the respective solutions in the unperturbed circular case, where the last variable refers to $\ep=0$. 
 
 \begin{rem}\label{oss extcont}
 	Focusing on the outer problem, from the continous dependence on $p_0$ and $p_1$ of the solution $z_E(s;p_0,p_1; 0)$ defined in Theorem \ref{thm ex globale esterna unp}, along with the invariance of the system under rotations, there exists $\rho_E>0$ such that for every $p_0, p_1, \tilde{p_0},\tilde{p}_1$ satisfying $|p_0-p_1|<2$ and $|\tilde{p}_0-p_0|,|\tilde{p}_1-p_1|<\rho_E$ one finds $T>0$ and a unique solution $z_E(s; \tilde{p}_0, \tilde{p}_1)$ of the problem 
 	\begin{equation}
 		\begin{cases}
 			(HS_E)[z(s)] &s\in[0,T]\\
 			z(0)=\tilde{p}_0, \quad z(T)=\tilde{p}_1. 
 		\end{cases}
 	\end{equation}
 	For computational reasons, we require $\rho_E<1$, ad set
 	\begin{equation}\label{strip}
 		\mathcal{S}_{\rho_E}=\bigcup_{p_0\in\partial D_0}B_{\rho_E}(p_0)=\{p\in\mathbb{R}^2\text{ }|\text{ } dist(p,\partial D_0)<\rho_E\}.
 	\end{equation} 
 \end{rem}
 \begin{prop}\label{extpert}
 	There exists $\delta>0$ such that for every $\tilde{p}_0,\tilde{p}_1\in\mathcal S_{\rho_E}$ with $|\tilde{p}_0-\tilde{p}_1|<\delta$ there is $T>0$ and a unique $z_E(s;\tilde{p}_0,\tilde{p}_1)$ solution of the fixed-ends problem
 	\begin{equation}
 		\begin{cases}
 			(HS_E)[z(s)] &s\in[0,T]\\
 			z(0)=\tilde{p}_0, \quad z(T)=\tilde{p}_1. 
 		\end{cases}
 	\end{equation}
 \end{prop}
 \begin{proof}
 	It is sufficient to set $\delta<2(1-\rho_E)$. Denoting in polar coordinates $\tilde{p}_0=r_0e^{i\theta_0}$ and $\tilde{p}_1=r_1e^{i\theta_1}$, cosider $p_0=e^{i\theta0}$ and $p_1=e^{i\theta_1}$: we have then $|\tilde{p}_0-p_0|, |\tilde{p}_1-p_1|<\rho_E$, and 
 	\begin{equation}
 		|p_0-p_1|\leq|\tilde{p}_0-p_0|+|\tilde{p}_0-\tilde{p}_1|+|\tilde{p}_1-p_1|<\delta+2\rho_E<2, 
 	\end{equation}
 	then, by Remark \ref{oss extcont}, the thesis is proved. 
 \end{proof}
 \begin{thm}\label{thm esistenza ext pert}
 	There are $\bar{\delta}_E>0, \bar{\epsilon}_E>0$ such that for every $\ep\in\R$ and for every $\xi_0,\xi_1\in\Rpi$ with $|\xi_0-\xi_1|<\bar{\delta}_E$ and $|\ep|<\bar\ep_E$ there is $T>0$ and a unique function $z_E(s;\gamma_{\epsilon}(\xi_0),\gamma_{\epsilon}(\xi_1))\equiv z_E(s; \x0, \x1; \ep)$ which is a classical solution of 
 	\begin{equation}\label{ext prob pert}
 		\begin{cases}
 			(HS_E)[z(s)] &s\in[0,T]\\
 			z(0)=\gamma_{\epsilon}(\xi_0), \quad z(T)=\gamma_{\epsilon}(\xi_1). 
 		\end{cases}
 	\end{equation}
 \end{thm}
 \begin{proof}
 	The claim is true if $\bar{\delta}_E<\delta/2$ and $\bar{\epsilon}_E<\min\{ 1/\left(\|f\|_\infty+\|\partial_{\xi} f\|_{\infty}\right), \rho_E/\|f\|_{\infty}\}$. For, fixed $\ep\in\R$ such that $|\ep|<\bar\ep_E$, one has that for every $\xi\in\Rpi$
 	\begin{equation*}
 	|\dot\gamma_\ep(\xi)|=|\ep\partial_\xi f(\xi, \ep)e^{i\xi}+\left(1+\ep f(\xi, \ep)\right)ie^{i\xi}|\leq1+|\ep|\left(\|\partial_\xi f\|_\infty+\|f\|_\infty\right)<2.
 	\end{equation*}
 If $\xi_0,\xi_1\in\Rpi$ are such that $|\xi_0-\xi_1|<\bar{\delta}_E<\delta/2$, defining $\tilde{p}_0=\gamma_{\epsilon}(\xi_0)$, $\tilde{p}_1=\gamma_{\epsilon}(\xi_1)$, $p_0=e^{i\xi_0}$ and $p_1=e^{i\xi_1}$: 
 	\begin{equation}
 		\begin{aligned}
 			&|\tilde{p}_0-p_0|=|\epsilon|\text{ } |f(\xi_0)|\leq|\epsilon|\text{ }\|f\|_{\infty}<\rho_E, \\
 			&|\tilde{p}_1-p_1|<\rho_E, \\
 			&|\tilde{p}_0-\tilde{p}_1|<\|\dot\gamma_\ep\|_\infty|\x0-\x1|<2|\x0-\x1|<\delta,
 		\end{aligned}
 	\end{equation}
 	then the hypotheses of Proposition \ref{extpert} hold and the claim is true. \\ 	
 \end{proof}

Passing to the inner dynamics, let us observe that the definition of windng number given in (\ref{wind unp}) can be extended to regular, simple curves joining points which are in a sufficiently small tubolar neighborhood of $\partial D_0$: if $\mathcal S_\rho$ is defined as in (\ref{strip}),  $\tilde{p}_0=r_0e^{i\theta_0},\tilde{p}_1=r_1e^{i\theta_1}\in \mathcal S$ with $|\theta_0-\theta_1|\neq\pi$ and $\tilde{\alpha}(s)$ is such that $\tilde{\alpha}(a)=\tilde{p}_0$, $\tilde{\alpha}(b)=\tilde{p}_1$, one can construct the closed curve $\Gamma_{\tilde{\alpha}}$ by following $\tilde{\alpha}$ in $[a,b]$, then the shortest arc of $\partial B_{r_1}(0)$ joining  $\tilde{p}_1$ and $\tilde{p}'_0=r_1e^{i\theta_0}$ and finally the ray of $e^{i\theta_0}$ from $\tilde{p}'_{0}$ to $\tilde{p}_0$ (see Figure \ref{fig:avvolgimento-pert}). If $\rho$ is sufficiently small, the definition (\ref{wind unp}) can be straightforwardly extended to $\tilde{\alpha}$. \\
\begin{figure}
	\centering
	\includegraphics[width=0.5\linewidth]{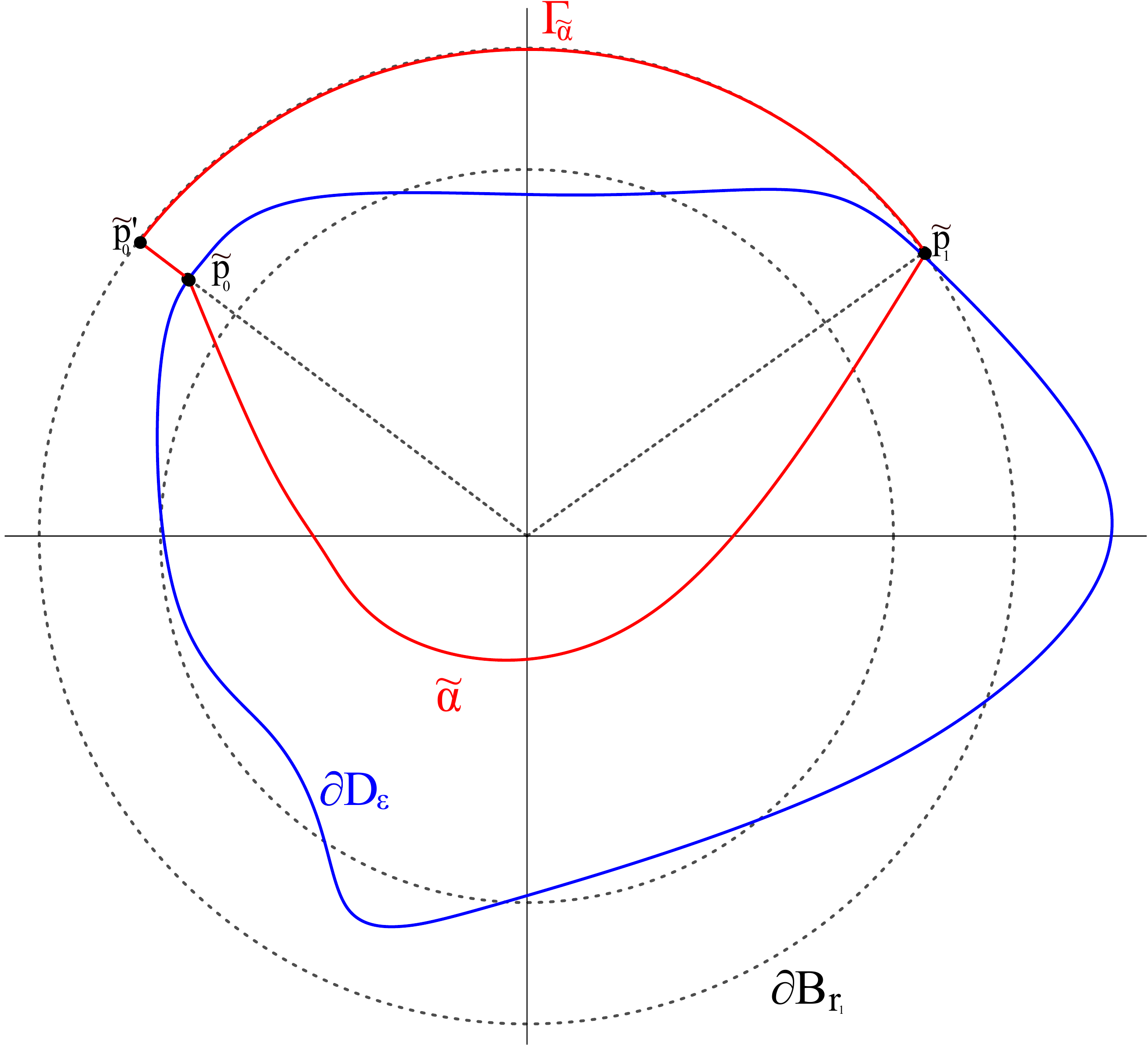}
	\caption{Curve $\Gamma_{\tilde \alpha}$ for the computation of the winding number in the perturbed case. }
	\label{fig:avvolgimento-pert}
\end{figure}

\begin{thmv}\label{thm ex inner pert}
	There exist $\bar{\ep}_I>0$, $\bar{\delta}_I>0$ and $C>0$ such that for every $\ep\in\R$,  $\xi_0,\xi_1\in\Rpi$ satisfying $|\xi_0-\x0|<\bar{\delta}_I$ and $|\ep|<\bar\ep_I$ there exists a unique $T(\x0,\x1)\equiv T>0$ and a unique solution $z_I(s; \x0, \x1; \ep)$ of 
	\begin{equation}\label{PB}
		\begin{cases}
		(HS_I)[z(s)] &s\in[0,T]\\
		z(s)\in D_\ep &s\in(0, T)\\
		z(0)=\gep(\x0),\text{ }z(T)=\gep(\x1) 
		\end{cases}
	\end{equation}	
	with the following properties: 
	\begin{itemize}
		\item if $\x0\neq\x1$, then $z(s; \x0,\x1; \ep)$  is a classical solution of (\ref{PB}) such that $|I(z([0,T]),0)|=1$; 
		\item if $\x0=\x1$, $z(s; \x0,\x0; \ep)$ is an ejection-collision solution.
	\end{itemize}
	In any case, $z(s; \x0,\x1; \ep)$ is of class $C^1$ with respect to variations of $\x0$ and $\x1$ and, if we define 
	\begin{equation}
		\beta_0=\angle(\dot{\gep}(\x0),z'(0; \x0,\x1; \ep)), \text{ }\beta_1=\angle(\dot{\gep}(\x1),z'(T; \x0, \x1; \ep)), 
	\end{equation}
	one has that $|\beta_0|>C$ and $|\beta_1|>C$.
\end{thmv}

\begin{proof}
	As in the case of the outer dynamics, in view of Theorem \ref{thm ex globale interna unp}, there exists $0<\rho_I<1$ such that, choosing  $\delta<2(1-\rho_I)$, for every $\tilde p_0=r_0e^{i\theta_0}, \tilde p_1=r_1e^{i\theta_1}\in\mathcal S_{\rho_I}$ the problem 
	\begin{equation*}
		\begin{cases}
			(HS_I)[z(s)]\quad &s\in[0, T]\\
			z(0)=\tilde p_0,\text{ }z(T)=\tilde p_1
		\end{cases}
	\end{equation*}
	admits a unique solution $z_I(s; \tilde p_0, \tilde p_1)$, with $T>0$ depending on the endpoints. Moreover, possibly reducing $\rho_I$, defining $p_0=e^{i\theta_0}$ and $p_1=e^{i\theta_1}$ and considering $z_I(s; p_0, p_1)$ introduced by Theorem \ref{thm ex globale interna unp}, one has that $z_I(s; \tilde p_0, \tilde p_1)$ and $z_I(s; p_0, p_1)$ have the same winding number. Choosing then $\ep_1<\min\left\{\frac{\rho_I}{\|f\|_\infty}, \frac{1}{\|f\|_\infty+\|\partial_\xi f\|_\infty}\right\}$ and $\bar\delta_I=\delta<2(1-\rho_I)$, for any $\x0, \x1\in\Rpi$ with $|\x0-\x1|<\delta$ and  any $|\ep|<\ep_1$,  following the proof of Theorem \ref{thm esistenza ext pert}, the existence and uniqueness of the solution $z_I(s; \x0, \x1; \ep)$ of problem (\ref{PB}) is ensured, with preservation of the winding number with respect to the unperturbed case.   
	
 From the regularity of $\gep$ with respect to $\xi$ and of $z(s;\tilde p_0, \tilde p_1)$ with respect to the endpoints, we have that $z(s;\x0, \x1; \ep)$ is of class $C^1$ in the variables $\x0$ and $\x1$. \\
	To prove the transversality properties of $z_I(s; \x0, \x1; \ep)$, let us observe that by Proposition \ref{transv} and the differentiability of $z(s)\equiv z(s;\x0, \x1; \ep)$, as well as by the invariance under rotations of the system, there is $C_1>0$, possibly lower than $C$, such that for every $\x0,\x1$ satisfying the existence hypotheses, defined $\tilde{p}_0$ and $\tilde{p}_1$ as above, one has 
	\begin{equation}
		1\geq-\frac{\tilde{p}_0}{|\tilde{p}_0|}\cdot\frac{z'(0)}{|z'(0)|}>C_1, \text{ }
		1\geq\frac{\tilde{p}_1}{|\tilde{p}_1|}\cdot\frac{z'(T)}{|z'(T)|}>C_1. 
	\end{equation}
	Let us now consider $\alpha_1=\angle{\left(\tilde{p}_1, z'(T)\right)}$ (if $\alpha_0=\angle(-\tilde{p}_0,z'(0))$ we proceed analogously): setting $C_2=\arccos(C_1)\in(0,\pi/2)$, we have $|\alpha_1|<C_2$, and, taking $\beta_1=\angle(\dot{\gep}(\x1),z'(T))$, one has 
	\begin{equation}
		\beta_1=\angle(\dot{\gep}(\x1), \tilde{p}_1^{\perp})+\angle(\tilde{p}_1^{\perp},z'(T)), 
	\end{equation}
	where $\tilde{p}_1^{\perp}=i\tilde{p}_1$. Then we have that
	\begin{equation}
		\begin{aligned}
			|\beta_1|\geq \Big|\angle(\tilde{p}_1^{\perp},z'(T))\Big|-\Big|\angle(\dot{\gep}(\x1), \tilde{p}_1^{\perp})\Big|, \\
			\Big|\angle(\tilde{p}_1^{\perp},z'(T))\Big|>\frac{\pi}{2}-C_2\equiv C_3\in\left(0,\frac{\pi}{2}\right). 
		\end{aligned}
	\end{equation}
	To estimate $\angle(\dot{\gep}(\x1), \tilde{p}_1^{\perp})$, let us observe that 
	\begin{equation}
		\big|\sin\left(\angle(\dot{\gep}(\x1), \tilde{p}_1^{\perp})\right)\big|=\frac{|\dot{\gep}(\x1)\wedge\tilde{p}_1^\perp|}{|\dot{\gep}(\x1)||\tilde{p}_1^\perp|}, 
	\end{equation}
where $\tilde{p}_1^\perp/|\tilde{p}_1^\perp|=ie^{i\x1}$ and $\dot{\gep}(\x1)=\ep \partial_\xi f(\x1,\ep)e^{i\x1}+(1+\ep f(\x1, \ep))ie^{i\x1}$. If $|\ep|<\ep_1$, 
	\begin{equation}
		\begin{aligned}
			 &|\dot{\gep}(\x1)|=\sqrt{(\ep \partial_\xi f(\x1, \ep))^2+(1+\ep f(\x1, \ep))^2}\geq1-|\ep|\text{ }\|f\|_\infty>1-\rho_I\\
			&\Rightarrow \big|\sin\left(\angle(\dot{\gep}(\x1), \tilde{p}_1^{\perp})\right)\big|=\frac{|\ep| \text{ }|\partial_\xi f(\x1, \ep)|}{|\dot{\gep}(\x1)|}<\frac{|\ep|}{1-\rho_I}|\partial_\xi f(\x1,\ep)|. 
		\end{aligned}
	\end{equation}
	If we consider $C_4>0$ such that $0<\arcsin(C_4)<C_3$, setting $\ep<\min\left\{\ep_1,\frac{C_4(1-\rho_I)}{\|\partial_\xi f\|_{\infty}}\right\}=\bar{\ep}_I$:  
	\begin{equation}
		\big|\sin\left(\angle(\dot{\gep}(\x1), \tilde{p}_1^{\perp})\right)\big|<C_4\Rightarrow|\beta_1|>C_3-\arcsin{(C_4)}\equiv\bar{C}>0. 
	\end{equation}
	Recalling the definitions which lead to $\bar{\ep}$ and $\bar{C}$, it is clear that they do not depend on $\xi_0$ nor $\x1$, but only on $\rho$ and the global properties of $f$. \\
	Finally, the condition $z(s)\in D_\ep$ for $s\in(0,T)$ follows from the smallness of $\ep$ and the transversality of $z(s)$ with respect to the perturbed domain $\partial D_\ep$, which ensures that $z(s)$ do not intersect twice the domain's boundary in a neighborhood of $z(T)$. 
\end{proof}

\begin{rem}
	Using the same transversality argument described in details for the inner dynamics, one can prove that, if $\ep$ is small enough and $\x0, \x1$ sufficiently close, the solution $z_E(s; \x0, \x1; \ep)$ of (\ref{ext prob pert}), whose existence is ensured by Theorem \ref{thm esistenza ext pert}, is such that $z_E(s; \x0, \x1; \ep)\notin\bar D_\ep$ for $s\in(0, T)$. 
\end{rem}

\subsection{Invariant sets for $\mathcal{F}_\ep$}\label{sec: kam e mather}
The good definition of the distances $d_E(p_0, p_1)$ and $d_I(p_0, p_1)$ for $p_0, p_1\in\partial D_\ep$ allows to consider the associated generating function 
\begin{equation}
	S(\x0,\x1; \ep)=S_{E}(\x0,\xt; \ep )+S_{I}(\xt, \x1; \ep)=d_E(\gep(\x0),\gep(\xt))+d_I(\gep(\xt),\gep(\x1)).  
\end{equation} 
When well defined, $S(\x0, \x1; \ep)$ has the same regularity of $f(\xi, \ep)$ as a function of both the angle variables $\x0, \x1$ and the perturbative parameter $\ep$.\\ 
This Section aims to prove that, under suitable assumptions, the results proved for the circle regarding the twist condition and the existence of invariant sets with prescribed rotation numbers (see Section \ref{sec: cerchio orbite}) can be extended to the perturbed dynamics as described in (\ref{def dominio perturbato}).\\
Recalling the notation of Section \ref{sec: cerchio}, $\bar{\mathcal{I}}$ is the finite set in $\mathcal I=\left(-I_c, I_c\right)$ for which $\mathcal F_0$ is not well defined, with $I_c=\sqrt{\E-\om/2}$. To highlight the dependence on $\ep$, from now on we will use the notation $\mathcal F_0(\x0, I_0)\equiv \mathcal F(\x0, I_0; 0)$. 
\begin{prop}\label{prop mappa perturbata}
	Let $[a,b]\subset\mathcal I\backslash\bar{\mathcal{I}}$, and suppose that, in (\ref{def dominio perturbato}), $f\in C^k\left(\Rpi\times\mathcal I\right)$ with $k\geq2$. Then there exists $\bar\ep>0$ such that for every $\ep\in\R$, $|\ep|<\bar\ep$, the perturbed first return map 
	\begin{equation*}
		\mathcal F(\x0, I_0; \ep)=\left(\x1\left(\x0, I_0; \ep\right), I_1\left(\x0, I_0; \ep\right)\right)
	\end{equation*} 
	is well defined and of class $C^{k-2}(\Rpi\times[a,b])$. Moreover, $\mathcal{F}(\cdot, \cdot, \ep)$ is conservative and twist. 
\end{prop}
\begin{proof}
	Let us consider $[a,b]\subset\mathcal I\backslash\mathcal{\bar I}$, and, with reference to  (\ref{mappa xi i}), define
	\begin{equation*}
		K=\left\{(\x0, \x0+\bar\theta(I_0))\text{ }|\text{ }\x0\in\Rpi, I_0\in[a,b]\right\}: 
	\end{equation*}
in view of Propositions \ref{buona definizione gen unp} and \ref{prop mappa unp}, the generating function $S(\x0, \x1; 0)$ is well defined and infinitely many differentiable in $K$, and the same holds for $\mathcal F(\x0, I_0; 0)$ in $\Rpi\times[a,b]$. Moreover, $K$ is a compact subset of the torus $\Rpi\times\Rpi$. In particular, one has that the quantity
\begin{equation}
	\partial^2_bS_E(\x0, \xt(\x0, \x1; 0); 0)+\partial^2_aS_I(\xt(\x0, \x1; 0), \x1; 0), 
\end{equation}
with $\xt(\x0, \x1; 0)$ such that $\partial_bS_E(\x0, \xt(\x0, \x1; 0); 0)+\partial_aS_I(\xt(\x0, \x1; 0), \x1; 0)=0$, 
is different from $0$ and has always the same sign for $(\x0, \x1)\in K$. \\
Let us now fix $(\bar{\xi}_0, \bar{\xi}_1)\in K$: by the implicit function theorem, there are two neighborhoods $A_{\bar{\xi}_0}, A_{\bar{\xi}_1}$ respectively of $\bar{\xi}_0$ and $\bar{\xi}_1$, a quantity $\bar\ep(\bar{\xi}_0, \bar{\xi}_1)>0$ and a unique function  $\xt(\x0, \x1; \ep)$, defined in $A_{\bar{\xi}_0}\times I
A_{\bar{\xi}_1}\times[-\bar\ep(\bar{\xi}_0, \bar{\xi}_1), \bar\ep(\bar{\xi}_0, \bar{\xi}_1)]$, such that the refraction law 
\begin{equation*}
	\partial_bS_E(\x0, \xt(\x0, \x1; \ep);\ep)+\partial_aS_I(\xt(\x0, \x1; \ep), \x1; \ep)=0
\end{equation*}
	holds also in the perturbed case. Moreover, the function $\xt$ is of class $C^{k-1}$ in all its variables. As a consequence, the generating function $S(\x0, \x1; \ep)$ is well defined in $A_{\bar{\xi}_0}\times A_{\bar{\xi}_1}\times[-\bar\ep(\bar{\xi}_0, \bar{\xi}_1), \bar\ep(\bar{\xi}_0, \bar{\xi}_1)]$. Varying $(\bar{\xi}_0, \bar{\xi}_1)\in K $, the family 
	\begin{equation*}
		\left\{A_{\bar{\xi}_0}\times A_{\bar{\xi}_1}\text{ }|\text{ }(\bar{\xi}_0, \bar{\xi}_1)\in K\right\}
	\end{equation*}
is a covering of $K$ such that, if $(A_{\bar{\xi}_0}\times A_{\bar{\xi}_1})\cap(A_{\bar{\xi}_0'}\times A_{\bar{\xi}_1'})\neq\emptyset$, then $\xt(\x0, \x1; \ep)$ coincide in the intersection. Sice $K$ is compact, there exists a finite sequence $(\bar{\xi}_0^{(i)}, \bar{\xi}_1^{(i)})_{i=1}^N$ such that 
\begin{equation*}
	K\subset\bigcup_{i=1}^N A_{\bar{\xi}_0^{(i)}}\times A_{\bar{\xi}_1^{(i)}}.  
\end{equation*} 
Setting $\bar{\ep}'=\min_{i=1, \dots, N}\bar{\ep}(\bar{\xi}_0^{(i)}, \bar{\xi}_1^{(i)})$, one has that for every $(\x0,\x1)\in K $ and every $\ep\in \R$ such that $|\ep|<\bar{\ep}'$, the perturbed generating function $S(\x0, \x1; \ep)$ is well defined and of class $C^{k-1}$. In such set one can define the canonical actions 
\begin{equation*}
	I_0(\x0, \x1; \ep)=-\partial_{\x0}S(\x0, \x1; \ep), \quad I_1(\x0, \x1; \ep)=\partial_{\x1}S(\x0, \x1; \ep), 
\end{equation*} 
and, by the definition of $K$, one has that for every $(\x0, \x1)\in K$, $I_0(\x0, \x1; 0)\in [a, b]$. Fixing $\bar{\xi}_0\in\Rpi$ and $\bar{I}_0\in [a,b]$, set $\bar{\xi}_1=\x1(\bar{\xi}_0, \bar{I}_0; 0)$: from the proof of Proposition \ref{prop mappa unp}, one has that $\partial_{\x1}\left(\bar{I}_0+\partial_{\x0}S(\bar{\xi}_0, \bar{\xi}_1; 0)\right)\neq 0$, and then, varying $(\bar{\xi}_0, \bar{I}_0)$ in the compact rectangle $[0,2\pi]\times[a,b]$, one can apply the same reasoning used before to find $0<\bar\ep<\bar\ep'$ such that for every $\x0\in[0,2\pi]$, $I_0\in[a,b]$ and $|\ep|<\bar\ep$ the function $\x1(\x0, I_0; \ep)$  is well defined and of class $C^{k-2}$. Extending $\x1(\x0, I_0; \ep)$ by periodicity for $\x0\in \Rpi$, one has then that, for every $\ep>0$, $|\ep|<\bar\ep$, the perturbed first return map 
\begin{equation*}
	\mathcal F(\x0, I_0; \ep)=\left(\x1(\x0, I_0; \ep), I_1\left(\x0, I_0; \ep\right)\right), 
\end{equation*}
where $I_1(\x0, I_0; \ep)=I_1(\x0, \x1(\x0, I_0; \ep); \ep)$, satisfies the claim in terms of good definition and regularity. The conservativity is a straightforward consequence of the existence of the perturbed generating function, while the  twist property depends on the existence of $\x1(\x0, \x1; \ep)$, since 
\begin{equation*}
	\frac{\partial \x1}{\partial I_0}=\left(\frac{\partial I_0}{\partial \x1}\right)^{-1}=-\frac{1}{\partial_{\x0\x1}S(\x0,\x1; \ep)}. 
\end{equation*}
\end{proof}
\begin{rem}
	Proposition \ref{prop mappa perturbata} remains valid if we ask weaker regularity hypotheses on $f(\xi, \ep)$. In particular, if $f$ is of class $C^k$ in $\xi$ and is continous, along with all its $k$ $\xi-$derivatives, in $\ep$, one can find $\bar\ep>0$ such that, for $|\ep|<\bar\ep$, the map $\mathcal F(\x0, I_0; \ep)$ is of class $C^{k-2}$ in $\Rpi\times[a,b]$ and continous in $\ep$, and the same holds for all its $k-2$ derivatives. 
\end{rem}

The map $\mathcal F(\x0 ,I_0;\ep )$, whose existence under suitable conditions and for subsets of $\Rpi\times\mathcal I$ is ensured by Proposition \ref{prop mappa perturbata}, can be expressed in the form 
\begin{equation}\label{pert func}
	\mathcal F(\x0, I_0; \ep)=
	\begin{cases}
		\x1=\x0+\bar\theta(I_0)+F(\x0, I_0; \ep)\\
		I_1=I_0+G(\x0, I_0; \ep)
	\end{cases}
\end{equation}
where $F$ and $G$ are of class $C^{k-2}$ in all the variables and 
\begin{equation*}
	\|F\|_{C^{k+2}}\xrightarrow{\ep\to0}0, \quad \|G\|_{C^{k-2}}\xrightarrow{\ep\to0}0. 
\end{equation*}

We can now prove the existence of particular orbits with prescribed rotation number for $\F_\ep$. We will make use of \textit{KAM Theorem} in the finitely differentiable version of Moser (cfr \cite{moser1962invariant}); before stating the Theorem, let us now give some preliminary definitions.

\begin{defns}\label{defKAM}
Let $s	\geq 1$ and  $f(\xi, I)$ of class $C^s$ in $\Rpi\times[a,b]$. The $s$-th derivative norm of $f$ is given by 
\begin{equation*}
	|f|_s=\sup\Bigg|\left(\frac{\partial}{\partial I}\right)^{m_1}\left(\frac{\partial}{\partial \xi}\right)^{m_2}f(\xi, I)\Bigg|, \quad m_1+m_2\leq s. 
\end{equation*}
Let us now consider $\mathcal F(\x0, I_0)=\left(\x1(\x0, I_0), I_1(\x0, I_0)\right)$ a given map on the annulus $\Rpi\times[a,b]$. We say that $\mathcal F$ has the \textit{intersection property}  if for any closed curve $\alpha$ near the circle, that is, of the form
\begin{equation*}
	\alpha(\x1)=\left(\x1, f(\x1)\right)
\end{equation*} 
	where $f$ is $2\pi-$periodic and with $f'$ small, one has 
	\begin{equation*}
		supp(\alpha)\cap supp\left(\mathcal F(\alpha)\right)\neq\emptyset. 
	\end{equation*}
If $\mathcal F$ is area-preserving, the intersection property is straightforwardly verified (see \cite{birkhoff1927dynamical}). \\
Finally, given $\sigma>0$, define
\begin{equation*}
	\mathcal D(\sigma)=\left\{\omega\in\R\text{ }|\text{ }\forall n, m\in\mathbb{Z}, \text{ }n>0, \text{ }|n \omega-m2\pi|\geq\sigma n^{-3/2}\right\}
\end{equation*}
the set of Diophantine numbers with respect to the constant $\sigma/2\pi$ and the exponent $5/2$. According to \cite{gole2001symplectic}, given 
\begin{equation*}
	\mathcal D=\bigcup_{\sigma>0}\mathcal D(\sigma), 
\end{equation*}
one has that $\mathcal D$ is dense in $[0,1]$, and by extension in every closed interval of $\R$. As a consequence, for every $[c,d]\subset\R$ and every $\rho\in(c,d)\cap\mathcal D$ there exists $\sigma_\rho>0$ such that 
\begin{equation*}
	\forall\sigma<\sigma_\rho\quad \rho\in\left(c+\sigma, d-\sigma\right)\cap D(\sigma). 
\end{equation*}
Let us point out that, although we consider $\mathcal D$ in accordance to the original statement by Moser in \cite{moser1962invariant}, in all the following results such set can be replaced by the set of all Diophantine numbers of exponent $\tau>2$, that is, 
\begin{equation*}
	\tilde{\mathcal{D}}=\bigcup_{\substack{\sigma>0\\ \tau>2}}\left\{\omega\in\R\text{ }|\text{ }\forall n, m\in\mathbb{Z}, \text{ }n>0, \text{ }\Big|\frac{\omega}{2\pi} -\frac{m}{n}\bigg|\geq\frac{\sigma}{ n^{\tau}}\right\}
\end{equation*}   
\end{defns}
\begin{thm}(KAM Theorem, \cite{moser1962invariant})\label{teorema KAM}
	Let $a,b\in\R$ such that $0<a<b$ and $b-a\geq 1$, and let 
	\begin{equation*}
		\mathcal F_0 (\x0, I_0)=\begin{cases}
			\x1=\x0+\bar\theta(I_0)\\
			I_1=I_0
		\end{cases}
	\end{equation*}
be a map on the annulus $\Rpi\times[a,b]$; suppose that there is $c_0\geq1$ such that 
\begin{equation*}
	c_0^{-1}\leq \frac{\partial \bar\theta}{\partial I_0}(I_0)\leq c_0. 
\end{equation*}
Moreover, let 
\begin{equation*}
	\mathcal{F}(\x0, I_0)=\begin{cases}
		\x1=\x0+\bar\theta(I_0)+F(\x0, I_0)\\
		I_1=I_0+G(\x0, I_0)
	\end{cases}
\end{equation*}
a perturbation of $\mathcal F_0$ that satisfies the intersection property. \\
Fixed $\sigma>0$ and  $s\geq1$, there are $\delta_0=\delta_0(c_0, \sigma, s)>0$ and an integer $l=l(s)>0$ such  that, if
\begin{enumerate}
	\item $|F|_0+|G|_0<\delta_0$, 
	\item $F$ and $G$ are of class $C^l(\Rpi\times[a,b])$ and $|\bar\theta|_l+|F|_l+|G|_l<c_0$, 
\end{enumerate}
then $\mathcal F$ admits a closed invariant curve 
	\begin{equation}\label{KAM 1}
		\begin{cases}
			\xi=u+p(u)\\
			I=\bar{I}+q(u), 
		\end{cases}
	\end{equation}
with $\bar I\in[a,b]$, which induces a mapping 
\begin{equation}\label{KAM 3}
	u_1=u_0+\bar\theta(\bar I)
\end{equation} 
and such that $p$ and $q$ are $2\pi-$periodic functions in the parameter $u$ with $s$ continous derivatives and 
\begin{equation}\label{localizzazione}
	|p|_s+|q|_s<\sigma. 
\end{equation}
Moreover, for every $\omega\in\left(\bar\theta(a)+\sigma, 	\bar\theta(b)-\sigma \right)\cap \mathcal D(\sigma)$ there exists an invariant curve of the form (\ref{KAM 1}) with rotation number $\bar\theta(\bar I)=\omega$. 
\end{thm}

\begin{rem}
	The rotation number of the invariant curve (\ref{KAM 1}) can be derived from the mapping (\ref{KAM 3}) as follows: let us take $u_0\in\R$ and consider the sequence $\{u_n\}_{n\in\mathbb{N}}$ produced by (\ref{KAM 3}), which is trivially given by $u_n=u_0+n \bar\theta(\bar I)$. The orbit of $\mathcal F$ generated by (\ref{KAM 3}) and lying in the invariant curve (\ref{KAM 1}) is then $\{\left(\xi_n, I_n\right)\}_{n\in\mathbb{N}}$, with
	\begin{equation}
		\begin{cases}
			\xi_n=u_0+n \bar\theta(\bar I)+p\left(u_0+n\bar\theta(\bar I)\right)\\
			I_n= \bar I+q\left(u_0+n\bar\theta(\bar I)\right). 
		\end{cases}
	\end{equation}
From the definition (\ref{def rotation number}) and given that $p$ is bounded, one can easily compute the rotation number associated to the initial condition $(\x0, I_0)=(\xi(u_0), I(u_0))$ through the map $\mathcal F$ as 
\begin{equation*}
	\rho(\x0, I_0)=\lim_{n\to\infty}\frac{\xi_n-\x0}{n}=\lim_{n\to\infty}\bar\theta(\bar I)+\frac{p(u_0+n \bar\theta(\bar I))-p(u_0)}{n}=\bar\theta(\bar I). 
\end{equation*} 

\end{rem}
\begin{rem}
	Although in the original paper \cite{moser1962invariant} for $s=1$ the minimal number of continous derivatives required for the application of Theorem (\ref{teorema KAM}) is $l=333$, R\"ussman and Hermann reduced this number to $l=5$ and then to $l>3$ (see \cite{russmann1970kleine} and \cite{herman1983courbes}). For this reason, and in view of Proposition \ref{prop mappa perturbata}, we require the normal perturbation $f(\xi, \ep)$ to be of class $C^{k}\left(\Rpi\times[-\bar\ep, \bar\ep]\right)$, with $k>5$: as a consequence, $\mathcal F(\x0, I_0; \ep)\in C^{k'}\left(\Rpi\times[a,b]\times[-\bar\ep, \bar\ep]\right)$, with $k'>3$, and the invariant curves, if existing, are of class $C^1(\Rpi)$. \\
\end{rem}

\begin{thm}\label{teorema curve invarianti}
	Let us suppose that $\bar\theta'(I_0)>0$ in $[a,b]$, and take $\rho_0, \rho_1\in\left(\bar\theta(a),\bar\theta(b)\right)\cap\mathcal D$. Then there exists $\bar\ep_{\rho_0\rho_1}$ such that for every $\ep\in\R$, $|\ep|<\bar\ep_{\rho_0\rho_1}$ the map $\mathcal F(\x0, I_0; \ep)$ defined in (\ref{pert func}) admits two closed invariant curves of class $C^1$ with rotation numbers $\rho_0$ and $\rho_1$. 
\end{thm}
\begin{proof}
	To verify the hypotheses of Theorem \ref{teorema KAM}, let us choose $C>\left(b-a\right)^{-1}$ such that $\rho_0'=\rho_0/C, \rho_1'=\rho_1/C\in \mathcal{D}$ (such $C$ exists for the density of $\mathcal D$ in $\R$), and consider the canonical change of coordinates 
	\begin{equation*}
		\begin{cases}
		\xi'=\frac{\xi}{C}, \quad I'=C\text{ } I. 
		\end{cases}
	\end{equation*}
	Expressing $\mathcal F_0$ and $\mathcal F_\ep$ in the new variables, one obtains the rescaled problem 
\begin{equation*}
	\tilde{\mathcal{F}}_0(\x0', I'_0)=
	\begin{cases}
		\x1'=\x0'+\Theta(I_0')\\
		I_1'=I_0'
	\end{cases}
\quad
\tilde{\mathcal{F}}(\x0', I_0'; \ep)=
\begin{cases}
	\x1'=\x0'+\Theta(I_0')+\tilde{F}(\x0', I_0'; \ep)\\
	I_1'=I_0'+\tilde{G}(\x0', I_0'; \ep), 
\end{cases}
\end{equation*}
where $I'=C I\in[a', b']=C[a,b]$, $b'-a'>1$, and 
\begin{equation*}
	\begin{split}
	\Theta(I'_0)=\frac{\bar\theta\left(\frac{I'_0}{C}\right)}{C}&=\frac{\bar\theta(I_0)}{C}, \quad\tilde F(\x0', I_0'; \ep)=\frac{1}{C}F\left(C\x0', \frac{I_0'}{C}; \ep\right)=\frac{1}{C}F(\x0, I_0; \ep),\\ &\tilde G(\x0', I_0'; \ep)=C\text{ }G\left(C\x0', \frac{I_0'}{C}; \ep\right)=C\text{ }G(\x0, I_0; \ep)
	\end{split}
\end{equation*}
are defined for $\x0'\in\Rpi$, $I'_0\in[a', b']$ and $|\ep|<\bar\ep$ (see Proposition \ref{prop mappa perturbata}). Note that the monotonicity and convergence properties 
	\begin{equation}\label{KAM 2}
		\partial_{I_0'}\Theta\left(I_0'\right)>0, \text{ }\|\tilde F\|_{C^{k-2}}\xrightarrow{\ep\to0}0, \text{ }\|\tilde G\|_{C^{k-2}}\xrightarrow{\ep\to0}0
	\end{equation}
hold also for the rescaled map, as well as the conservativity. Moreover, as $\Theta(a')=\bar\theta(a)/C$ and $\Theta(b')=\bar\theta(b)/C$, one finds $\sigma>0$ such that 
$\rho_0', \rho_1'\in(\Theta(a')+\sigma, \Theta(b')-\sigma)\cap\mathcal D(\sigma)$. \\
Since $\tilde{\mathcal F}$ is conservative, it satisfies the intersection property, and, given that $\Theta\in C^1([a', b'])$, there is $c_0>1$ such that 
\begin{equation*}
	\forall I_0'\in[a', b']\quad c_0^{-1}\leq \Theta'(I'_0)\leq c_0. 
\end{equation*}
Fixed $s=1$, let us consider $l=l(s)$ as in Theorem \ref{teorema KAM}, and, eventually taking a higher $c_0$,  suppose $c_0>|\Theta|_l$. By Theorem \ref{teorema KAM}, there exists $\delta_0=\delta_0(c_0, \sigma, s)>0$ such that, if $(1)$ and $(2)$ hold for $\tilde{\mathcal F}$, then the existence of the two invariant orbits for the rescaled problem is ensured. From (\ref{KAM 2}), one can choose  $0<\bar\ep_{\rho_0\rho_1}<\bar\ep$, such that for every $\ep\in[-\bar\ep_{\rho_0\rho_1}, \bar\ep_{\rho_0\rho_1}]$
\begin{equation*}
	|\tilde F|_{0}+|\tilde G|_0<\delta_0\quad\text{and}\quad |\tilde F|_l+|\tilde G|_l<c_0-|\Theta|_l, 
\end{equation*}
then the hypotheses of Theorem \ref{teorema KAM} hold and the invariant curves obtained for $\tilde{\mathcal F}$ can be reparametrized to be invariant curves for $\mathcal F$. In particlar, such curves have rotation number $\rho_0$ and $\rho_1$: for example, let us consider the invariant curve for the rescaled problem with rotation number $\rho_0'$, which, in view of Theorem \ref{teorema KAM}, can be expressed as 
\begin{equation*}
	\begin{cases}
		\x0'=u'+\tilde p(u')\\
		I_0'=\bar I'+\tilde q(u')
	\end{cases}
\text{ with mapping }u'_1=u'_0+\Theta(\bar I')=u_0'+\rho_0'. 
\end{equation*}
Returnig to the original coordinates and setting $u=C\text{ }u'$, one gets the rescaled invariant curve
\begin{equation*}
		\begin{cases}
		\x0=u+\ p(u)\\
		I_0=\bar I+ q(u)
	\end{cases}
	\text{ with mapping }u_1=u_0+\bar\theta(\bar I)=u_0+\rho_0,  
\end{equation*}
with $p(u)=C\tilde p\left(u/C\right)$ and $q(u)=C^{-1}\tilde q\left(u/C\right)$. 
\end{proof}
In the phase space $(\xi, I)$, the curves obtained in Theorem \ref{teorema curve invarianti} can be identified as the graphs of functions of the form $I_\rho(\xi; \ep)\in C^1(\Rpi)$: fixing $\ep\in(-\bar\ep_{\rho_0\rho_1}, \bar\ep_{\rho_0\rho_1})$, let us consider for example the closed invariant curve of $\mathcal F$ of rotation number $\rho_0$, which can be expressed, according to (\ref{KAM 1}) and (\ref{KAM 3}), as
\begin{equation}\label{KAM 4}
	\begin{cases}
		\xi_{\rho_0}(u; \ep)=u+p(u; \ep)\\
		I_{\rho_0}(u; \ep)=\bar I+q(u; \ep)
	\end{cases}
\quad
\text{with}\quad \bar\theta(\bar I)=\rho_0. 
\end{equation}
From the boundedness of $p$ asserted in (\ref{localizzazione}), if $\sigma$ is small enough (e.g. $\sigma<1$) the quantity
\begin{equation}
	\partial_u \xi(u; \ep)=1+\partial_up(u; \ep)
\end{equation}
is always positive: one can then invert the first equation in (\ref{KAM 4}) obtaining $u(\xi)$, which is differentiable. As a consequence, one can parametrize the curve (\ref{KAM 4}) as the graph of the $C^1$ function 
\begin{equation}
	I_{\rho_0}: \Rpi\to \R, \quad I_{\rho_0}(\xi; \ep)=I_{\rho_0}(u(\xi); \ep). 
\end{equation}
\begin{rem}
	Taking $\sigma$ sufficiently small and a suitable $\bar\ep_{\rho_0\rho_1}$, one can find invariant curves of $\mathcal F$ which are arbitrarily close to the unperturbed orbits $\Rpi\times\{\bar{I}\}$ in the plane $(\xi, I)$. Then, as $\ep\to 0$, the functions $I_\rho(\xi; \ep)$ which define the invariant curves in the perturbed phase space tend in norm $C^1(\Rpi)$ to the constant functions $\bar I_\rho$ with $\bar\theta(\bar I_\rho)=\rho$. \\
	\textcolor{black}{Moreover, Theorem \ref{teorema KAM} can be extended to negative twist maps, leading to the existence result for invariant curves as stated in Theorem \ref{teorema curve invarianti intro}}
\end{rem}
\begin{figure}
	\centering
	\includegraphics[width=0.7\linewidth]{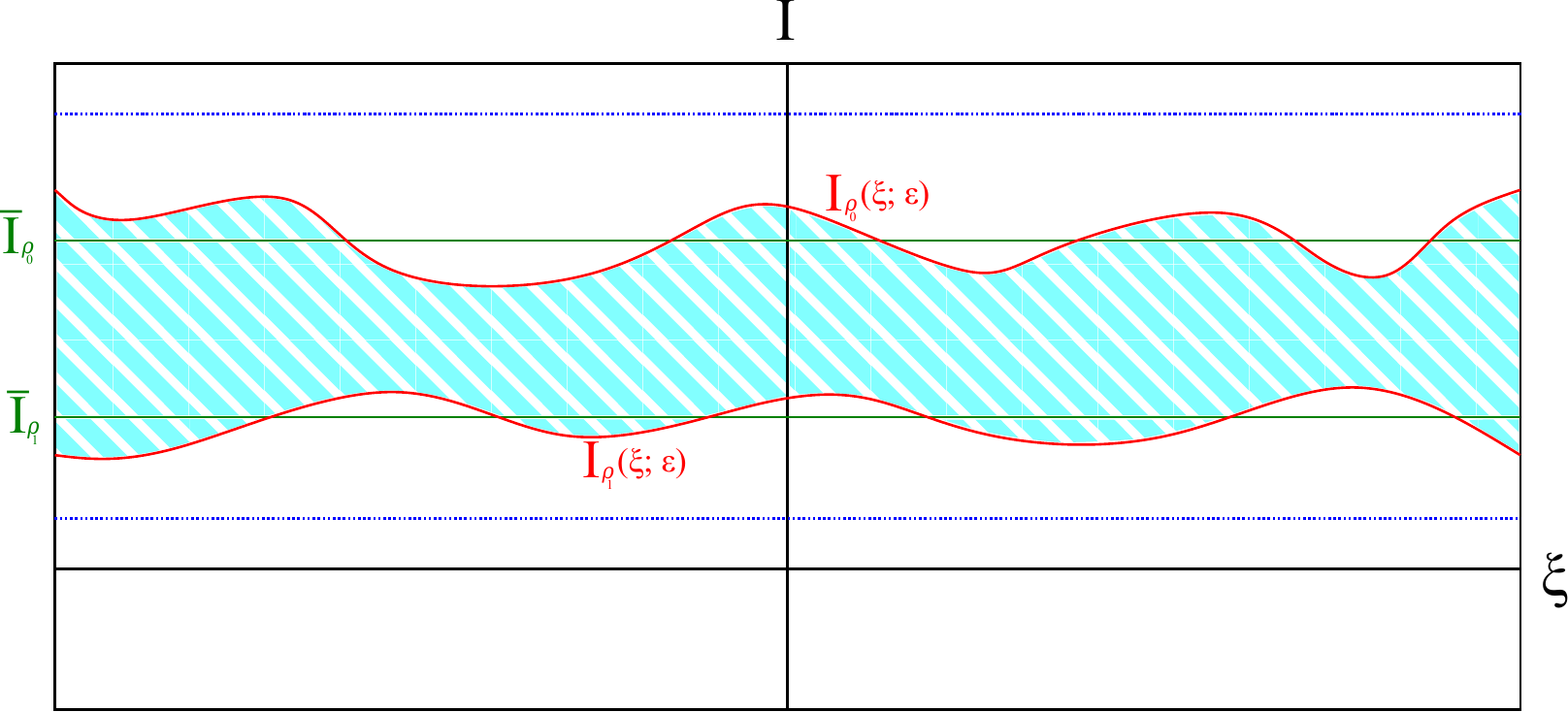}
	\caption{Sketch of the perturbed dynamics in the region described by Proposition \ref{prop mappa perturbata} and Theorem \ref{teorema curve invarianti} in the phase plane $(\xi, I)$. Red: the invariant curves of Diophantine rotation numbers $\rho_0$ and $\rho_1$, which are deformations of the unperturbed invariant straight lines $I=\bar I_{\rho_0}$, $I=\bar I_{\rho_1}$ (green) such that $\bar\theta\left(\bar I_{\rho_0}\right)=\rho_0$ and $\bar\theta\left(\bar I_{\rho_1}\right)=\rho_1$. In the striped region the map $\mathcal F(\x0, I_0; \ep)$ is area-preserving and twist. The blue dashed lines denote two singular action values for the unperturbed dynamics (i.e. $I\in\bar{\mathcal{I}}$). }
	\label{fig:orbite-perturbate}
\end{figure}

As sketched in Figure \ref{fig:orbite-perturbate}, the presence of two invariant curves with irrational rotation number leads to a  \textit{confinement} in the dynamics of $\mathcal F(\x0, I_0; \ep)$, where, in view of Proposition (\ref{prop mappa perturbata}), the map is conservative and twist. More precisely, the set 
\begin{equation}\label{def regione invariante}
	\mathcal A=\left\{(\xi, I)\in\Rpi\times\R\text{ }|\text{ }I_{\rho_0}(\xi; \ep)\leq I\leq I_{\rho_1}(\xi; \ep)\right\}
\end{equation}
is invariant under $\mathcal F_\ep$, as well as its boundaries.\\
In the  unperturbed dynamics, the existence of periodic orbits of any rotation number in a suitable interval is the simple consequence of the continuity of the total shift $\bar\theta(I)$; when $\ep\neq0$, one can not take advantage of the explicit formulation of the perturbed map, then this strategy is no longer suitable. Nevertheless, the broad properties of the map, such as its conservativity and the existence of the invariant curves ensure Theorem \ref{teorema curve invarianti}, enable the use of more sophisticated topological results, where the existence of orbits with prescribed rotation number is ensured under more general assumptions: this is the case of the Poincar\'e-Birkhoff theorem, here presented in the version of \cite{gole2001symplectic}. \\

\begin{thm}[Poincar\'e-Birkhoff]\label{thm Poincaré-B}
	Let $\mathcal F$ an area preserving map on the annulus  $\Rpi\times[c,d]$ \textcolor{black}{which preserves the boundaries} and $\mathcal{F_{\R}}$ its lift on $\R\times[c,d]$. Suppose that $\mathcal F_\R$ satisfies the boundary twist condition, that is, the restrictions of $\mathcal F_\R$ to each boundary component $u_-=\R\times\{c\}$ and $u_+=\R\times\{d\}$ have rotation numbers $\rho_\pm$ with $\rho_-<\rho_+$ \textcolor{black}{(the case $\rho_+<\rho_-$ is analogous). } If $2\pi\frac{m}{n}\in[\rho_-, \rho_+]$ and $m, n$ are coprime, then $\mathcal F$ has at least two $(m, n)-$orbits.   
\end{thm}
\textcolor{black}{
\begin{rem}
	Theorem \ref{thm Poincaré-B} can be extended to conservative maps which preserve invariant strips in $\Rpi\times\R$ whose boundaries are fixed by $\mathcal F$ and are graphs of $C^1$ functions over the $\xi-$axis. Let us take the set $\mathcal A$ defined in \ref{def regione invariante} and consider $\Phi_0(\xi; \ep), \Phi_1(\xi; \ep), \bar I_0, \bar I_1$ such that 
	\begin{equation*}
		\partial_\xi\Phi_0(\xi;\ep)=I_{\rho_0}(\xi; \ep), \quad \partial_\xi\Phi_1(\xi; \ep)=I_{\rho_1}(\xi; \ep), \quad \bar\theta(\bar I_{\rho_0})=\rho_0, \quad \bar\theta(\bar I_{\rho_1})=\rho_1, 
	\end{equation*} 
choosing $\Phi_1$ and $\Phi_0$ such that $\Phi_1(\xi; 0)=\bar I_1\xi$ and $\Phi_0(\xi; 0)=\bar I_0\xi$. 
For $\ep$ sufficiently small, consider the quantity
\begin{equation*}
	A(\ep)=\int_0^{2\pi}I_{\rho_1}(\xi; \ep)-I_{\rho_0}(\xi; \ep)d\xi=\Phi_1(2\pi; \ep)-\Phi_0(2\pi; \ep)-\left(\Phi_1(0; \ep)-\Phi_0(0; \ep)\right)>0,
\end{equation*}
	and, noted that $A(0)=\left(\bar I_1-\bar I_0\right)2\pi$, define the change of coordinates
	\begin{equation*}
		\Psi(\xi, I; \ep)=
 		\begin{cases}
			\xi'=\frac{2\pi}{A(\ep)}\left(\Phi_1(\xi; \ep)-\Phi_0(\xi; \ep)\right)\\
			I'=\frac{A(\ep)}{2\pi}\left(\frac{I-I_{\rho_0}(\xi; \ep)}{I_{\rho_1}(\xi; \ep)-I_{\rho_0}(\xi; \ep)}+\frac{\bar I_0}{\bar I_1-\bar I_0}\right).
		\end{cases}
	\end{equation*}  
From direct computations, one has that:
\begin{itemize}
	\item[(i)] $\Psi(\xi, I; \ep)$ is $C^1$ in all its variables and for every fixed $\ep\in\R$ $\det\left(D_{(\xi; I)}\Psi(\xi, I; \ep)\right)=1$: hence, $\Psi$ defines a canonical change of variables; 
	\item[(ii)] for $\ep=0,$ $\Psi(\xi, I; 0)=Id$;
	\item[(iii)]  $\Psi$ maps the horizontal boundaries of $\mathcal A$, that is, $\{(\xi, I_{\rho_0}(\xi; \ep))\text{ }|\text{ }\xi\in\Rpi\}$ and  \\$\{(\xi, I_{\rho_1}(\xi; \ep))\text{ }|\text{ }\xi\in\Rpi\}$ respectively into the straght lines 
\begin{equation*}
	I=I_0'=\frac{A(\ep)}{2\pi}\frac{\bar I_0}{\bar I_1-\bar I_0}>0 \quad \text{and}\quad I=I_1'=\frac{A(\ep)}{2\pi}\frac{\bar I_1}{\bar I_1-\bar I_0}>I_0'; 
\end{equation*} 
	\item[(iv)] for every $\xi\in\Rpi$ and every fixed $\ep$, one has $\Phi_1(\xi+2\pi; \ep)-\Phi_0(\xi+2\pi; \ep)=\Phi_1(\xi; \ep)-\Phi_0(\xi; \ep)+A(\ep)$, and then 
	\begin{equation*}
		\begin{aligned}
		\xi'(\xi+2\pi)&=\frac{2\pi}{A(\ep)}\left(\Phi_1(\xi+2\pi; \ep)-\Phi_0(\xi+2\pi; \ep)\right)=\\
	&=\frac{2\pi}{A(\ep)}\left(\Phi_1(\xi; \ep)-\Phi_0(\xi; \ep)+A(\ep)\right)=\xi'(\xi)+2\pi; 
		\end{aligned}
	\end{equation*}
\item[(v)] $\xi'$ is strictly increasing in $\xi$, while $I'$ is $2\pi-$periodic in $\xi$. 
\end{itemize} 
Globally, $\Psi$ maps the $\mathcal F_\ep$invariant set $\mathcal A$ into the straight line $\mathcal B=\Rpi\times[I_0', I'_1]$ preserving the orientation and the boundaries. One can then consider the map $\bar{\mathcal F}_\ep: \mathcal B\to\mathcal B$ such that $\bar{\mathcal F}_\ep\circ \Psi=\Psi\circ\mathcal F_\ep$, namely, such that $\bar{\mathcal F}_\ep=\Psi\circ{\mathcal F}_\ep\circ \Psi^{-1}$. It can be proved that $\bar{\mathcal F}_\ep$ preserves the rotation number of the corresponding orbits of $\mathcal F$: for a $(m,n)-$periodic orbit, it is a simple consequence of $(iv)$, as, taken $\{(\xi_k, I_k)\}_{k\in \mathbb N}$ $(m,n)-$periodic for $\mathcal F_\ep$ and defined for every $k\in\mathbb N$ $(\xi'_k, I_k')=\Psi(\xi_k, I_k)$, one has 
\begin{equation*}
	(\xi_{k+n}', I'_{k+n})=\Psi(\xi_{k+n}, I_{k+n})=\Psi(\xi_{k}+2\pi m, I_k)=(\xi'_k+2\pi m, I'_k). 
\end{equation*}
Let us now take a $\mathcal F_\ep-$orbit with rotation number $\rho\in\R$ parametrized, according to Mather's definition in \cite{mather58084existence}, by $(\xi_k, I_k)=(\psi_1(t_k), \psi_2(t_k))$ such that 
\begin{gather*}
	t_{k+1}=t_k+\rho, \quad F(\xi_k, I_k)=(\xi_{k+1}, I_{k+1})=\left(\psi_1(t_k+\rho), \psi_2(t_k+\rho)\right), \\
	\left(\psi_1(t+2\pi), \psi_2(t+2\pi)\right)=\left(\psi_1(t)+2\pi, \psi_2(t)\right),
\end{gather*}
with $\psi_1:\R\to\R$ a weakly order preserving map (not necessarily continous). Now, setting $(\xi'_k, I'_k)=\Psi(\xi_k, I_k)=\Psi(\psi_1(t_k), \psi_2(t_k))$, one has
\begin{gather*}
	\bar{\mathcal F}_\ep(\xi'_k, I'_k)=\Psi\circ F(\xi_k, I_k)=\Psi(\xi_{k+1}, I_{k+1})=(\xi_{k+1}', I_{k+1}'), 
\end{gather*} 
and, defined $(\tilde\psi_1(t), \tilde\psi_2(t))=\Psi(\psi_1(t), \psi_2(t))$, 
\begin{equation*}
	(\tilde\psi_1(t+2\pi), \tilde\psi_2(t+2\pi))=(\tilde\psi_1(t)+2\pi, \tilde\psi_2(t)),
\end{equation*}
leading to the conclusion that the $\bar{\mathcal F}_\ep-$orbit $\{(\xi_k',I_k')\}_{k\in\mathbb N}$ has period $\rho$.
 Note that for $\ep$ sufficiently small $\bar{\mathcal F}_\ep$ satisfied the hypotheses of Theorem \ref{thm Poincaré-B}, as for $\ep=0$ the identity map is trivially twist. As the preservation of the rotation number holds also for $\Psi^{-1}$, given a twist map on invariant sets of the type $\mathcal A$ which preserves the horizontal boundaries one can pass to the strip $\mathcal B$ and use Theorem \ref{thm Poincaré-B} to prove the existence of $(m,n)-$periodic orbits for $\bar{\mathcal F}_\ep$; returning then to the map $\mathcal F_\ep$, this translates to the existence of $(m,n)-$periodic orbits for the original map. 
\end{rem}
}
Making use of Theorem \ref{thm Poincaré-B}, one can prove, under suitable conditions on the perturbation, the existence of periodic orbits for the dynamics induced by the map $\mathcal F(\x0, I_0; \ep)$ with $\ep\neq0$. We recall that in \S \ref{sec: cerchio} we denoted with $\mathcal I\backslash\bar{\mathcal I} $ the set of well defnition of the unperturbed map $\mathcal F(\x0, I_0; 0)$ and we proved that it is the finite union of open intervals in $\R$. In particular, the set of the singular points $\bar{\mathcal I}$ is composed by the critical points of the $C^1$ function $\bar\theta(I)$ (see Proposition \ref{prop mappa unp}), and one can set 
\begin{equation*}
	\mathcal I\backslash\bar{\mathcal I}=\bigcup_{i=1}^N A_i
\end{equation*} 
with $N>0$ (possibly $N=1$) and $A_i$ open intervals in $\mathcal I$. In the following, to ensure the good definition of the perturbed map in a compact set, a finit union of closed intervals in $\mathcal I\backslash\bar{\mathcal I}$ will be fixed, : in particular, we fix $a_i, b_i\in\mathcal I$ such that $\forall i\in\{1, \dots, N\}$
\begin{equation}\label{def ai, bi}
	\begin{aligned}
	 &[a_i, b_i]\subset A_i \text{ and, if } \bar\theta_-=\min_i\left\{\bar\theta(a_i), \bar\theta(b_i)\right\}, \text{ }\bar\theta_+=\max_i\left\{\bar\theta(a_i), \bar\theta(b_i)\right\},\\
	& \bar\theta_-^{(i)}=\min\{\bar\theta(a_i), \bar\theta(b_i)\}, \text{ and  }\bar\theta_+^{(i)}=\max\{\bar\theta(a_i), \bar\theta(b_i)\}, \text{ one has }\\
	& \bigcup_{i=1}^N\left[\bar\theta_-^{(i)}, \bar\theta_+^{(i)}\right]=[\bar\theta_-, \bar\theta_+].  	
	\end{aligned}
\end{equation}
Note that, by the continuity of $\bar \theta, $ such sets $\{a_i\}_{i=1}^N,$ $\{b_i\}_{i=1}^N$ exist. 
\begin{prop}\label{prop orbite pert ep uniforme}
	Let $a_i, b_i\in\mathcal I$ as in (\ref{def ai, bi}), and fix $\rho_{\pm}^{(i)}\in\mathcal D$ such that for every $i=1, \dots, N$ one has
	$\bar\theta^{(i)}_-<\rho_-^{(i)}<\rho_+^{(i)}<\bar\theta_+^{(i)}$. 
Then there exists $\bar\ep>0$ such that for every $\ep\in\R$, $|\ep|<\bar\ep$, and for every $m, n\in\mathbb{Z}$ coprime, $n>0$, with  $2\pi\frac{m}{n}\in(\rho_-^{(i)}, \rho_+^{(i)})$ for some $i\in\{1,\dots, N\}$, the map $\mathcal F(\x0, I_0; \ep)$ admits at least  $2k$ $(m,n)$-orbits, where $k$ is the number of the pairs $(\rho_-^{(i)}, \rho_+^{(i)})$ such that $\rho^{(i)}_-<2\pi\frac{m}{n}<\rho_+^{(i)}$. 
\end{prop}  
\begin{proof}
	According to Theorem \ref{teorema curve invarianti}, for every pair $\rho_-^{(i)}, \rho_+^{(i)}$ there is $\bar\ep^{(i)}_{\rho_\pm}$ such that for every $|\ep|<\bar\ep^{(i)}_{\rho_\pm}$ the map $\mathcal{F}(\x0, I_0; \ep)$ admits two orbits of rotation numbers $\rho^{(i)}_-$ and $\rho_+^{(i)}$. Moreover, the perturbed map is conservative and twist between these two orbits. Setting $\bar\ep=\min_{i\in\{1,\dots, N\}}$, one has that if $\ep$ is such that $|\ep|<\bar\ep$ all the orbits of rotation numbers $\rho_\pm^{(i)}$ are preserved and the perturbed map in between is well defined and conservative. \\
	Fixing $\ep\in\R$, $|\ep|<\bar\ep$, if $m, n$ are such that $2\pi\frac{m}{n}\in\left(\rho_-^{(i)}, \rho_+^{(i)}\right)$, then by Theorem \ref{thm Poincaré-B} the perturbed map $\mathcal F(\x0,I_0; \ep)$ admits at least $2$ $(m,n)$-orbits. As this reasoning can be repeated whenever a pair $(\rho_-^{(i)}, \rho_+^{(i)})$ is such that $\rho^{(i)}_-<2\pi\frac{m}{n}<\rho_+^{(i)}$, the claim is true. 
\end{proof}
Proposition \ref{prop orbite pert ep uniforme} claims the existence of a unique treshold value of $\ep$ under which the presence of periodic orbits of prescribed rotation numbers in a certain set is guaranteed. Another slightly different approach is proposed in Proposition \ref{prop orbite pert ep dipendente da rho}, where, fixed $m, n$ such that $2\pi\frac{m}{n}$ lies in a suitable interval which does not depend on prefixed boundary rotation numbers $\rho^{(i)}_\pm$, one can find a treshold $\bar\ep_{mn}$, depending on $m,n$, such that for every $|\ep|<\bar\ep_{mn}$ the presence of the corresponding $(m,n)$-orbit is ensured.

\begin{prop}\label{prop orbite pert ep dipendente da rho}
	Given $\{a_i\}_{i=1}^N$, $\{b_i\}_{i=1}^N$, $\{\bar\theta_+^{(i)}\}_{i=1}^N, \{\bar\theta_-^{(i)}\}_{i=1}^N$ as in (\ref{def ai, bi}), let $m,n\in\mathbb Z$ coprime, $n>0$, such that $2\pi\frac{m}{n}\in\left(\bar\theta^{(i)}_-, \bar\theta_+^{(i)}\right)$ for some $i\in\left\{1,\dots,N\right\}$. Then $\exists\bar\ep_{mn}>0$ such that for every $\ep\in\R$, $|\ep|<\bar\ep_{mn}$ the map $\mathcal F(\x0, I_0; \ep)$ admits at least 2$k$ $(m,n)$-orbits, where $k$ is the number of intervals $(a_i,b_i)$ such that $2\pi\frac{m}{n}$ is between $\bar\theta_-^{(i)}$ and $\bar\theta_+^{(i)}$. 
\end{prop}
\begin{proof}
By the density of $\mathcal D$ in every bounded interval, one can find $\rho_\pm\in \mathcal D$ with $\bar\theta_-^{(i)}<\rho_-<2\pi\frac{m}{n}<\rho_+<\bar\theta_+^{(i)}$. From Theorem \ref{teorema curve invarianti}, one can find $\bar\ep_{mn}=\bar\ep_{\rho_\pm}$ such that, if $\ep$ is such that $|\ep|<\bar\ep_{mn}$, the map $\mathcal F(\x0, I_0; \ep)$ admits two orbit with rotation numbers $\rho_\pm$, and it is conservative between them. Applying again Theorem \ref{thm Poincaré-B}, the claim follows. 		
\end{proof}
 Extending the discussion beyond perodic orbits, one may search for more general class of invariant sets. KAM theory allowed us to claim the persistence of orbits with Diophantine rotation numbers within certain ranges, while Poincar\'e-Birkhoff theorem extended the existence result to periodic number with $2\pi$-rational numbers between them. The Aubry-Mather theory allows to move further, providing the existence of orbits of the perturbed map of every prescribed rotation number in suitable subsets of $\R$. 
\begin{thm}[Aubry-Mather on the compact annulus]\label{teorema Aubry-Mather}
Let $\mathcal F$ an area and orientation- preserving twist homeomorphism of the annulus $\Rpi\times[a,b]$ which preserves $\Rpi\times\{a\}$ and  $\Rpi\times\{b\}$, and define $\rho_a$ and $\rho_b$ as the rotation numbers of the two boundary components. Then for every $\rho\in[\rho_a, \rho_b]$ there exists at least an orbit for $\mathcal F$ with rotation number $\rho$. In particular: 
\begin{itemize}
	\item if $\rho=m/n\in\mathbb Q$, such orbit is periodic of period $n$; 
	\item if $\rho\notin\mathbb Q$, the orbit rotates either on a closed continous curve or on a Cantor set.  
\end{itemize}
In any case, the orbits with the same rotation number belong to a common invariant set $\Gamma_\rho$, called Mather set, which is a subset of the graph of a Lipschitz-continous function over the $\xi-$axis. 
\end{thm}
We refer to \cite{aubry1983discrete, jurgen1986recent, mather58084existence} for the definition of Mather set and for a thorough discussion on the Aubry-Mather theory. 
\begin{rem}
	As in the case of Poincar\'e-Birkhoff Theorem \ref{thm Poincaré-B}, with the same reasoning also Aubry-Mather Theorem can be extended to maps on invariant sets of the type $\mathcal A$ defined in \ref{def regione invariante}. 
\end{rem}
Making use of the same arguments used in the proofs of Propositions \ref{prop orbite pert ep uniforme} and \ref{prop orbite pert ep dipendente da rho}, one can take advantage of Theorem \ref{teorema Aubry-Mather} to state these existence results in a more general way. 
\begin{thmv}\label{thm:korbits}
	Let $\{a_i\}_{i=1}^N$, $\{b_i\}_{i=1}^N$, $\{\bar\theta_-^{(i)}\}_{i=1}^N$, $\{\bar\theta_+\}_{i=1}^N$ as in (\ref{def ai, bi}). Then: 
	\begin{itemize}
		\item letting $\rho_\pm^{(i)}\in \mathcal D $ as in Proposition \ref{prop orbite pert ep uniforme}, there exists $\bar\ep>0$ such that for every $\ep\in\R$ with $|\ep|<\bar\ep$ and for every $\rho\in\R$ such that $\rho\in[\bar\rho_-^{(i)}, \bar\rho_+^{(i)}]$ for some $i\in\{1,\dots, N\}$ the map $\mathcal F(\x0, I_0; \ep)$ admits $k$ orbits with rotation number $\rho$, with $k$ defined as in Proposition \ref{prop orbite pert ep uniforme};
		\item for every $\rho\in\R$ such that $\rho\in[\bar\theta_-^{(i)},\bar\theta_+^{(i)}]$ for some $i\in\{1,\dots, N\}$  there is $\bar\ep_\rho>0$ such that for every $\ep$ with $|\ep|<\bar\ep_\rho$ the map $\mathcal F(\x0, I_0; \ep)$ admits $k$ orbits with rotation number $\rho, $ where $k$ is defined as in Proposition \ref{prop orbite pert ep dipendente da rho}. 
	\end{itemize}
In both cases, if $\rho=2\pi\frac{m}{n}$ then for $\ep$ sufficiently small there are at least $2k$ $(m,n)$-orbits, where $k$ is defined suitably according to the cases. 
\end{thmv}
\subsection{Caustics for the perturbed case}\label{sec: caustiche perturbate}

The persistence of invariant curves with Diophantine rotation numbers ensured by the KAM theorem has important consequences for the existence of caustics in the perturbed dynamics. As a matter of fact, for such invariant tori (which are dense in the phase space) it is possible to find, although not explicitely, the inner and outer caustics also for small perturbations of the circular domain $D_0$. 
\begin{thmv}\label{teorema caustiche pert}
	Let $\x0\in[0,2\pi]$, $I_0\in\mathcal I\backslash\bar{\mathcal{I}}$ such that $\theta(I_0)\in\mathcal D $. Then there exists $\bar{\ep}>0$ such that for every $|\ep|<\bar{\ep}$ there are $\Gamma_E(\xi; \ep, \theta(I_0))$, $\Gamma_I(\xi; \ep, \theta(I_0))$ respectively outer and inner caustics related to the perturbed orbit of rotation number $\theta(I_0)$. 
\end{thmv} 

The proof of Theorem \ref{teorema caustiche pert} relies on showing that, for $\ep$ small enough, system (\ref{sistema caustica}) evaluated both for the outer and inner dynamics admits a unique solution for each $\xi\in[0,2\pi]$, which defines a regular and closed curve. To prove that, it is worthwile to derive the form of $G_{E\backslash I}(x,y; \xi)$ for a perturbed domain. 

\paragraph{\textbf{Outer dynamics}} 

Let us consider $\xi\in[0,2\pi]$, $p_0=\gamma_\ep(\xi)$ and $v_0\in\R^2$ such that $|v_0|=\sqrt{2 V_E(p_0)}$ and $\alpha=\angle{(p_0, v_0)}\in(-\pi/2, \pi/2)$. \\
To fix the notation, recall the definition of $\gamma_\ep(\xi)=(1+\ep f(\xi, \ep))e^{i\xi}=\rho(\xi; \ep)e^{i\xi}:$ 
as the perturbation of the circle is only in the normal direction, the curve's parameter $\xi$ still represents the polar angle of the point $\gamma_\ep(\xi)$. We want to find the Cartesian equation of the outer elliptic arc of initial conditions $p_0$ and $v_0$. \\
Following the same reasoning of Appendix \ref{appA} and denoting with $(x(s), y(s))$ the parametrization of such ellipse, 
its maximal and minimal distances from the origin can be then computed as 
\begin{equation}\label{a2 b2 pert}
	\begin{aligned}
	a^2=\max_{s\in[0, 2\pi/\omega]}r^2(s)=A+\frac{\E}{\om}=\frac{\E+\sqrt{(\E-\om|p_0|^2)^2+\om(p_0\cdot v_0)^2}}{\om}\\
	b^2=\min_{s\in[0, 2\pi/\omega]}r^2(s)=-A+\frac{\E}{\om}=\frac{\E-\sqrt{(\E-\om|p_0|^2)^2+\om(p_0\cdot v_0)^2}}{\om}: 
	\end{aligned}
\end{equation}
in the reference frame $R(O, x'', y'')$ whose axes coincide with the ellipse's ones the latter is then implicitely defined by the equation
\begin{equation*}
	\frac{{x''}^2}{a^2}+\frac{{y''}^2}{b^2}=1. 
\end{equation*}
Let us now search for the angle $\bar\beta$ such that the rotated ellipse 
\begin{equation*}
	\frac{(x'\cos\bar\beta+y'\sin\bar\beta)^2}{a^2}+\frac{(y'\cos\bar\beta-x'\sin\bar\beta)^2}{b^2}=1
\end{equation*}
passes from $p_0=|p_0|(1,0)$ in $R(O,x', y')$: one has to solve the equation 
\begin{equation*}
	\frac{|p_0|^2\cos^2\bar\beta}{a^2}+\frac{|p_0|^2\sin^2\bar\beta}{b^2}-1=0\Rightarrow \sin^2\bar\beta=\frac{b^2}{a^2-b^2}\left(\frac{a^2}{|p_0|^2}-1\right)\geq0. 
\end{equation*}
Denoting by $(v'_x, v'_y)$ the components of $v_0$ in $R(O, x',y')$, one has that 
\begin{equation}\label{sinbeta cosbeta pert}
	\sin\bar\beta=
	\begin{cases}
		-\frac{b}{\sqrt{a^2-b^2}}\sqrt{\frac{a^2}{|p_0|^2}-1}\quad &\text{if }v'_y<0\\
		\frac{b}{\sqrt{a^2-b^2}}\sqrt{\frac{a^2}{|p_0|^2}-1}\quad &\text{if }v'_y>0\\
	\end{cases}
\Rightarrow \cos\bar\beta=\frac{a}{\sqrt{a^2-b^2}}\sqrt{1-\frac{b^2}{|p_0|^2}}. 
\end{equation}
Returning to the original frame $R(O,x,y)$, one can then retrieve the Cartesian equation of the outer arc as 
\begin{equation}\label{caustica esterna pert}
	G_{E}(x,y;\xi,\ep)=\frac{\left(x\cos(\xi+\bar\beta)+y\sin(\xi+\bar\beta)\right)^2}{a^2}+\frac{\left((y\cos(\xi+\bar\beta)-x\sin(\xi+\bar\beta)\right)^2}{b^2}-1=0
\end{equation}
Once obtained the general expression for an ellipse of initial conditions $p_0$ and $v_0$, we shall return to the framework of our perturbed problem. Let us then consider $I_0\in\mathcal I\backslash\bar{\mathcal{I}}$ such that $\theta(I_0)$ is Diophantine: from Theorem \ref{teorema KAM} there exists $\bar{\ep}^{(1)}>0$ such that, if $|\ep|<\bar\ep^{(1)}$, we can define $I(\xi; \ep)$ invariant curve in the plane $(\xi, I)$ for the perturbed map $\mathcal\F_\ep$ such that $I(\xi; 0)\equiv I_0$ and with rotation number $\theta(I_0)$. Moreover, $I(\xi;\ep)$ is continous in $\ep$ and differentiable in $\xi$, with $\partial_\xi I(\xi; \ep)$ continous in $\ep$: as a consequence, since $\theta(I_0)\in\mathcal D$ implies $I_0\neq0$, possibly reducing $\bar\ep^{(1)}$ one can assume that $I(\xi;\ep)$ has always the same sign of $I_0$. \\
For the caustic of the orbit associated to $(\xi, I(\xi; \ep))$ to be well defined, it is necessary that the system 
\begin{equation}\label{sistema caustiche pert}
	\begin{cases}
		G_E(x, y; \xi, \ep)=0\\
		\partial_\xi G_E(x, y; \xi, \ep)=0
	\end{cases}
\end{equation}
defines implicitely a unique curve $\Gamma_E(\xi; \ep)$ for $\xi\in[0,2\pi]$, that is, that $x$ and $y$ can be expressed as functions of $(\xi, \ep)$ globally defined for $\xi\in[0,2\pi]$. As already pointed out in Section \ref{sec: caustiche imperturbate}, from the implicit function theorem the local existence of $\Gamma_E(\xi; \ep)$ is then ensured by requiring the nondegeneracy condition 
\begin{equation}\label{nondeg cond pert}
	\nabla_{(x,y)}G_E(x,y;\xi, \ep)\nparallel\nabla_{(x,y)}\partial_\xi G_E(x,y; \xi, \ep)
\end{equation}
on the solutions of (\ref{sistema caustiche pert}). 
\begin{lemmav}\label{lema caustiche esterne pert}
	If $I_0\in\mathcal I\backslash\bar{\mathcal{I}}$ is such that $\theta(I_0)$ is Diophantine, then there is $\bar\ep^{(2)}>0$ such that, if $|\ep|<\bar\ep^{(2)}$, then $G_E(x, y; \xi, \ep)$ is continous in $\ep$, differentiable in $\xi$ and such that $\partial _\xi G_E(x, y; \xi, \ep)$ is continous in $\ep$. 
\end{lemmav} 
\begin{proof}Recalling (\ref{caustica esterna pert}), the proof of the Lemma relies on showing that all the quantities involved in the definition of $G_E(x,y; \xi, \ep)$, namely, $a^{-2},b^{-2}, \cos\bar\beta$ and $\sin\bar\beta$ are continous in $\ep$, differentiable in $\xi$ and with derivative continous in $\ep$, provided the latter is small enough.\\
Starting from $a^{-2}$ and $b^{-2}$, from (\ref{a2 b2 pert}) it is clear that the expression of $p_0\cdot v_0$ as a function of $\xi$ and $\ep$ is needed. 
	 Recalling the definition (\ref{def dominio perturbato}), denoted with $t(\xi; \ep)$ and $n_e(\xi; \ep)$ the tangent and the outward-pointing normal unit vectors to $\gamma_\ep$ in $p_0$, one has that
	\begin{equation*}
		v_0=\sqrt{2\E-\om \rho(\xi;\ep)^2}\left(\cos\alpha~ n_e(\xi; \ep)+\sin\alpha ~t(\xi; \ep)\right). 
	\end{equation*}
	Expliciting $\cos\alpha, \sin\alpha, t(\xi; \ep), n_e(\xi; \ep)$ and setting for simplicity $\rho\equiv\rho(\xi; \ep)$, $\rho'\equiv d\rho(\xi; \ep)/d\xi$ and $I(\xi; \ep)\equiv I$, one obtains 
	\begin{equation}\label{vx vy}
		\begin{aligned}
			&	v_0=\begin{pmatrix}
				v_x\\v_y
			\end{pmatrix}=\frac{1}{\sqrt{\rho^2+\rho'^2}}
			\begin{pmatrix}
				\sqrt{2}I\left(\rho'\cos\xi-\rho\sin\xi\right)+\sqrt{2\E-\om\rho^2-2I^2}\left(\rho'\sin\xi+\rho\cos\xi\right)\\
				\sqrt{2}I\left(\rho'\sin\xi+\rho\cos\xi\right)+\sqrt{2\E-\om\rho^2-2I^2}\left(\rho\sin\xi-\rho'\cos\xi\right)
			\end{pmatrix}\\
			&
			\Rightarrow p_0\cdot v_0=\frac{\rho(\xi;\ep)}{\sqrt{\rho^2(\xi; \ep)+{\rho'}^2(\xi;\ep)}}\left(\rho(\xi; \ep)\sqrt{2\E-\om \rho^2(\xi; \ep)-2I^2(\xi; \ep)}+\sqrt{2}I(\xi, \ep)\rho'(\xi; \ep)\right), 
		\end{aligned}
	\end{equation}
which has the desired continuity and differentiability properties provided $\ep$ is small enough. This implies that $a^2$ and $b^2$ have the same properties. Moreover, it is trivial that $a^2>0$ and, since for $\ep=0$ 
\begin{equation*}
	b^2_{|\ep=0}=\frac{\E-\sqrt{\E^2-2\om I_0^2}}{\om}>0, 
\end{equation*}
by the continuity of $b$ with respect to $\ep$ we have also $b^2>0$ for $\ep$ small enough. Applying the same reasoning, we can infer $\sqrt{a^2-b^2}>0$. \\
Going back to (\ref{sinbeta cosbeta pert}), $\cos\bar\beta$ is then continous and differentiable, and the same conclusion holds for $\sin\bar\beta$ if one can ensure that $v_y'$ has the same sign for all the points of the orbit $(\xi, I(\xi; \ep))$. From (\ref{vx vy}), in the plane $R(O,x',y')$ one has 
\begin{equation*}
	v_y'=\frac{1}{\sqrt{\rho^2(\xi; \ep)+{\rho'}^2(\xi; \ep)}}\left(\rho(\xi;\ep)\sqrt{2}I(\xi; \ep)-\rho'(\xi;\ep)\sqrt{2\E-\om\rho^2(\xi;\ep)-2I^2(\xi;\ep)}\right), 
\end{equation*}
which for $\ep=0$ translates in
\begin{equation*}
	{v'_y}_{|\ep=0}=\sqrt{2}I_0\neq0. 
\end{equation*}
Taking again advantage of the continuity of $v'_y$ with respect $\ep$, we can finally ensure that for $\ep$ small enough the thesis is proved.
\end{proof}
\begin{prop}\label{prop caustiche esterne pert}
		If $I_0\in\mathcal I\backslash\bar{\mathcal{I}}$ is such that $\theta(I_0)\in\mathcal D$, then there exists $\bar\ep_E$ such that for $|\ep|<\bar\ep_E$ the caustic $\Gamma_E(\xi; \ep, \theta(I_0))$ is globally well defined. 
\end{prop}
\begin{proof}
	As the nondegeneracy condition (\ref{nondeg cond pert}) holds for $\ep=0$ (cfr. (\ref{nondeg ex nonpert})), from Lemma \ref{lema caustiche esterne pert}, for every $\bar\xi\in[0,2\pi]$ there exists $\bar\ep^{(2)}(\bar\xi)$ such that for every $|\ep|<\bar\ep^{(2)}(\bar\xi)$ condition (\ref{nondeg cond pert}) is satisfied. By the implicit function theorem, there are $\lambda_\xi(\bar\xi), \lambda_\ep(\bar\xi)>0$ such that the curve $(x(\xi; \ep), y(\xi, \ep))$ solution of (\ref{sistema caustiche pert}) is well defined in $R(\bar\xi)=\left(\bar\xi-\lambda_\xi(\bar\xi),\bar\xi+\lambda_\xi(\bar\xi)\right)\times\left(-\lambda_\ep(\bar\xi), \lambda_\ep(\bar\xi)\right)$. For the uniqueness of the solution, if $\bar\xi_1$ and $\bar\xi_2$ are such that $R(\bar\xi_1)\cap R(\bar\xi_2)\neq\emptyset$, the curve coincides in such intersection. As $[0, 2\pi]$ is compact, it is possible to find $N>0$, $\{\bar\xi_1, \dots\bar\xi_N\}\subset[0,2\pi]$ such that 
	\begin{equation*}
		[0,2\pi]\subset\bigcup_{i=1}^N\left(\bar\xi_i-\lambda_\xi(\bar\xi_i),\bar\xi_i+\lambda_\xi(\bar\xi_i)\right), 
	\end{equation*}
then, setting 
\begin{equation*}
	\bar\ep_E=\min_{i\in\{1, \dots, N\}}\lambda_\ep(\bar\xi_i), 
\end{equation*}
for every $\ep>0$ such that $|\ep|<\bar\ep_E$ the curve $\Gamma_E(\xi; \ep, I_0)=(x(\xi; \ep), y(\xi; \ep) )$ is globally well defined in $[0, 2\pi]$.
\end{proof}

\paragraph{\textbf{Inner caustics}} 

Following the same reasoning applied for the outer caustic, let us consider the inner problem 
\begin{equation*}
\begin{cases}
	z''(s)=-\frac{\mu}{|z(s)|^3}z(s), \quad &s\in[0,T_I]\\
	\frac{1}{2}|z'(s)|^2-\E-h-\frac{\mu}{|z(s)|}=0 &s\in[0,T_I]\\
	z(0)=p_0, \quad z'(0)=v_0
\end{cases}	
\end{equation*}
by fixing $p_0=|p_0|e^{i\xi}$, $v_0=\sqrt{2(\Eh+\mu/|p_0|)}e^{i\theta_v}$ such that $\theta_v-\xi\in(\pi/2, 3\pi/2)$. This last assumption, which is done to guarantee that the hyperbola points inward a circle of radius $|p_0|$, can be ensured for $\ep$ small enough and suitable bounds on $I(\xi)$. Rotating again the reference frame $R(O, x,y)$ by an angle $-\xi$, we obtain $R(O, x', y')$ such that $p_0=|p_0|(1,0)$. \\
Recalling (\ref{cartesiana iperbole 0}), in the reference frame $R(O, x'', y'')$ were the hiperbola's pericenter lies on the positive half of the $x$-axis, its Cartesian equation is given by: 
\begin{equation*}
(e^2-1){x''}^2-{y''}^2-2pex''+p^2=0\text{ with }x\leq\frac{p}{e+1}, 	
\end{equation*}
where 
\begin{equation*}
	p=\frac{k^2}{\mu}, e=\frac{\sqrt{\mu^2+2(\Eh)k^2}}{\mu}, \quad k=|p_0\wedge v_0|. 
\end{equation*}
To find the corresponding equation in the reference frame $R(O, x', y')$, one can search again for the angle $\bar\delta$ such that the arc defined by
\begin{equation}\label{sistema bardelta }
	\begin{aligned}
	&(e^2-1)(x'\cos\bar\delta+y'\sin\bar\delta)^2-(y'\cos\bar\delta-x'\sin\bar\delta)^2-2ep(x'\cos\bar\delta+y'\sin\bar\delta)+p^2=0, \\
	&x'\cos\bar\delta+y'\sin\bar\delta\leq \frac{p}{e+1}
	\end{aligned}
\end{equation}
passes from $p_0=|p_0|(1,0)$. Solving (\ref{sistema bardelta }) with $x'=|p_0|$ and $y'=0$, one obtains 
\begin{equation*}
\cos\bar\delta=\frac{p-|p_0|}{e|p_0|}, 	
\end{equation*}
which is in $[-1,1]$ if we take non-degenerate hyperbol\ae. Referring to $v_y'$ as the vertical component of  $v_0$ in $R(O, x',y')$, one has then 
\begin{equation*}
	\sin\bar\delta=\begin{cases}
		\frac{(e^2-1)|p_0|^2+2 p |p_0|-p^2}{e|p_0|}\quad \text{if }v'_y>0\\
		-\frac{(e^2-1)|p_0|^2+2 p |p_0|-p^2}{e|p_0|}\quad \text{if }v'_y<0
	\end{cases}.
\end{equation*}
Returning to the original reference frame $R(O, x,y)$, one obtains then the Cartesian equation for the inner Keplerian arc
\begin{equation}\label{equazione cartesiana interna}
	\begin{aligned}
G_I(x,y;\xi\ep)=&(e^2-1)(x\cos(\bar\delta+\xi)+y\sin(\bar\delta+\xi))^2-(y\cos(\bar\delta+\xi)-x\sin(\bar\delta+\xi))^2+\\&-2 p e (x\cos(\bar\delta+\xi)+y\sin(\bar\delta+\xi))+p^2=0\\
&x\cos(\bar\delta+\xi)+y\sin(\bar\delta+\xi)\leq\frac{p}{e+1}. 
\end{aligned}
\end{equation}
Note that, with reference to the polar angles $\xi$ and $\theta_v$, the angle $\bar\delta$ can be also expressed as 
\begin{equation}\label{bardelta int}
	\bar\delta=\frac{\sin(\theta_v-\xi)}{|\sin(\theta_v-\xi)|}\arccos\left(\frac{p-|p_0|}{e|p_0|}\right). 
\end{equation}
As in the case of the outer dynamics, the global good definition of the inner caustic $\Gamma_I(\xi; \ep, \theta(I_0))$ depends on proving that $G_I(x, y; \xi, \ep)$  differentiable in $\xi$ and that both $G_I$ and $\partial_\xi G_I$ are continous in $\ep$. 
\begin{lemmav}\label{lemma caustiche interne pert}
	If $I_0\in\mathcal I\backslash\bar{\mathcal{I}}$ is such that $\theta(I_0)\in\mathcal D$, then there is $\bar\ep^{(3)}>0$ such that, if $|\ep|<\bar\ep^{(3)}$, then $G_I(x, y; \xi, \ep)$ is continous in $\ep$, differentiable in $\xi$ and such that $\partial _\xi G(x, y; \xi, \ep)$ is continous in $\ep$. 
\end{lemmav} 
\begin{proof}
 As in the case of Lemma \ref{lema caustiche esterne pert}, one needs to prove the desired regularity properties on the quantities $p$ and $e$, as well as $\sin\bar\delta$ and $\cos\bar\delta$. As all these quantities depend on $k=|p_0\wedge v_0|$, let us find the expression of the angular momentum as a function of $\xi$. As already done in \S \ref{sec: first return}, let us now denote with $\alpha$ the angle between $v_0$ and the \textit{inward-}pointing normal unit vector to $\gamma_\ep$ in $p_0$, which we indicate with $n_i(\xi)$; then referring to  (\ref{I alpha}) and using the same notation of Lemma \ref{lema caustiche esterne pert}, we have 
 \begin{equation*}
 	\begin{aligned}
 		v_0&=\sqrt{2\left(\Eh+\frac{\mu}{\rho}\right)}\left(\sin\alpha~t(\xi)+\cos\alpha~n_i(\xi)\right)=\\
 		&=\frac{1}{\sqrt{\rho^2+\rho'^2}}\begin{pmatrix}
 			\sqrt{2}I\left(\rho' \cos\xi-\rho\sin\xi\right)-\sqrt{2(\Eh+\mu/\rho-I^2)}\left(\rho'\sin\xi+\rho\cos\xi\right)\\
 			\sqrt{2}I\left(\rho' \sin\xi+\rho\cos\xi\right)+\sqrt{2(\Eh+\mu/\rho-I^2)}\left(\rho'\cos\xi-\rho\sin\xi\right). 
 		\end{pmatrix}
 	\end{aligned}
 \end{equation*}
 And, since $p_0=\rho e^{i\xi}$,  
 \begin{equation*}
 	\begin{aligned}
 		k=|p_0\wedge v_0|=\frac{\sqrt{2}\rho(\xi; \ep)}{\sqrt{\rho^2(\xi; \ep)+{\rho'}^2(\xi; \ep)}}\left(I(\xi; \ep)\rho(\xi; \ep)+\rho'(\xi; \ep)\sqrt{\Eh+\frac{\mu}{\rho(\xi; \ep)}-I^2(\xi;\ep)}\right)\\
 		p=\frac{k^2}{\mu}=\frac{2\rho^2(\xi; \ep)}{\mu(\rho^2(\xi; \ep)+{\rho'}^2(\xi; \ep))}\left(I(\xi; \ep)\rho(\xi; \ep)+\rho'(\xi; \ep)\sqrt{\Eh+\frac{\mu}{\rho(\xi; \ep)}-I^2(\xi;\ep)}\right)^2\\
 		e=\sqrt{1+\frac{4(\Eh)\rho^2(\xi; \ep)}{(\rho^2(\xi; \ep)+{{\rho'}^2(\xi; \ep)})\mu^2}\left(I(\xi; \ep)\rho(\xi; \ep)+\rho'(\xi; \ep)\sqrt{\Eh+\mu/\rho(\xi; \ep)-I^2(\xi; \ep)}\right)^2}
 	\end{aligned}
 \end{equation*}
The regularity of $p$ and $e$ is then ensured whenever $\rho^2(\xi;\ep)+{\rho'}^2(\xi; \ep)\neq0$, which is true for $\ep$ small enough.  
As for $\sin\bar\delta$ and $\cos\bar\delta$, from (\ref{bardelta int}) one can infer that the requested regularity is ensured if $\sin(\theta_v-\xi)$ has always the same sign on the orbit $(\xi,I(\xi; \ep))$. As in the case of the outer orbit, this is a consequence of the continuity of $\rho(\xi; \ep)$, $\rho'(\xi; \ep)$ and $I(\xi;\ep)$ with respect to $\ep$. Denoting with $\theta_{n_i}$ the polar angle of $n_i(\xi)$, from the definition of  $\alpha$ one has $\theta_v-\xi=\theta_{n_i}-\xi-\alpha$, and then 
\begin{equation*}
	\begin{aligned}
	\sin(\theta_v-\xi)&=\sin(\theta_{n_i}-\xi)\cos\alpha+\cos(\theta_{n_i}-\xi)\sin\alpha=\\&=\frac{1}{\rho(\xi; \ep)\sqrt{\Eh+\frac{\mu}{\rho(\xi; \ep)}}}\left(\sqrt{\Eh+\frac{\mu}{\rho(\xi; \ep)}-I^2(\xi; \ep)}|n_I(\xi)\wedge\gamma_\ep(\xi)|-I(\xi; \ep)\gamma_\ep(\xi)\cdot n_i(\xi)\right)= \\
	&=\frac{\rho'(\xi; \ep)\sqrt{\Eh+\frac{\mu}{\rho(\xi; \ep)}-I^2(\xi; \ep)}+\rho(\xi; \ep)I(\xi; \ep)}{\sqrt{\rho^2(\xi; \ep)+{\rho'}^2(\xi; \ep)}\sqrt{\Eh+\frac{\mu}{\rho(\xi; \ep)}}}. 
	\end{aligned}
\end{equation*} 
For $\ep=0$, $\sin(\theta_v-\xi)_{|\ep=0}=I_0/\sqrt{\Eh+\mu}\neq0$, then, if $\ep$ is small enough, $\sin(\theta_v-\xi)$ has always the same sign of $I_0$, and $\sin\bar\delta$, $\cos\bar\delta$ are differentible in $\xi$ and continous in $\ep$, with derivative continous in $\ep$. 
\end{proof}
Making use of Lemma \ref{lemma caustiche interne pert} and following the same reasoning used in the proof of Proposition \ref{prop caustiche esterne pert}, it is possible to prove the existence of a well-defined inner caustic $\Gamma_I(\xi; \ep, \theta(I_0))$ related to the invariant curve for the map $\mathcal\F_\ep$ with rotation number $\theta(I_0)$. 
\begin{prop}\label{prop caustiche interne pert}
	If $I_0\in\mathcal I\backslash\bar{\mathcal{I}}$ is such that $\theta(I_0)\in\mathcal D$, then there exists $\bar\ep_I$ such that for $|\ep|<\bar\ep_I$ the caustic $\Gamma_I(\xi; \ep, \theta(I_0))$ is globally well defined. 
\end{prop}

\appendix

\section{Proofs of Theorems \ref{thm ex globale esterna unp} and \ref{thm ex globale interna unp}}\label{appA}

\subsection{Outer arcs}
In the case of the inner arcs, the proof of the existence and uniqueness Theorem \ref{thm ex globale esterna unp} lies on the direct computation of the solutions of problem \ref{problema esterno testo} for fixed $p_0$ and $p_1$. Moreover, the results obtained in the following proof are used in Section \ref{sec: mappa cerchio esplicita} to give the explicit expression of the outer shift in the unperturbed case. 
\begin{proof}[Proof of Theorem \ref{thm ex globale esterna unp}]
	Fix $p_0=e^{i\theta_0}\in\partial D_0$, and, given $\alpha\in(-\pi/2,\pi/2)$, consider the Cauchy problem 
	
	\begin{equation}\label{extP}
		\begin{cases}
			z''(s)=-\om z(s),\\
			z(0)=p_0, z'(0)=v_0=\sqrt{2\E-\om}e^{i\theta_0+\alpha},
		\end{cases}
	\end{equation}
	whose solution $z(s;p_0, v_0)$ is an ellipse whose parameters depend on the initial conditions and can be decoupled as
	\begin{equation}
		\begin{aligned}
			z(s)=(x(s),y(s))&=({p}_{0,x}\cos{\omega s }+\frac{v_{0,x}}{\omega}\sin{\omega s},{p}_{0,y}\cos{\omega s }+\frac{v_{0,y}}{\omega}\sin{\omega s})  
		\end{aligned}
	\end{equation}
	Since  $\alpha\in(-\pi/2,\pi/2)$,  the orbit is exterior to $D$ in a neigborhood of $s=0$. Let $s_1>0$ the first positive instant for which $z(s_1; 		p_0, v_0)\in\partial D_0$ again, and define $p_1=e^{i\theta_1}=p_1(p_0, \alpha)=z(s_1; p_0, v_0)$. As the system is invariant under rotations, the shift $\theta_E$ from $\theta_0$ to $\theta_1$ and $s_1$ depend only on the direction of $v_0$ with respect to the radial direction, i.e. on $\alpha$. We can then fix $p_0=\bar{p}_0=(1,0)$, and we have $\theta_E(\alpha)=\theta_1$.  
	The solution $z(s; \bar{p}_0,v_0)$ simplifies as 
	\begin{equation}
		\begin{aligned}
			z(t)=(x(s),y(s))=\left(\cos{\omega s }+\frac{v_x}{\omega}\sin{\omega s},\frac{v_y}{\omega}\sin{\omega s}\right), 
		\end{aligned}
	\end{equation}
	
	from which  one has 
	\begin{equation}
		r^2(s)=x^2(s)+y^2(s)=\frac{\om-\E}{\om}\cos{(2\omega s)}+\frac{v_x}{\omega}\sin{(2\omega s)}+\frac{\E}{\om}=A\cos{(2\omega s+\bar{\alpha})}+\frac{\E}{\om}, 
	\end{equation}
	with $A\in\mathbb{R}$ and $\bar{\alpha}\in[0,2\pi)$ such that 
	\begin{equation}
		\begin{cases}
			A\cos{\bar{\alpha}}=\frac{\om-\E}{\om},\quad A\sin{\bar{\alpha}}=-\frac{v_x}{\omega},  
		\end{cases}
	\end{equation}
	and, since $v_x>0$,  
	\begin{equation}\label{baralpha1}
		\begin{aligned}
			\cot{\bar{\alpha}}=\frac{\E-\om}{\omega v_x}\Rightarrow \bar{\alpha}=\text{arccot}\left(\frac{\E-\om}{\omega v_x}\right)\in\left(0,\frac{\pi}{2}\right),\\
			\cos{\bar{\alpha}}=\frac{\E-\om}{\sqrt{\om v_x^2+(\E-\om)^2}}, \quad \sin{\bar{\alpha}}=\frac{\omega v_x}{\sqrt{\om v_x^2+(\E-\om)^2}}\\
			A=-\frac{v_x}{\omega \sin{\bar{\alpha}}}=-\frac{\sqrt{\om v_x^2+(\E-\om)^2}}{\om}<0.
		\end{aligned}
	\end{equation}
	The time $s_1>0$ is such that $\rho(s_1)=1$ and is given by  $s_1=(\pi-\bar{\alpha})/\omega$: if $y(s_1)\neq 0$ (namely, $\alpha\neq0$), the polar angle $\theta_1$ of the point $p_1$ is given by
	\begin{equation}
		\theta_1=
		\begin{cases}
			\text{arccot}\left(\frac{x(s_1)}{y(s_1)}\right)&\text{ if }\alpha>0,  \\
			\text{arccot}\left(\frac{x(s_1)}{y(s_1)}\right)-\pi&\text{ if }\alpha<0, 
		\end{cases}
	\end{equation}
	where we took into account that, for $\alpha<0$, $\bar{\theta}_E\in[\pi,2\pi]$, then one has to take the second determination of arccot. \\
	Direct computations of the homotetic solution (corresponding to $\alpha=0$) and equation  (\ref{baralpha1}), along with the definition of $\theta_E$, lead finally to 
	\begin{equation}
		\theta_E(\alpha)=
		\begin{cases}
			\theta_E^+(\alpha)=\text{arccot}\left(\frac{\om}{(2\E-\om)\sin{(2\alpha)}}+\cot{(2\alpha)}\right) &\text{ if }\alpha>0, \\
			0 &\text{ if }\alpha=0, \\
			\theta_E^-(\alpha)=\text{arccot}\left(\frac{\om}{(2\E-\om)\sin{(2\alpha)}}+\cot{(2\alpha)}\right)-	\pi &\text{ if }\alpha<0.  
		\end{cases}
	\end{equation}
	If $\E>\om$, the function $\theta_E(\alpha)$ is of class $C^1$ in $(-\pi/2, \pi/2)$ and assumes all the values in $(-\pi,\pi)$. Moreover, 
	\begin{equation}
		\frac{d \theta_E}{d\alpha}(\alpha)=\frac{(2\E-\om)(2\E-\om+\om\cos{(2\alpha)})}{2\E(\E-2\om)-(2\E-\om)\om\cos{(2\alpha)}}>0\text{ for all }\alpha\in\left(-\frac{\pi}{2},\frac{\pi}{2}\right). 
	\end{equation}
	From the inverse function theorem, there exist a unique function $\alpha: (-pi,\pi)\to(-\pi/2,\pi/2)$, $\theta_1\mapsto\alpha(\theta_1)$ such that for every $\theta_1$ we have
	\begin{equation}
		p_1=e^{i\theta_1}=z\left(s_1(\alpha(\theta_1)); \bar{p}_0, \sqrt{2\E-\om}e^{i\alpha(\theta_1)}\right).
	\end{equation} 
	Moreover, $\alpha(\theta_1)\in C^1(-\pi,\pi)$. \\
	Fixing now $p_0, p_1\in D_0$ such that $|p_0-p_1|<2$, we have that $|\theta_0-\theta_1|<\pi$, then problem 
	\begin{equation}
		\begin{cases}
			z''(s)=-\om z(s),\\
			z(0)=p_0, z'(0)=v_0=\sqrt{2\E-\frac{\om}{2}}e^{i\left(\theta_0+\alpha(\theta_1-\theta_0)\right)},
		\end{cases}
	\end{equation}
	admits the unique solution $z(s; p_0,p_1)$. If we define $T=s_1$ as above, we have that $z(T; p_0,p_1)=p_1$ and $|z(s)|>1$ for every $s\in(0,T)$, while the energy conservation law is ensured by the choice of $v_0$.
	Moreover, by the differentiable dependence on the initial contitions of the Cauchy problem and the fact that $\alpha(\theta_1)$ is of class $C^1$, one can conclude that $z(s; p_0, p_1)$ is differentiabl as a function of its endpoints.

\end{proof}

\subsection{Inner arcs}
Unlike the outer case, the inner Kepler problem presents a singularity in the origin, which should be treated with more sophisticated strategies. To this end, techniques such as the Levi-Civita regularisation and more general results from Riemannian Geometry are used. \\
The Levi-Civita regularization technique consists in a change both in the temporal parameter and the spatial coordinates, in order to remove the singularity of Kepler-type potentials.
\begin{lemmav}\label{lem levi civita}
	Let $p_0, p_1\in\R^2\backslash\{0\}$. The fixed end problem
	\begin{equation} \label{PF}
		\begin{cases}
			(HS_I)[z(s)], &s\in[0,T]\\
			z(0)=p_0, \text{ }z(T)=p_1
		\end{cases}
	\end{equation}
	is conjugated to the Levi-Civita problem 
	\begin{equation}\label{PLC}
		\begin{cases}
			\ddot w(\tau)=\Omega^2 w(\tau) \quad &\tau\in[0,\tilde{T}]\\
			\frac{1}{2}|\dot{w}(\tau)|^2-\frac{\Omega^2}{2}|w(\tau)|^2-E=0, &\tau\in[0,\tilde{T}]\\
			w(0)=w_0, \text{ }w(\tilde{T})=w_1
		\end{cases}
	\end{equation}
	with $\Omega^2=2(\Eh)$, $E=\mu$, $w_0^2=p_0$, $w_1^2=p_1$ and $\tau=\tau(s)$ such that $\frac{d\tau}{ds}=\frac{1}{2|z(s)|}$. 
\end{lemmav}
\begin{proof}
	Let us consider the reparametrisation $s=s(\tau)$ such that $\frac{d}{ds}=\frac{1}{2r}\frac{d}{d\tau}$ with $r=|z(s)|$. Then, denoting with the dot the derivation with respect to $\tau$, 
	\begin{equation}
		\frac{d^2}{ds^2}=\frac{1}{4}\left(-\frac{1}{r^3}\dot{r}\frac{d}{d\tau}+\frac{1}{r^2}\frac{d^2}{d\tau^2}\right):
	\end{equation}
	the first and second equations in (\ref{PF}) can be expressed with respect to $\tau$ as
	\begin{equation}\label{lem1}
		r\ddot{z}-\dot{r}\dot{z}+4\mu z=0, \quad \frac{1}{8r^2}|\dot{z}|^2=\Eh+\frac{\mu}{r}. 
	\end{equation} 
	Considering now the new spatial variable in the complex plane $\mathbf{C}\simeq\R^2$ given by $z=w^2$, we have from (\ref{lem1}) :
	\begin{equation}
		\frac{1}{2}|\dot{w}|^2-\frac{\Omega^2}{2}|w|^2-E=0, \quad r\ddot{w}+w\left(2E-|\dot{w}|^2\right)=0\Rightarrow \ddot{w}=\Omega^2w. 
	\end{equation}
\end{proof}
We will refer to the time variable $\tau$ as the Levi-Civita time, and to the new reference system as the Levi-Civita plane; the original time and coordinate space will be called \textit{physical} time and plane. Moreover, following up on Notation \ref{notazione}, we will denote the first two lines of system (\ref{PLC}) with the abbreviation $(HS_{LC})[w]$. \\
\begin{rem}\label{ossLC}
	From Lemma \ref{lem levi civita} we have that, if $w(\tau; w_0,w_1)$ is a solution of (\ref{PLC}), then $z(s;p_0, p_1)=\left(w(\tau;w_0,w_1)_{|\tau=\tau(s)}\right)^2$ is a solution of (\ref{PF}) with endpoints $p_0=w_0^2$ and $p_1=w_1^2$. On the other hand, as the complex square determines a double covering of $\R^2\backslash\{0\}$, if we fix $p_0=r_0e^{i\theta_0},p_1=r_1e^{i\theta_1}\in\R^2\backslash\{0\}$, with $p_0\neq p_1$, and $z(z,p_0,p_1)$ is a solution of (\ref{PF}), we can find two distinct solutions of (\ref{PLC}) such that $\left(w(\tau;w_0,w_1)_{|\tau=\tau(s)}\right)^2=z(s;p_0, p_1)$. For, there are two pairs of points in the Levi-Civita plane, namely, $w_0^{\pm}=\pm\sqrt{r_0}e^{i\theta_0/2}$ and $w_1^{\pm}=\pm\sqrt{r_1}e^{i\theta_1/2}$, such that $w_{0\backslash1}^{\pm}=p_{0\backslash1}$. Supposing that problem (\ref{PLC})
	admits the solutions $w(\tau;w_0^-, w_1^-)$ and $w(\tau; w_0^-,w_1^+)$, it is straightforward that $w(\tau;w_0^+, w_1^+)=-w(\tau;w_0^-, w_1^-)$ and $w(\tau; w_0^+,w_1^-)=-w(\tau; w_0^-,w_1^+)$: passing to the physical plane, we have then two distinct solutions of (\ref{PF}), given by $\left(w(\tau; w_0^-,w_1^+)_{\tau=\tau(s)}\right)^2$ and $\left(w(\tau; w_0^-,w_1^-)_{\tau=\tau(s)}\right)^2$.	
	\\If instead $p_0=p_1=r_0e^{i\theta_0}$, the solution given by $w(\tau, w_0^-, w_1^-)$ collapses into a single point: there is only one solution of (\ref{PLC}) conjugated to $z(s; p_0,p_1)$, and it can be computed explicitely by choosing $w_0=-\sqrt{r_0}e^{i\theta_0/2}$ and $w_1=\sqrt{r_0}e^{i\theta_0/2}$: 
	\begin{equation}
		w(\tau, w_0, w_1)=\sqrt{r_0}\sqrt{\frac{2E}{\Omega}}\sinh{(\Omega(\tau-\tau_0))}e^{i\theta_0/2}, \quad \tau_0=\frac{1}{\Omega}\arcsinh\left(\sqrt{\frac{\Omega}{2E}}\right)=\frac{T}{2}, 
	\end{equation}
	which corresponds to an ejection-collision solution $z(s; p_0,p_0)$ parallel to the direction $e^{i\theta_0}$. 
\end{rem}
To find solutions of (\ref{PF}) we can then search for solutions of $(\ref{PLC})$ with suitable endpoints. We will prove that the fixed ends problem in the Levi-Civita plane admits a unique solution for every pair of points in $\R^2$: to this purpose, we need to introduce some known results from Riemannian Geometry. \\
The solutions of (\ref{PLC}) can be seen as reparametrizations of geodesic curves which connect $w_0$ and $w_1$ in the Riemannian manifold $(\R^2,\tilde{g})$, with the metric $\tilde{g}$ given by the metric tensor 
\begin{equation}\label{metrica conforme}
	\tilde{g}_{ij}=e^{\sigma(w)}\delta_{ij}, \text{ }\sigma(w)=\ln{\left(\frac{\Omega^2}{2}|w|^2+E\right)}>0 
\end{equation}
with $i,j\in\{1,2\}$: if we prove the existence of a unique geodesic in $(\R,\tilde{g})$ with connects $w_0$ and $w_1$, our claim follows straightforwardly. Once we have verified that its hypotheses hold, this will follow from the Cartan-Hadamard theorem: to retrieve its statement, along with the definition of all the involved quantities, we refer to  \cite{lee2006riemannian} and  \cite{carmo1992riemannian}. 
Taking into account (\ref{metrica conforme}), one can prove that, given $w\in\R^2$, the sectional curvature of $T_w\R^2$ at $w$ is given by 
\begin{equation}\label{curv sezionale}
	K(w)=-\frac{\Delta\sigma(w) }{e^{\sigma(w)}}=-\frac{\Omega^2 E}{\left(\frac{\Omega^2}{2}|w|^2+E\right)}<0. 
\end{equation}
As $\R^2$ is simply connected, to apply the Cartan-Hadamard theorem one has then to verify that the conformal metric $(M, \tilde g)$ is complete. In view of Hopf-Rinow Theorem, this is equivalent to show that $(M, d_{\tilde g})$ is complete as a metric space, where the distance $d_g(w_0, w_1)$ is defined as the infimum of all the lengths in $(M; g)$ of piecewise $C^1$ curves which connect $w_0$ and $w_1$. To prove this last assertion, we shall take advantage on the following Lemma, stated as part of Theorem $1$ in \cite{dirmeier2012growth, dirmeier2013392}.  

\begin{lemmav}
	Let $(M, g)$ be a non-compact and complete Rieannian manifold and $A: M\to(0,\infty)$ a positive function. We denote a conformally transformed metric on $M$ by $\tilde g=g/A^2$. If $A$ grows at most linerly towards $g$-infinity on $M$, that is, 
	\begin{equation}\label{cond cresc}
		\forall x_0\in M\text{ } \exists c_1, c_2>0 \text{ such that }\forall x\in M  \quad A(x)\leq c_1 d_g(x_0, x)+c_2, 
	\end{equation}
	 then $(M, \tilde g)$ is complete. 
\end{lemmav}
In our regularised system, $g$ is the usual Euclidean metric on $\R^2$, then $A=\left(\sqrt{\frac{\Omega^2}{2}|w|^2+E}\right)^{-1}$, and (\ref{cond cresc}) is trivially verified with $c_1=1$ and $c_2=1/\sqrt{\mu}$, as
\begin{equation*}
	A(w)=\frac{1}{\sqrt{\frac{\Omega^2}{2}|w|^2+E}}\leq\frac{1}{\sqrt{\mu}}\leq |w_0-w|+\frac{1}{\sqrt{\mu}}
\end{equation*}
for every $w_0, w\in\R^2$. The Cauchy-Hadamard theorem can be then applied, and as a consequence the existence and uniqueness of an orbit for any pairs of points in the Levi-Civita plane can be proved. 
\begin{prop}\label{propLC} 
	For every $w_0,w_1\in\R^2$ $\exists|w(\tau; w_0,w_1)$ solution of (\ref{PLC}) for some $\tilde{T}>0$. Moreover, $w(\tau;w_0,w_1)$ is of class $C^1$ with respect to variations of $w_0$ and $w_1$. 
\end{prop}
Taking together Proposition \ref{propLC} and Remark $\ref{ossLC}$, we can then pass to the physical plane, obtaining the below existence theorem. 
\begin{thmv}\label{thmunp}
	For every $p_0,p_1\in \partial D_0$ with $p_0\neq p_1$ there are exactly two classical solutions $z_0(s;p_0,p_1)$ and $z_1(s;p_0,p_1)$ of problem (\ref{PF})  for some $T_0,T_1>0$, which are of class $C^1$ with respect to $p_0$ and $p_1$.\\If instead $p_0=p_1$, there is a unique solution of (\ref{PF}), which is ejection-collision.
\end{thmv}
The previous results ensure the existence and differentiability of exactly two orbits (one in the case of ejection-collision solutions) which join two different points in $\partial D_0$; on the other hand, to gain transversality properties an explicit expression of the two solutions is needed. 
Classical arguments of Celestial Mechanics (see \cite{celletti2010stability}) allow us to obtain their Cartesian equations: let us consider the first equation of Problem (\ref{PF}), namely,
$z''(s)+\frac{\mu}{|z(s)|^3}z(s)=0$, and express it in polar coordinates, obtaining
\begin{equation}\label{polarprob}
	\begin{cases}
		r''(s)-r(s)\theta'(s)=-\frac{\mu}{r^2(s)},\\
		r(s)\theta''(s)-2r'(s)\theta'(s)=0. 
	\end{cases}
\end{equation}
The second equation express the conservation of the modulus of the angular momentum 
\begin{equation}
	k=r^2(s)\theta'(s)=|z(s)\wedge z'(s)|=|z(0)\wedge z'(0)|. 
\end{equation}
Taking into account the uniqueness of the ejection-collision solution for $p_0=p_1$, we can state that $k=0\Leftrightarrow p_0=p_1$: assuming $p_0\neq p_1$, we have then $k\neq0$. \\
The polar equation associated to a solution of  $z''(s)+\frac{\mu}{|z(s)|^3}z(s)=0$ having energy $\Eh$ and angular momentum $k$ is
\begin{equation}\label{polar}
	r(f)=\frac{p}{1+e\cos{f}}, 
\end{equation}
where $f\in[0,2\pi)$ is traditionally called \textit{true anomaly} and 
\begin{equation}\label{p e int}
	p=\frac{k^2}{\mu}, \quad e=\sqrt{1+\frac{2k^2(\Eh)}{\mu^2}}>1. 
\end{equation}
With a suitable choice of a reference frame denoted with $\mathcal{R}(0,x',y')$, we can see that equation (\ref{polar}) represents the branch of particular hyperbol\ae~ with the concavity facing the origin. More precisely, if $a=\mu/\left[2(\Eh)\right]$: 
\begin{itemize}
	\item if we choose $(x',y')=(r\cos{f},r\sin{f})$, we obtain that for every $s\in\R$
	\begin{equation}
		\left(x'(s),y'(s)\right)\in\mathcal{H}_1=\left\{(x,y)\in\R^2\text{ }|\text{ }(x-ae)^2(e^2-1)-y^2=a^2(e^2-1), \text{ }x\leq a(e-1) \right\}
	\end{equation}
	\item if instead $(x',y')=(-r\cos{f},r\sin{f})$, we have that $\forall s\in\R$
	\begin{equation}
		(x'(s),y'(s))\in\mathcal{H}_0=\left\{(x,y)\in\R^2\text{ }|\text{ }(x+ae)^2(e^2-1)-y^2=a^2(e^2-1), \text{ }x\geq a(1-e) \right\}. 
	\end{equation}
\end{itemize}
\begin{prop}\label{hip}
	For every $p_0=e^{i\theta_0},p_1=e^{i\theta_1}\in \partial D_0$, $p_0\neq p_1$, the two solutions $z_0(s;p_0,p_1)$ and $z_1(s;p_0,p_1)$ of (\ref{PF}) are such that:
	\begin{itemize}
		\item for every $s\in[0,T_0]$ 
		\begin{equation*}
			\begin{aligned}
				z_0(s; p_0,p_1)\in\mathcal{H}_0&(p_0,p_1)=\\&=\left\{(x',y')\in\R^2\text{ }|\text{ }(x'+ae_0)^2(e_0^2-1)-y'^2=a^2(e_0^2-1), \text{ }x'\geq a(1-e_0) \right\},
			\end{aligned}
		\end{equation*}
		\item for every $s\in[0,T_1]$ 
		\begin{equation*}
			\begin{aligned}
				z_1(s; p_0,p_1)\in\mathcal{H}_1&(p_0,p_1)=\\&=\left\{(x',y')\in\R^2\text{ }|\text{ }(x'-ae_1)^2(e_1^2-1)-y'^2=a^2(e_1^2-1), \text{ }x'\leq a(e_1-1) \right\}, 
			\end{aligned}
		\end{equation*}
	\end{itemize}
	where 
	\begin{equation}
		\begin{pmatrix}
			x'\\y'
		\end{pmatrix}
		=
		\begin{pmatrix}
			\cos{\alpha_0}&\sin{\alpha_0}\\
			-\sin{\alpha_0}&\cos{\alpha_0}
		\end{pmatrix}\begin{pmatrix}
			x\\y
		\end{pmatrix}, 
		\quad \alpha_0=\frac{\theta_1+\theta_0}{2}, 
	\end{equation}
	\begin{equation}\label{ecc}
		e_0=\frac{-x_0+\sqrt{4a^2+4a+x_0^2}}{2a}, \text{ }e_1=\frac{x_0+\sqrt{4a^2+4a+x_0^2}}{2a}, \text{ }x_0=\cos{\left(\frac{\theta_1-\theta_0}{2}\right)}. 
	\end{equation}
\end{prop}
\begin{proof}
	As the system is invariant under rotations, it is sufficient to prove the claims for $p_0,p_1$ symmetric with respect to the $x-$axis, namely, $\bar{p}_0=e^{-i\beta}, \bar{p}_1=e^{i\beta}$, $\beta\in(0,\pi)$: any other cases can be treated as this one after a rotation of angle $\alpha_0$. \\
	For $p_0=\bar{p}_0$ and $p_1=\bar{p}_1$, we have $\alpha_0=0$ and $x_0=\cos{\beta}$: direct computations shows that $e_0$ and $e_1$ as defined in (\ref{ecc}) are the only values of eccentricity such that $\bar{p}_0, \bar{p}_1\in\mathcal{H}_0$ and $\bar{p}_0,\bar{p}_1\in\mathcal{H}_1$. The thesis follows from the uniqueness of the two solutions of (\ref{PF}). 
\end{proof}
\begin{rems}\label{osship}
	\begin{itemize}
		\item One can easily verify that, according to Definition \ref{wind unp}, for every $p_0,p_1\in\partial B_1(0)$, $p_0\neq p_1$, $|p_0-p_1|<2$,
		\begin{equation}
			Ind(z_0[0,T_0], 0)=0 \text{ and }|Ind(z_1[0,T_1],0)|=1. 
		\end{equation}
		\item For $p_1\to p_0$, the solution $z_1(s; p_0,p_1)$ tends to the ejection-collision solution $z(s;p_0,p_0)$ in the direction of $p_0$. For, $p_1\to p_0$ implies $e_1\to 1$ and $k_1\to0$.
		\item It is clear that for every $s\in(0,T_1)$ we have that $|z(s;p_0,p_1)|<1$. 
	\end{itemize}
\end{rems}
The analytic expressions of the traces of $z_0(s;p_0, p_1)$ and $z_1(s;p_0,p_1)$ allow to derive the transversality properties of the latter; as we are interested to the solution which converges to the ejection-collision one for $p_1\to p_0$, we will focus on the transversality of $z_1(s; p_0,p_1)$ (an analogous result can be obtained for $z_0(s;p_0,p_1)$). 
\begin{prop}\label{proptra}
	There is $0<C<1$ such that for every $p_0, p_1\in\partial B_1(0)$, $|p_0-p_1|<2$, we have 
	\begin{equation}\label{transv}
		-p_0\cdot\frac{z_1'(0;p_0, p_1)}{|z_1'(0;p_0, p_1)|}>C\quad\text{ and }\quad
		p_1\cdot\frac{z_1'(T_1;p_0, p_1)}{|z_1'(T_1;p_0, p_1)|}>C. 
	\end{equation}
\end{prop}
\begin{proof}
	If $p_0=p_1$, the ejection-collision solution is orthogonal to $\partial D_0$, then the claims are trivially true. Let us assume $p_0\neq p_1$. 	
	As in Proposition \ref{hip}, it is sufficient to prove the claims for $p_0=e^{-i\beta}=(x_0,-\sqrt{1-x_0^2}), p_1=e^{i\beta}=(x_0,\sqrt{1-x_0^2})$, with $\beta\in(0,\pi/2)$. The positive branch of $\mathcal{H}_1$ can be parametrized as $y(x)=\sqrt{e_1^2-1}\sqrt{(x+ae_1)^2-a^2}$, with $e_1$ as in (\ref{ecc}). Writing $p_1=p_1(x_0)=(x_0,y(x_0))$ and $v_1=(1,\partial_x y(x_0))$, we can express the cosine of the angle between $p_1$ and $v_1$ as a function
	\begin{equation}
		c(x_0)\equiv p_1\cdot\frac{v_1}{|v_1|}=e_0\sqrt{\frac{2a+x_0(x_0+\sqrt{4a^2+4a+x_0^2})}{2+4a}}, 
	\end{equation}
	which is strictly increasing for $x_0\in[0,1]$ and such that $c(0)=\sqrt{\frac{1+a}{1+2a}}=C<1$ and $c(1)=1$: this prove the claim for $p_1$ and $s=T_1$. The same estimate holds for $-p_0$ and $s=0$, taking into account that $z_1'(0,p_0,p_1)$ points inward the domain $D$.     
\end{proof}
\begin{rem}
	The value of $C$ depends on the physical parameters of the problem: in particular, with reference to the proof of Proposition \ref{proptra}, one has 
	\begin{equation*}
		c(0)=\sqrt{\frac{1+a}{1+2a}}=\sqrt{\frac{\Eh+\mu/2}{\Eh+\mu}}, 
	\end{equation*}
which tends to $1$ when $\Eh\to\infty$. As a consequence, one can control the transversality of $z_1(s; p_0, p_1)$ by acting on the value of the total inner energy. This fact is of particular importance in view of Remark \ref{oss buona definizione F}, since for the first return $F$ to be well defined one needs that the angle $\beta_1=\angle{(p_1, v_1)}$ is such that 
\begin{equation}\label{cond riflessione totale}
	|\sin{\beta_1}|\leq\sqrt{\frac{V_E(p_1)}{V_I(p_1)}}=\sqrt{\frac{\E-\om/2}{\Eh+\mu}}. 
\end{equation} 
If $\Eh$ is such that $\sqrt{1-C^2}<\sqrt{(\E-\om/2)/(\Eh+\mu)}$, and this is true if $\sqrt{\mu/2}<\sqrt{\E-\om/2}$, Eq.(\ref{cond riflessione totale}) is satisfied. \\
Moreover, since for $p_0\to p_1$ the arc $z_1(s; p_0, p_1)$ tends to the ejection-collision solution, the lower bound $C$ can be controlled also by choosing the endpoint to be close enough. 
\end{rem}
Taking together Theorem \ref{thmunp}, Propositions \ref{hip} and \ref{proptra} and Remark \ref{osship}, we can finally state the final existence and uniqueness theorem for the unperturbed boundary.

\bibliography{bibliog}
\bibliographystyle{acm}
\end{document}